\documentclass[oneside, headsepline, 10pt, BCOR=1mm,DIV=18, titlepage=false]{scrartcl}

\usepackage{amsfonts}
\usepackage{amssymb}
\usepackage{amsthm}
\usepackage{amsmath, amsfonts, amsxtra}
\usepackage{esint}
\usepackage{verbatim} %
\usepackage[usenames,dvipsnames]{xcolor}
\usepackage[colorlinks=true, citecolor=BrickRed, linkcolor=blue]{hyperref}%

\newtheorem{theorem}{Theorem}[section]%
\newtheorem{lemma}[theorem]{Lemma}%
\newtheorem{proposition}[theorem]{Proposition}%
\newtheorem{corollary}[theorem]{Corollary}%
\newtheorem{remark}[theorem]{Remark}%









\usepackage{array} 
\newenvironment{ma}{\begin{array}{>{\displaystyle}r >{\displaystyle}c >{\displaystyle}l}}{\end{array}}%

\newcommand{\aleq}{\lesssim}
\newcommand{\ageq}{\gtrsim}
\newcommand{\aeq}{\approx}

\newcommand{\subsubset}{\subset\subset}

\newcommand{\N}{{\mathbb N}}
\newcommand{\R}{{\mathbb R}}

\newcommand{\Z}{{\mathbb Z}}
\renewcommand{\S}{{\mathbb S}}

\newcommand{\Rz}{{\mathcal{R}}} 

\newcommand{\Sw}{{\mathcal{S}}}

\newcommand{\fracm}[1]{{\frac{1}{#1}}}

\newcommand{\BMO}{\operatorname{BMO}}%
\newcommand{\supp}{\operatorname{supp}}%
\newcommand{\dist}{\operatorname{dist}}%

\newcommand{\Rszpot}{I}
\newcommand{\lap}{\Delta}
\newcommand{\D}{\abs{\nabla}}

\newcommand{\laps}[1]{\D^{#1}}
\newcommand{\lapms}[1]{\Rszpot_{#1}}

\newcommand{\lapa}{\D^{\alpha}}

\newcommand{\abs}[1]{{\left \vert #1 \right \vert}}%
\newcommand{\sabs}[1]{{\vert #1 \vert}}%
\newcommand{\brac}[1]{{\left ( #1 \right )}}%
\newcommand{\Vrac}[1]{{\left \Vert #1 \right \Vert }}%
\newcommand{\vrac}[1]{{\Vert #1 \Vert }}%

\newcommand{\ontop}[2]{{\genfrac{}{}{0pt}{}{#1}{#2}}}

\newcommand{\intl}{\int \limits}

\def\XXint#1#2#3{{\setbox0=\hbox{$#1{#2#3}{\int}$}%
     \vcenter{\hbox{$#2#3$}}\kern-.5\wd0}}%

\newcommand{\sref}[2]{#1.\ref{#2}}

\numberwithin{equation}{section}%

\title{$\varepsilon$-regularity for systems involving non-local, antisymmetric operators}
\author{Armin Schikorra\thanks{Max-Planck Institute for Mathematics in the Sciences, Inselstr. 22, 04103 Leipzig, Germany, armin.schikorra@mis.mpg.de}}
\date{\today}

 \usepackage{fancyhdr}

\begin{document}
\maketitle

\thispagestyle{empty}
\begin{abstract}
\noindent
We prove an epsilon-regularity theorem for critical and super-critical systems with a non-local antisymmetric operator on the right-hand side.\\
These systems contain as special cases, Euler-Lagrange equations of conformally invariant variational functionals as Rivi\`{e}re treated them, and also Euler-Lagrange equations of fractional harmonic maps introduced by Da Lio-Rivi\`{e}re.\\
In particular, the arguments presented here give new and uniform proofs of the regularity results by Rivi\`{e}re, Rivi\`{e}re-Struwe, Da-Lio-Rivi\`{e}re, and also the integrability results by Sharp-Topping and Sharp, not discriminating between the classical local, and the non-local situations.
\end{abstract}

\tableofcontents
\section{Introduction}
In recent years there has been quite some research on the effect of antisymmetric potentials in the regularity theory of critical and super-critical 	elliptic partial differential equations. This was initiated by Rivi\`{e}re who in his celebrated \cite{Riv06} proved that solutions $u \in W^{1,2}(D,\R^N)$ to the equation
\begin{equation}\label{eq:rivpde}
 \lap u = \Omega \cdot \nabla u \quad \mbox{in $D \subset \R^2$},
\end{equation}
which is a contracted notation of
\[
 \lap u^i = \sum_{k=1}^N \Omega_{ik} \cdot \nabla u^k \quad \mbox{$1 \leq i \leq N$, in $D \subset \R^2$},
\]
are H\"older continuous, under the condition that $\Omega_{ij} \in L^2(D,\R^2)$ and the at first sight seemingly non-descript condition
\begin{equation}\label{eq:rivantisym}
\Omega_{ik} = - \Omega_{ki},\quad  1 \leq i,k \leq N.
\end{equation}
As Rivi\`{e}re showed, \eqref{eq:rivpde} with \eqref{eq:rivantisym} is essentially the general form of Euler-Lagrange equations of conformally invariant variational functionals which allow the characterization of Gr\"uter \cite{Gruter84}, take for example a manifold $\mathcal{N} \subset \R^N$ and the Dirichlet energy
\[
\int_{\R^2} \abs{\nabla u}^2, \quad u : D \subset \R^2 \to \mathcal{N} \subset \R^N.
\]
We refer the interested reader to the introduction of \cite{Riv06} for more details. In \cite{RivStru08} this was generalized to an epsilon-regularity theorem for $D \subset \R^m$, $m \geq 3$.\\
If the antisymmetry-condition \eqref{eq:rivantisym} is violated, solutions to \eqref{eq:rivpde} might exhibit singularities such as Frehse's \cite{Frehse73} counter-example $\log \log \frac{1}{\abs{x}}$. In fact, the antisymmetry is shown to be closely related to the appearance of Hardy spaces, and also to H\'elein's \cite{Hel91} moving frame technique, cf. \cite{IchEnergie}.\\
Motivated by this, Da Lio and Rivi\`{e}re \cite{DR1dMan} (for $m=1$) showed that this regularizing effect of antisymmetry exists and appears also in the setting of $m/2$-harmonic maps, critical points of the energy
\[
 \int_{\R^m} \abs{\laps{\frac{m}{2}} u}^2, \quad u : \R^m \to \mathcal{N} \subset \R^N.
\]
which satisfy (roughly) an Euler-Lagrange equation of the form
\begin{equation}\label{eq:DLantisym}
 \lap^{\frac{m}{2}} u^i = \sum_{k=1}^N \Omega_{ik}\ \laps{\frac{m}{2}} u^k \quad \mbox{$1 \leq i \leq N$, in $D \subset \R^m$}.
\end{equation}
Here, $\Omega_{ij} \in L^2(\R^m)$ satisfies again \eqref{eq:rivantisym}, and $\laps{\alpha} = (-\lap)^{\frac{\alpha}{2}}$ is the elliptic differential operator of differential order $\alpha$ with the symbol $\abs{\xi}^\alpha$, for the precise definition we refer to Section~\ref{s:fracfacts} . 

As well in the classical situation \cite{Riv06}, as also in the case of fractional harmonic maps, the argument relies on transforming the equation with an orthogonal matrix $P$ (in a similar way as H\`{e}lein's moving frame technique, cf.~\cite{IchEnergie}). That is, one computes the respective equation $P \nabla u$ instead of $\nabla u$, or $P  \lap^{\frac{n}{4}} u$ instead of $ \lap^{\frac{n}{4}} u$ and obtains a transformed $\Omega_P$, which for the right choice of $P$ exhibits better properties than the original $\Omega$: In the classical case, $\operatorname{div}(\Omega_P) = 0$, while in the fractional case, $\Omega_P \in L^{2,1}$ (where $L^{2,1} \subsetneq L^2$ is the Lorentz space dual to the weak $L^2$, denoted  by $L^{2,\infty}$). Note that while a condition like $\operatorname{div} (f) = 0$ is destroyed under a distortion like $\tilde{f} := fg$, even for $g \in L^\infty$, the condition $f \in L^{2,1}$ is also valid for $\tilde{f} = fg$, if $g \in L^\infty$. 

Thus, the techniques developed in the fractional setting \cite{DR1dMan, DR1dSphere, Sfracenergy, DndMan, SNHarmS10}, seem somewhat more dynamic and stable under certain distortions. For example, in \cite{DLSsphere, DLSman} Da Lio and the author were able to extend some of the results to the degenerate situation of the energy
\[
 \int_{\R^m} \abs{\laps{\alpha} u}^{\frac{m}{\alpha}} , \quad u : \R^m \to \mathcal{N} \subset \R^N,
\]
the Euler-Lagrange equation of which have the form
\[
 \laps{\alpha} (\abs{\laps{\alpha} u}^{\frac{m}{\alpha}-2}\laps{\alpha} u) = \abs{\laps{\alpha} u}^{\frac{m}{\alpha}-2} \sum_{k=1}^N \Omega_{ik}\ \laps{\alpha} u^k \quad \mbox{$1 \leq i \leq N$, in $D \subset \R^m$}.
\]
The aim of the present work is to shed more light on the connection between the two systems \eqref{eq:DLantisym} and \eqref{eq:rivpde} in the critical and supercritical case, and we are going to extend the techniques developed in \cite{DR1dMan, DR1dSphere, Sfracenergy, SNHarmS10} to give a uniform argument for $\varepsilon$-regularity for quite general systems which in particular include as special cases both \eqref{eq:DLantisym} and \eqref{eq:rivpde}.\\[0.5em]
Setting $w := (-\lap)^{\frac{1}{2}} u \equiv \laps{1} u \in L^2(\R^n)$, \eqref{eq:rivpde} reads as
\begin{equation}\label{eq:riv2dalio}
 \lap^{\frac{1}{2}} w^i = \sum_{\gamma = 1}^2 \sum_{k = 1}^N \Omega^\gamma_{ik} \Rz_{\gamma} [w^k],
\end{equation}
where $\Rz_{\gamma} \equiv \partial_\gamma \lap^{-\frac{1}{2}}$ denotes the Riesz transform. Thus, \eqref{eq:rivpde} is of the form \eqref{eq:DLantisym}, but $\Omega$ is not a pointwise matrix anymore, but a non-local, linear operator mapping $L^2(\R^m)$ into $L^1(\R^m)$. This was our main motivation, to study the regularity, and, in the super-critical regime, $\varepsilon$-regularity of solutions $w \in L^2(\R^m)$ of 
\begin{equation}\label{eq:wpde}
 \int w_i\ \laps{\mu} \varphi = -\int \Omega_{ik} [ w_k] \varphi \quad \mbox{for all $\varphi \in C_0^\infty(D)$},
\end{equation}
where $\Omega_{ik}$ is a linear mapping which maps $L^2(\R^m)$ into $L^1(\R^m)$. We will restrict ourselves to $\Omega$ of the form
\begin{equation}\label{eq:OmegaisRiesztrafo}
 \Omega_{ij}[] = \sum_{l=0}^m A^l_{ij}\Rz_l[],
\end{equation}
where $A^l_{ij} = - A^l_{ij}\in L^2(\R^m)$, $i,j \in 1,\ldots,m$, $\Rz_l[]$ is the $l$-th Riesz transform for $l = 1,\ldots,m$ and $\Rz_0[]$ is the identity on $\R^m$. The arguments presented here hold also for more general potentials $\Omega: L^2 \to L^1$, under suitable conditions on quasi-locality and its commutators. But as \eqref{eq:OmegaisRiesztrafo} contains already the most interesting examples (see below), we shall restrict our attention to this setting for the sake of overview.\\
Our main result is then the following $\varepsilon$-regularity:
\begin{theorem}\label{th:main}
Let $\mu \leq \min \{1,\frac{m}{2}\}$ or $\mu = \frac{m}{2}$. Let $D \subsubset \R^m$, $p \in (1,\infty)$, then there exists $\theta > 0$ such that the following holds: Let $w \in L^2(\R^m) \cap L^{(2)_{2\mu}}(D)$, that is,
\begin{equation} \label{eq:wassumption}
 \vrac{w}_{2,\R^m} + \sup_{B_\rho \subset D} \rho^{\frac{2\mu-m}{2}} \vrac{w}_{2,B_\rho} < \infty,
\end{equation}
be a solution to \eqref{eq:wpde}, where $\Omega$ is of the form \eqref{eq:OmegaisRiesztrafo}. If $\Omega$ satisfies moreover 
\begin{equation} \label{eq:regp}
 \sup_{B_\rho(x), x \in D} \rho^{\frac{2\mu-m}{2}} \vrac{A^l}_{L^2} \leq \theta,
\end{equation}
then $w \in L^{p}_{loc}(D)$. 
\end{theorem}
Let us remark the following corollaries from  Theorem \ref{th:main}. 

As mentioned above, by the representation \eqref{eq:riv2dalio} this gives a new proof of Rivi\`ere's theorem \cite{Riv06}, and also the $\varepsilon$-regularity theorem of \cite{RivStru08}. 

Moreover, from Theorem \ref{th:main} a new proof of Sharp and Topping's integrability theorem \cite{SharpTopping11} for \eqref{eq:rivpde} follows, and also an extension to the super-critical setting. The latter has been done independently, and by different methods by Sharp \cite{Sharp12}.\\

Also, we extend these integrability results to the non-local case for $\mu \leq 1$. For $\mu > 1$ it seems already in the classical setting of the biharmonic maps, cf. \cite{StruweBiharm}, that for $\varepsilon$-regularity we need more information on the growth of $\Omega$ in terms of the solution, a fact which appeared also in our setting and forced us to restrict $\mu = \frac{m}{2}$ if $\mu > 1$.\\

Another corollary worth mentioning is that the arguments presented here also enable us to treat $\epsilon$-regularity critical points of more general non-local energies, e.g.,
\begin{equation}\label{eq:energynablaalphau}
 E(u) = \int \abs{\nabla^\alpha u}^2 \quad \mbox{$u : \R^m \to \mathcal{N} \subset \R^N$},
\end{equation}
where for $\Rz = [\Rz_1,\ldots,\Rz_m]^T$, and $\Rz_i$ being the $i$-th Riesz transform,
\[
 \nabla^\alpha u := \Rz[\laps{\alpha} u].
\]

Another remark regards the smallness condition of \eqref{eq:regp}. In the critical setting $2\mu = m$, it is easy to verify, that this condition holds, if $D$ is chosen appropriately small. In the super-critical regime $2\mu < m$, this condition would follow from some kind of monotonicity formula for stationary points of energies of the form \eqref{eq:energynablaalphau}, which for the non-classical settings are unknown so far.\\

Let us now sketch the arguments we are going to need. Firstly, somewhat motivated by the arguments in \cite{RivStru08}, we are going estimate the growth of the norm possibly far below the natural exponent $2$. More precisely we estimate the growth in $R$ of
\begin{equation}\label{eq:intro:morreykappanorm}
 \sup_{B_r \subset B_R} r^{\frac{\lambda_\kappa-m}{p_\kappa}}\ \vrac{w}_{p_\kappa,B_r},
\end{equation}
starting with $\kappa = \mu$, where
\[
 \lambda_\kappa := \frac{m(2\mu - \kappa)}{m-\kappa},
\]
\[
 p_\kappa := \frac{m}{m-\kappa}.
\]
The main work is to show that for any $\kappa \in [\mu,2\mu)$ there is a good growth of these quantities, then starting for $\kappa_0 = \mu$, we can find a sequence of $\kappa_i$ which converges to $2\mu$, such that each growth of the $\kappa_i$-norm (that is \eqref{eq:intro:morreykappanorm} with $\kappa_i$) is controlled by the $\kappa_{i-1}$-norm. Finally, for $\kappa$ sufficiently close to $2\mu$, we show that we can actually have an estimate for $p > 2$. From this we have
\begin{theorem}\label{th:upto2}
There is $\theta_2 > 0$ such that if $\theta < \theta_2$, there exists $p > 2$, $\lambda < 2\mu$, such that
\[
 w \in L_{loc}^{(p)_{\lambda}}(D).
\]
\end{theorem}

For Theorem \ref{th:upto2}, the antisymmetry of $\Omega$ will be crucial. Once Theorem \ref{th:upto2} is established, the system \eqref{eq:wpde} becomes subcritical, and we can drop the antisymmetry condition and just by the growth of the PDE, we have
\begin{theorem}\label{th:uptoinfty}
Assume $w$ as in Theorem \ref{th:main}, where we do not require the antisymmetry of $\Omega$. Assume moreover, that $w \in L^{p_1}_{loc}(D)$ for $p_1 > 2$. Then for any $p > 2$, there is $\theta_p > 0$ such that if $\theta < \theta_p$ in \eqref{eq:regp}, also
\[
 w \in L^p_{loc}(D).
\]
\end{theorem}

The main difficulty is thus Theorem \ref{th:upto2} and the estimates of the Morrey norm. For the proof of this theorem we need the following two main technical ingredients: Firstly, we need to extend the known commutator results \cite{DR1dSphere, DR1dMan}, and also the pointwise estimates \cite{Sfracenergy, SNHarmS10}. 
We introduce the following commutators: 
Let $X$ be a linear space, For $\varphi \in C_0^\infty(\R^m)$, $T: L^p(\R^m) \to L^q(\R^m)$, $1 \leq p,q\leq \infty$. We then set for $f \in L^p(\R^m)$ the commutator $\mathcal{C}(\varphi,T)[f]$
\begin{equation}\label{eq:def:commutator}
 \mathcal{C}(\varphi,T)[f] := \varphi T[f] - T[\varphi f].
\end{equation}
This commutator was estimated in terms of Hardy spaces for $T = \Rz$ the Riesz transform or $T = \lapms{s}$ the Riesz potential in \cite{CRW76, Chanillo82}, nevertheless we need more precise estimates and generalizations.
The next bilinear commutator was introduced in \cite{DR1dSphere}, in \cite{Sfracenergy} pointwise estimates were given.
\begin{equation}\label{eq:def:Hsab}
 H_s(a,b) := \laps{s}(ab) - a\laps{s} b - b \laps{s}a.
\end{equation}

For these commutators we show the following

\begin{theorem}\label{th:commutators}
For any $\mu \in (0,1]$, we have the following Hardy-space $\mathcal{H}$ estimate (for $\Rz[]$ any zero-multiplier operator, we need it for the Riesz-transform, only)
\[
 \Vrac{\laps{\mu} \brac{\Rz[h]\ \lapms{\mu} b - \Rz[h\ \lapms{\mu} b]}}_{\mathcal{H}} \leq \Vert h \Vert_{2}\ \Vert b \Vert_{2}
\]
Moreover, we have
\[
 \vrac{ \mathcal{C}(f,\Rz)[\laps{\mu} \varphi] }_2 \aleq \vrac{\laps{\mu} f}_2\ [\varphi]_{BMO},
\]
and its pointwise counter-part: For any $\delta_i \in (0,1)$ and any $\gamma_i \in (0,\delta_i)$, $i = 1,2$,
 \[
\begin{ma}
\abs{\mathcal{C}(a,\Rz)[b]} &\leq& C_{\Rz,\delta_1,\gamma_1}\  \lapms{\delta_1-\gamma_1} \abs{\laps{\delta_1} a}\ \lapms{\gamma_1} \abs{b}
+ C_{\Rz,\delta_2,\gamma_2}\ \lapms{\gamma_2}\brac{\lapms{\delta_2-\gamma_2}\abs{b}\  \abs{\laps{\delta_2} a}}.
\end{ma}
\]
Finally we have
\[
 \vrac{H_\mu(\varphi,g)}_2 \aleq \vrac{\laps{\mu} g}_2\ [\varphi]_{BMO}.
\]
and 
\begin{equation}\label{eq:def:laphH1inhardy}
\vrac{\laps{\mu} H_{\mu}(a,b)}_{\mathcal{H}} \aleq \vrac{\laps{\mu} a}_2\ \vrac{\laps{\mu} b}_2 \quad \mbox{for $\mu \in (0,1]$},
\end{equation}
as well as its pointwise counterpart: for any $\mu \in [0,m]$ there is $L \in \N$ such that for any $\beta \in [0,\min (\mu,1))$, $\mu \in [0,m)$, $\tau \in (\max\{\beta,\mu+\beta-1\},\mu]$ there are, $s_k \in [0,\mu)$, $t_k \in [0,\tau)$, where $\tau-\beta -s_k-t_k \geq 0$, such that the following holds
\[
 \abs{\laps{\beta} H_\mu(a,b)} \aleq \sum_{k=1}^L \lapms{\tau-\beta  -s_k-t_k} \brac{\lapms{s_k} \abs{\laps{\mu} a}\ \lapms{t_k}\abs{\laps{\tau} b}}.
\]
\end{theorem}
\begin{remark}
For $\mu < 1$ the Hardy-space estimates above follow essentially from an obvious adaption of Da Lio and Rivi\`{e}re's argument \cite{DR1dSphere}, and \eqref{eq:def:laphH1inhardy} has been proven by them. For $\mu > 1$, already from the pointwise arguments in \cite{Sfracenergy} there is no hope for similar results. The interesting and new case $\mu = 1$, for which even \eqref{eq:def:laphH1inhardy} was unclear up to now, needs a more careful adaption of the arguments in \cite{DR1dSphere}.
\end{remark}

Equipped with a good understanding of these commutator we will show

\begin{theorem}\label{th:intro:energy}
There is a uniform $\Lambda > 0$ such that the following holds: Let $\Omega$ be as in \eqref{eq:OmegaisRiesztrafo}, and assume that $\Omega_{ij}[] = -\Omega_{ji}[]$. For any $B_r \subset \R^m$, we can then choose $P: \R^m \to SO(N)$, $\supp (P-I) \subset B_{r}$. Then for any $\varphi \in C_0^\infty(B_{r})$, 
\[
- \int \Omega^P[\laps{\mu} \varphi] \leq C\ r^{\frac{m}{2}-\mu}\ \vrac{A}_{2}\ [\varphi]_{BMO} + \vrac{A}_{2}^{2}\ \begin{cases}
                                                           [\varphi]_{BMO} \quad &\mbox{if $\mu \in (0,1]$},\\ 
                                                           \vrac{\laps{\mu} \varphi}_{(2,\infty)} \quad &\mbox{if $\mu > 1$},
                                                           \end{cases}
\]
where
\[
 \Omega^P_{ij} [f]:= \brac{\laps{\mu} P_{ik}}\ P^T_{kj}\ f + P_{ik} \Omega_{kl} [P^T_{lj} f],
\]
\end{theorem}
In \cite{IchEnergie} the construction of $P$ is done via minimization of $E(P) = \vrac{P \nabla P^T + P \Omega P}_{L^2}^2$ under the condition that $P$ maps into $SO(N)$, a.e.. This is the argument that H\'elein \cite{Hel91} essentially used for his moving-frame technique, and it provides an alternative to Rivi\'ere's adaption of Uhlenbecks \cite{Uhlenbeck82} gauge-theoretic construction of $P$ in \cite{Riv06}. Both techniques can be extended to the fractional case, where $\Omega$ is still a pointwise multiplication \cite{DR1dMan, Sfracenergy}. We adapt the arguments \cite{Sfracenergy, IchEnergie} to this case of a non-local operator $\Omega[]$, by minimizing in Section~\ref{th:intro:energy} the energy
\[
E(P) := \sup_{\psi \in L^2} \intl_{\R^m} \Omega^P [\psi],
\]
and showing that several terms of the Euler-Lagrange equations fall under the realm of Theorem \ref{th:commutators}.

\paragraph*{Notation}
Let $L^{p,q}$ be the Lorentz spaces, cf., e.g. \cite{Hunt66, Tartar07, GrafC08}, whose norm we denote with $\vrac{\cdot}_{(p,q)}$. We set
\begin{equation}\label{eq:def:Morreynorm}
 \Vert f \Vert_{(p,q)_\lambda} \equiv \Vert f \Vert_{\mathcal{M}((p,q),\lambda)} := \sup_{B_r \subset \R^m} r^{\frac{\lambda - m}{p}} \Vert f \Vert_{(p,q),B_r},
\end{equation}
and for $A \subset \R^m$,
\begin{equation}\label{eq:def:fpqlambdacentered}
 [f]_{(p,q)_\lambda, A} := \abs{A}^{\frac{\lambda - m}{mp}}\ \vrac{f}_{(p,q),A},
\end{equation}
\begin{equation}\label{eq:def:fpqlambdanotcentered}
 \vrac{f}_{(p,q)_\lambda, A} := \sup_{B_\rho \subset A} [f]_{(p,q),B_\rho}.
\end{equation}
We say that $f$ belongs to the Morrey space $L^{(p,q)_\lambda}(A)$, if the respective norm $\vrac{f}_{(p,q)_\lambda, A}$ is finite.

We will also use frequently the following annuli
\begin{equation}\label{eq:def:Alambdark}
A_{\Lambda, r}^k := B_{2^k \Lambda r} \backslash B_{2^{k-1} \Lambda r}, \quad A_{r}^k \equiv A_{1,r}^k.
\end{equation}

In Section~\ref{s:fracfacts} we recall several facts on the fractional laplacian, which we are going to use throughout this work.

\paragraph*{Acknowledgements.}
The author has received funding from the European Research Council under the European Union's Seventh Framework Programme (FP7/2007-2013) / ERC grant agreement no 267087, DAAD PostDoc Program (D/10/50763) and the Forschungsinstitut f\"ur Mathematik, ETH Z\"urich. He would like to thank Tristan Rivi\`ere and the ETH for their hospitality.
\newpage
\section{\texorpdfstring{$L^{2+\varepsilon}$}{L2+eps}-integrability: Proof of Theorem \ref{th:upto2}}\label{s:upto2}
It is helpful, to check once and for all, 
\begin{equation} \label{eq:mm2mulambda}
 m-2\mu = \frac{m-{\lambda_\kappa}}{p_\kappa}, \quad \kappa \in [\mu,2\mu),
\end{equation}
where
\begin{equation}\label{eq:def:lambdakappa}
 \lambda_\kappa := \frac{m(2\mu - \kappa)}{m-\kappa},
\end{equation}
\begin{equation}\label{eq:def:pkappa}
 p_\kappa := \frac{m}{m-\kappa}.
\end{equation}

Assume $w: \R^m \to \R^N$, $\mu\leq \frac{m}{2}$, $w \in L^2(\R^m)$, $\laps{\mu} w \in L^2(\R^m)$ is for $D \subsubset \R^m$ a solution to \eqref{eq:wpde}. We are going to establish that for any $\kappa \in [\mu,2\mu)$, if $\theta \equiv \theta_\kappa$ in \eqref{eq:regp} is suitably small, for any $\tilde{D} \subsubset D$, we have
\begin{equation}\label{eq:claiml2}
 \sup_{r > 0, x_0 \in \tilde{D}} r^{\frac{\lambda_\kappa-m}{p_\kappa}}\ \vrac{w}_{p_\kappa,B_r(x_0)} \leq C_{\tilde{D},w,\kappa}.
\end{equation}
Note that possibly $p_\kappa < 2$ for all $\kappa \in [\mu,2\mu)$. In order to show \eqref{eq:claiml2}, we first note that its satisfied by assumption \eqref{eq:wassumption} for $\kappa = \mu$. In fact, if $x_0 \in \tilde{D} \subsubset D$, then for any $r >0$, or $B_r(x_0) \subset D$ or $r > c \dist(\tilde{D},\partial D)$. Now, we show that for arbitrary $\kappa \in [\mu,2\mu)$, there is $\kappa_1 > \kappa$, so that \eqref{eq:claiml2} holds. Moreover, we will show a lower bound on $\kappa_1 - \kappa$, in order to ensure that we come arbitrarily close to $2\mu$ if we repeat this construction  finitely many times. 

Then we can show that if we choose $\kappa \in [\mu,2\mu)$ close enough to $2\mu$, \eqref{eq:claiml2} suffices to conclude the better integrability of Theorem \ref{th:upto2}.

\subsection*{Establishing \eqref{eq:claiml2}}
For mappings $P: \R^m \to SO(N)$, $P \equiv I$ on $\R^m \backslash D$ (denoting with $I = (\delta_{ij})_{ij} \in R^{N\times N}$ the identity matrix) from \eqref{eq:wpde} we have
\[
\begin{ma}
 \int P_{ik}w_{k}\ \laps{\mu} \varphi &=& \int w_{k}\ \laps{\mu} (P_{ik} \varphi) -\int w_{k}\ \brac{\laps{\mu} P_{ik}}\ \varphi - \int w_{k}\ H_{\mu}(P_{ik},\varphi)\\
&=& -\int \Omega_{kl} [w_{l}]\ P_{ik} \varphi -\int w_{k}\ \brac{\laps{\mu} P_{ik}}\ \varphi - \int w_{k}\ H_{\mu}((P-I)_{ik},\varphi)\\
\end{ma}
\]
Setting $v_i := P_{ik} w_k$, this is
\begin{equation}\label{eq:pdev}
 \int v_i\ \laps{\mu} \varphi = -\int \brac{P_{ik} \Omega_{kl} [P_{jl} v_j] + \brac{\laps{\mu} P_{ik}} P_{jk} v_{j} } \ \varphi - \int w_{k}\ H_{\mu}((P-I)_{ik},\varphi)\\
\end{equation}
\paragraph*{The Growth Estimates.}
From \eqref{eq:pdev}, Lemma~\ref{la:rhs:Htermest}, and Lemma~\ref{la:rhs:OmegaPest} we infer
\begin{theorem}[Right-hand side estimates]\label{th:rhsest}
If $\mu \in (0,\min \{1,\frac{m}{2}\}]$ or $2\mu = m$, there is a uniform $\Lambda \equiv \Lambda_\mu > 0$, depending only on $\mu$, such that the following holds: Let $B_r \subset \R^m$, and assume \eqref{eq:pdev} holds for all $\varphi \in C_0^\infty(B_{r})$. Then exists a choice of $P$ such that \eqref{eq:pdev} implies for any $\varphi \in C_0^\infty(B_{\Lambda^{-1}r})$, and for any $\tau \in (0,\mu]$ sufficiently close to, or greater than $2\mu-\kappa$, 
\[
\begin{ma}
 (\Lambda^{-1}r)^{2\mu-m} \int v\ \laps{\mu} \varphi 
&\leq& C_{\kappa}\ \theta\ \vrac{ \laps{\tau} \varphi }_{\brac{\frac{m}{\tau+\kappa -\mu},1}}\ \vrac{w}_{(p_\kappa,\infty)_{{\lambda_\kappa}},B_{r}}\\
&& + C_{\kappa}\ \theta\ \vrac{ \laps{\tau} \varphi }_{\brac{\frac{m}{\tau+\kappa -\mu},2}}\ \Lambda^{\kappa-3\mu}\ \sum_{k=1}^\infty 2^{k(\kappa-3\mu)}\ [w]_{(p_\kappa,\infty)_{{\lambda_\kappa}},A^k_{r}}.
\end{ma}
\]
where we recall that the right-hand side norms were defined in \eqref{eq:def:fpqlambdacentered}, \eqref{eq:def:fpqlambdanotcentered}, $A_{r}^k$ is as in \eqref{eq:def:Alambdark}, and $\lambda_\kappa$ as in \eqref{eq:def:lambdakappa}, $p_\kappa$ as in \eqref{eq:def:pkappa}. 
\end{theorem}
From Theorem \ref{th:rhsest} and Lemma \ref{la:lhsest} (applied to $\Lambda^{-1} r$ instead of $r$) we infer for any $\tau \in (0,\mu]$ sufficiently close to $\mu$ and any $\Lambda \gg \Lambda_\mu$ sufficiently large (for the right-hand side norms recall \eqref{eq:def:fpqlambdacentered} and \eqref{eq:def:fpqlambdanotcentered}), also in view of Proposition \ref{pr:summationdoesntmatter},
\[
\begin{ma}
 &&(\Lambda^{-2} r)^{2\mu-m}\ \vrac{\laps{\mu-\tau}v}_{(\frac{m}{m+\mu-\tau-\kappa},\infty),B_{\Lambda^{-2} r}} \\
&\leq& \Lambda^{-1}\ C_{\kappa}\ \theta \ \vrac{w}_{(p_\kappa,\infty)_{{\lambda_\kappa}},B_{r}}\\
&& + \Lambda^{-1}\ C_{\kappa}\ \theta\ \sum_{k=1}^\infty 2^{k(\kappa-3\mu)}\ [w]_{(p_\kappa,\infty)_{{\lambda_\kappa}},A^k_{r}}.\\
 &&+ C\ (\Lambda^{-2} r)^{2\mu-m}\ \Lambda^{\kappa-m+\tau-\mu}\vrac{w}_{(p_\kappa,\infty),B_{\Lambda^{-1} r}} \\
&&+ C\ (\Lambda^{-2} r)^{2\mu-m}\ \Lambda^{\kappa-m+\tau-\mu} \sum_{k=0}^\infty 2^{k(\kappa-m+\tau-\mu)}\ \vrac{w}_{(p_\kappa,\infty),A^k_{{\Lambda^{-1} r}}}\\
&\overset{\eqref{eq:mm2mulambda}}{\leq}& C_{\kappa}\ \theta\ \Lambda^{m-2\mu} \ \vrac{w}_{(p_\kappa,\infty)_{{\lambda_\kappa}},B_{r}}\\ 
&& + C_{\kappa}\ \theta\ \Lambda^{m-2\mu}\ \sum_{k=1}^\infty 2^{k(\kappa-3\mu)}\ [w]_{(p_\kappa,\infty)_{{\lambda_\kappa}},A^k_{r}} \\ 
&&+ C\ \Lambda^{\kappa+\tau-3\mu}\  \vrac{w}_{(p_\kappa,\infty)_{\lambda_\kappa},B_{\Lambda^{-1} r}} \\
&&+ C\ \Lambda^{\kappa+\tau-3\mu} \	 \sum_{k=0}^\infty 2^{k(\kappa+\tau-3\mu)}\  [w]_{(p_\kappa,\infty)_{\lambda_\kappa},A^k_{{\Lambda^{-1} r}}} \\ 
&\overset{\sref{P}{pr:summationdoesntmatter}}{\aleq}& (C_{\kappa}\ \theta\ \Lambda^{m-2\mu} + C\Lambda^{\kappa+\tau-3\mu})\ \vrac{w}_{(p_\kappa,\infty)_{{\lambda_\kappa}},B_{r}}\\ 
&& + (C_{\kappa}\ \theta\ \Lambda^{m-2\mu} + C\ \Lambda^{\kappa+\tau-3\mu})\ \sum_{k=1}^\infty 2^{k(\kappa+\tau-3\mu)}\ [w]_{(p_\kappa,\infty)_{{\lambda_\kappa}},A^k_{r}}.
\end{ma}
\]
For later reference, we write this as
\begin{equation}\label{eq:setup:2nditerationguy}
\begin{ma}
(\Lambda^{-2} r)^{2\mu-m}\ \vrac{\laps{\mu-\tau}v}_{(\frac{m}{m+\mu-\tau-\kappa},\infty),B_{\Lambda^{-2} r}} 
&\leq&  (C_{\kappa,\mu}\ \theta\ \Lambda^{m-2\mu} + C_\mu\ \Lambda^{\kappa+\tau-3\mu})\ \vrac{w}_{(p_\kappa,\infty)_{{\lambda_\kappa}},B_{r}}\\ 
&& +  (C_{\kappa,\mu}\ \theta\ \Lambda^{m-2\mu} + C_\mu\ \Lambda^{\kappa+\tau-3\mu})\ \sum_{k=1}^\infty 2^{k(\kappa+\tau-3\mu)}\ [w]_{(p_\kappa,\infty)_{{\lambda_\kappa}},A^k_{r}}.
\end{ma}
\end{equation}
For $\tau = \mu$, 
\begin{equation}\label{eq:setup:iterationguy}
\begin{ma}
(\Lambda^{-2} r)^{2\mu-m}\ \vrac{v}_{(p_\kappa,\infty),B_{\Lambda^{-2} r}} 
&\leq& (C_{\kappa,\mu}\ \theta\ \Lambda^{m-2\mu} + C_{\mu}\ \Lambda^{\kappa-2\mu})\ \vrac{w}_{(p_\kappa,\infty)_{{\lambda_\kappa}},B_{r}}\\ 
&& + (C_{\kappa,\mu}\ \theta\ \Lambda^{m-2\mu} + C_{\mu}\ \Lambda^{\kappa-2\mu})\ \sum_{k=1}^\infty 2^{k(\kappa-2\mu)}\ [w]_{(p_\kappa,\infty)_{{\lambda_\kappa}},A^k_{r}}.
\end{ma}
\end{equation}
\paragraph*{The Iteration Procedure.}
Note that $\abs{w} = \abs{v}$, so we can use them equivalently. Equation \eqref{eq:setup:iterationguy} holds for any $B_r(x_0)$, where $x_0 \in D$ and $r < \tilde{d}(x_0) := C\dist(x_0,\partial D)$ (the constant essentially only depending on the construction of $P$ and the set where $\Omega$ is small). For $x_0 \in D$ and $R > 0$ set
\[
 \Phi_{x_0} (R) := \sup_{B_\rho \subset B_R(x_0)} \rho^{2\mu-m} \vrac{w}_{(p_\kappa,\infty),B_{\rho}},
\]
and its centered counter-part
\[
 \Psi_{x_0} (R) := \sup_{\rho \in (0,R)} \rho^{2\mu-m} \vrac{w}_{(p_\kappa,\infty),B_{\rho}(x_0)} \leq \Phi_{x_0}(R)
\]
then from \eqref{eq:setup:iterationguy} for any $R$, $x_0 \in D$ with $R < d(x_0)$, we have 
\[
 \Phi_{x_0}(\Lambda^{-1} R) \leq \brac{C_{\kappa}\ \theta \Lambda^{m-2\mu}+ C\ \Lambda^{\kappa-2\mu}} \ \Phi_{x_0}(R) +  \brac{C_{\kappa}\ \theta \Lambda^{m-2\mu}+ C\ \Lambda^{\kappa-2\mu}}\ \sum_{k=1}^\infty 2^{k(\kappa-2\mu)}\ \Psi_{x_0}(2^k R).\\
 \]
Note that from \eqref{eq:claiml2}, we know that $\Phi_{x_0}(R) < C_{D,x_0,w}$ for any $x_0 \in D$. Now we can iterate, Lemma \ref{la:iteration}, satisfying the assumption \eqref{eq:it:smallness} by choosing $\Lambda \equiv \Lambda_\kappa := 2^{\frac{C_\mu}{\brac{2\mu-\kappa}^4}}$ with sufficiently large $C_\mu$ and assuming that $\theta<(\Lambda_\kappa)^{\kappa-m}$. Then, for any $r < R$,
\[
 \sup_{B_\rho \subset B_r(x_0)} \rho^{2\mu-m}\ \vrac{w}_{(p_\kappa,\infty),B_{\rho}} = \sup_{B_\rho \subset B_r(x_0)} \rho^{2\mu-m}\ \vrac{v}_{(p_\kappa,\infty),B_{\rho}}  \aleq C_{\kappa, w,\Lambda,R}\  r^{\sigma_\kappa}, \quad \mbox{where $\sigma_\kappa = \frac{(2\mu-\kappa)^4}{C_\mu}$}.
\]
We can assume, that $\sigma_\kappa < 2\mu-\kappa$. Since
\[
 \sup_{\rho > R} r^{2\mu-\sigma_\kappa-m}\ \vrac{w}_{(p_\kappa,\infty),B_{r}(x_0)} \aleq R^{-\sigma_\kappa}\ \vrac{w}_{(p_\kappa,\infty)_{{\lambda_\kappa}},B_{4r}},
\]
we arrive at
\[
 r^{2\mu-\sigma_\kappa-m}\ \vrac{w}_{(p_\kappa,\infty),B_{r}(x_0)} \leq C_{\kappa,w,x_0}
\]
so we get for any $B_r(x_0) \subset B_R(x_0)$
\[
 r^{-\sigma_\kappa} \vrac{\chi_{B_r} w}_{(p_\kappa,\infty)_{{\lambda_\kappa}}} + \sup_{\rho > 0} \rho^{2\mu-\sigma_\kappa - m} \vrac{w}_{(p_\kappa,\infty), B_\rho(x_0)} \leq C_{\kappa,w}.
\]
Plugging this into \eqref{eq:setup:2nditerationguy}, we have for all $\tau \in (0,\mu]$ sufficiently close to, or greater than $2\mu-\kappa$,
\begin{equation}\label{eq:Dmmtauvest}
r^{2\mu-m}\ \vrac{\laps{\mu-\tau}v}_{(\frac{m}{m+\mu-\tau-\kappa},\infty),B_{\Lambda^{-1}r}(x_0)} \\
\aleq C_{\kappa,w} r^{\sigma_\kappa} + C_{\kappa,w}\ r^{\sigma_\kappa}\ \sum_{k=1}^\infty 2^{k(\kappa-2\mu+\sigma_\kappa)},
\end{equation}
so that we have for all small $r$,
\[
r^{2\mu-m-\sigma_\kappa}\ \vrac{\laps{\mu-\tau}v}_{(\frac{m}{m+\mu-\tau-\kappa},\infty),B_{r}(x_0)} \\
\aleq C_{w,\kappa,R}.
\]
Moving the $B_R(x_0)$, for any $D_1 \subsubset D$, we have that
\begin{equation}\label{eq:setup:nablamumtaulambdabounded}
 \vrac{\laps{\mu-\tau}v}_{(\frac{m}{m+\mu-\tau-\kappa},\infty)_{\lambda},D_1} \aleq C_{w,\kappa,\D_1,D}
\end{equation}
for $\lambda$ such that (choosing $\tau$ possibly even closer to $\mu$, ensuring that $\abs{\mu-\tau} \leq \frac{\sigma_{\kappa}}{2}$)
\[
 \frac{\lambda}{m}  = \frac{3\mu-\tau-\kappa-\sigma_\kappa}{m+\mu-\tau-\kappa} \leq  \frac{3\mu-\tau-\kappa-\sigma_\kappa}{m-\kappa}  \overset{\eqref{eq:def:lambdakappa}}{=} \frac{\lambda_\kappa}{m}+\frac{\mu-\tau-\sigma_\kappa}{m-\kappa}
\]
\paragraph*{Choosing the next $\kappa$.}
Assume for a moment that $2\mu < m$. we can guarantee
\begin{equation}\label{eq:lambdaNleqlambda0}
 0 < \lambda <  \lambda_\kappa - c_{m}\sigma_\kappa,
\end{equation}
and we choose $\kappa_{1,1} \in(\kappa,2\mu)$ via
\[
 \lambda =: m\frac{2\mu - \kappa_{1,1}}{m-\kappa_{1,1}}.
\]
By \eqref{eq:lambdaNleqlambda0}, 
\[
 m\frac{2\mu - \kappa_{1,1}}{m-\kappa_{1,1}} < m\frac{2\mu - \kappa}{m-\kappa} - c_{m-2\mu}\ \sigma_\kappa
\]
and thus we have
\[
 \kappa_{1,1} > \kappa + \sigma c_{m-2\mu}\ \frac{(m-\kappa_{1,1})(m-\kappa)}{m}.
\]
On the other hand, by a localized version of Adams' \cite{Adams75}-argument on Riesz potentials, we infer from \eqref{eq:setup:nablamumtaulambdabounded} that for any $D_2 \subsubset D_1$,
\[
 \vrac{v}_{(p,\infty)_{\lambda},D_2} = \vrac{w}_{(p,\infty)_{\lambda},D_2} < \infty,
\]
where
\[
 \fracm{p} = \frac{m+\mu-\tau-\kappa}{m} - \frac{\mu -\tau}{\lambda} \in (0,1).
\]
Letting
\[
 \frac{m}{m-\kappa_{1,2}} := p,
\]
we can estimate 
\[
 \frac{\kappa_{1,2}-\kappa}{m}=\brac{\mu-\tau} \brac{ \fracm{\lambda} - \fracm{m} } \geq \sigma_{\kappa} c_\mu.
\]
Thus for a certain $\alpha > 0$,
\[
 \kappa_1 := \min {\kappa_{1,1},\kappa_{1,2}} \geq \kappa_0 + c_0 (2\mu - \kappa)^\alpha,
\]
and since
\[
 p \geq \frac{m}{m-\kappa_1},\quad \lambda < m\frac{2\mu - \kappa_1}{m-\kappa_1} ,
\]
for any $D_3 \subsubset D$, we arrive at
\[
\vrac{w}_{(p_{\kappa_1},\infty)_{m\frac{2\mu-\kappa_1}{m-\kappa_1}},D_3} < \infty.
\]
Varying this in $D_3 \subsubset D$, we have \eqref{eq:claiml2} for $\kappa_1$.
If $2\mu = m$, we use this same argument, to conclude that $w \in L^p(D_3)$ for some $p > 2$, which is already the claim of Theorem \ref{th:upto2}.
\paragraph*{Estimating the growth of $\kappa$.}
Iterating this procedure (for smaller and smaller $\theta$ in \eqref{eq:regp}), we obtain $\kappa_k \in [\mu,2\mu)$, and
\[
 \kappa_{k+1} \geq \kappa_{k}+c_0 (2\mu-\kappa)^\alpha.
\]
Since the sequence $(\kappa_{k})_k$ is monotone and bounded, and the only fixed point is $\kappa_\infty = 2\mu$, for any $\varepsilon > 0$ there is a step-count $L$ such that $\abs{\kappa_L - 2\mu} < \varepsilon$. This shows \eqref{eq:claiml2}.\\
\subsection*{Integrability slightly above 2}
By the arguments above, fixing $\tilde{D} \subsubset D$, going back to \eqref{eq:Dmmtauvest}, if $2\mu - \kappa < \varepsilon$ small enough, for $\tau \in (\varepsilon,\mu]$, ignoring $\sigma_\kappa > 0$,
\[
\sup_{B_r \subset \tilde{D}} r^{2\mu-m}\ \vrac{\laps{\mu-\tau}v}_{(\frac{m}{m+\mu-\tau-\kappa},\infty),B_{\Lambda^{-1}r}(x_0)} 
\aleq C_{\kappa,w,\tilde{D}}.
\]
If $2\mu = m$, choosing $\tau = \mu$, we have
\[
 \frac{m}{m+\mu-\mu-\kappa} \xrightarrow{\kappa \to 2\mu=m} \infty,
\]
which proves Theorem \ref{th:upto2}, and in fact even Theorem \ref{th:main}. So let from now on $2\mu < m$, $\mu \leq 1$. Then for $\lambda_{s,\varepsilon} \in (0,m)$, $s := \mu-\tau$,
\[
\begin{ma}
&&\frac{\lambda_{s,\varepsilon}-m}{\frac{m}{m+\mu-\tau-\kappa}} = 2\mu-m\\
&\Leftrightarrow& \lambda_{s,\varepsilon} = \frac{m}{m+\mu-\tau-\kappa} (3\mu-\tau-\kappa) \xrightarrow{\tau \to \mu,\kappa \to 2\mu} 0\\
\end{ma}
\]
and 
\[
\begin{ma}
 \fracm{\tilde{p}} &:=&  \frac{m+\mu-\tau-\kappa}{m} - \frac{\mu-\tau}{\frac{m}{m+\mu-\tau-\kappa} (3\mu-\tau-\kappa)} 
= \frac{m+\mu-\tau-\kappa}{m} - \frac{(\mu-\tau)\brac{m+\mu-\tau-\kappa}}{{m}(3\mu-\tau-\kappa)}\\
&=& 1 + \frac{\mu-\tau-\kappa}{m(3\mu-\tau-\kappa)} \brac{2\mu+\tau-\kappa}  - \frac{(\mu-\tau)}{3\mu-\tau-\kappa}
\end{ma}
\]
we have by Adams' \cite{Adams75},
\[
 v \in L_{loc}^{(\tilde{p},\infty)_{\lambda_{s,\varepsilon}}}(D).
\]
One checks that there is $\varepsilon > 0$ such that if $\abs{\tau-\mu} < \varepsilon$, $2\kappa-\mu < \varepsilon$, then $\tilde{p} > 2$, $\lambda_{s,\varepsilon} < 2\mu$.

\newpage
\subsection{Estimates of the \texorpdfstring{$H$}{H}-term: Proof of Theorem \ref{th:rhsest} (I)}
This is to estimate for $\varphi \in C_0^\infty(B_r)$ the following term 
\begin{equation}\label{eq:rhs:secondguy:1}
\int w\ H_{\mu}(P-I,\varphi) = \int \lapms{\beta} w\ \laps{\beta} H_{\mu}(P,\varphi)
\end{equation}

\begin{lemma}\label{la:rhs:Htermest}
Let $\mu \in (0,\frac{m}{2}]$, $\mu \leq 1$ or $\mu = \frac{m}{2}$. For any $\kappa \in [\mu,2\mu)$, there are $C_{\kappa,\mu} > 0$, $\tau \in (0,\mu)$ such for any $\varphi \in C_0^\infty(B_{\Lambda^{-1} r})$ the following holds: If $\supp (P-I) \subset B_{\Lambda^{-1} r}$, 
\[
\begin{ma}
(\Lambda^{-1} r)^{2\mu-m} \int w\ H_{\mu}(P-I,\varphi) &\leq& C_{\kappa,\mu}\ \vrac{\laps{\tau} \varphi}_{(\frac{m}{\kappa+\tau-\mu},2)}\ (\Lambda^{-1} r)^{\mu - \frac{m}{2}} \vrac{\laps{\mu}P}_{2}\ \vrac{w}_{\brac{{p_\kappa},\infty}_{\lambda_\kappa},B_{r}}  \\
 &&+ C_{\kappa,\mu}\ \vrac{\laps{\tau} \varphi}_{(\frac{m}{\kappa+\tau-\mu},2)}\ (\Lambda^{-1} r)^{\mu-\frac{m}{2}}\ \vrac{\laps{\mu}P}_{2}\ \sum_{k=1}^\infty (2^k \Lambda)^{\kappa-3\mu} [w]_{\brac{{p_\kappa},\infty}_{{\lambda_\kappa} },A_{r}^k}
\end{ma}
\]
where we recall the definition $A_r^k$ from \eqref{eq:def:Alambdark}, $\lambda_\kappa$ from \eqref{eq:def:lambdakappa}, and $p_\kappa$ from \eqref{eq:def:pkappa}.
As for the asymptotic behavior as $\kappa \to 2\mu$, one can choose $\tau$ approaching $\max\{\mu-1,0\}$, and $C_{\kappa,\mu}$ blows up.
\end{lemma}
\begin{proof}[Proof of Lemma \ref{la:rhs:Htermest}]
For a somewhat clearer presentation, we are going to show the following claim for $\varphi \in C_0^\infty(B_{r})$ and $\supp (P-I) \subset B_{r}$
\[
\begin{ma}
r^{2\mu-m} \int w\ H_{\mu}(P-I,\varphi) &\leq& C_{\kappa,\mu}\ \vrac{\laps{\tau} \varphi}_{(\frac{m}{\kappa+\tau-\mu},2)}\ r^{\mu - \frac{m}{2}} \vrac{\laps{\mu}P}_{2}\ \vrac{w}_{\brac{{p_\kappa},\infty}_{\lambda_\kappa},B_{\Lambda r}}  \\
 &&+ C_{\kappa,\mu}\ \vrac{\laps{\tau} \varphi}_{(\frac{m}{\kappa+\tau-\mu},2)}\ r^{\mu-\frac{m}{2}}\ \vrac{\laps{\mu}P}_{2}\ \sum_{k=1}^\infty (2^k \Lambda)^{\kappa-3\mu} [w]_{\brac{{p_\kappa},\infty}_{{\lambda_\kappa} },A_{\Lambda,r}^k}.
\end{ma}
\]
Applied to $\tilde{r} := \Lambda^{-1} r$ gives the original claim.\\[1em]
As usual, we decompose
\[
\int w\ H_{\mu}(P-I,\varphi) =  I + \sum_{k = 1}^\infty II_k,
\]
where
\[
 I := \int \chi_{B_{\Lambda r}}w\ H_{\mu}(P-I,\varphi),
\]
and, denoting $A_k := A_{\Lambda, r}^k$,
\[
 II_k := \int w\ H_{\mu}(P-I,\varphi)\chi_{A_k}.
\]
\underline{As for $II_k$}, since $\supp \varphi \cup \supp (P-I) \subset \overline{B_r}$
\[
 H_{\mu}(P-I,\varphi) \chi_{A_k} = \chi_{A_k}\laps{\mu} ((P-I)\varphi). 
\]
By Lemma~\ref{la:QuasiLocality} we then have for any $\tau \in (0,\mu]$, using also Lemma \ref{la:sobpoinc},
\[
\begin{ma}
 \vrac{H_{\mu}(P-I,\varphi)}_{(\frac{m}{\kappa},1).A_k} &\aleq& \ \brac{2^k \Lambda r}^{-m-\mu}\ \brac{2^k \Lambda r}^{\kappa}\
 r^{\frac{m}{2}-\kappa+\mu}\ \vrac{\varphi}_{(\frac{m}{\kappa-\mu},\infty)}\ \vrac{P-I}_{2}\\ 
 &\aleq& \brac{2^k \Lambda r}^{-m-\mu}\ \brac{2^k \Lambda r}^{\kappa}\
 r^{\frac{m}{2}-\kappa+2\mu} \vrac{\laps{\tau} \varphi}_{(\frac{m}{\kappa+\tau-\mu},\infty)}\ \vrac{\laps{\mu}P}_{2}\\
 &=& \brac{2^k \Lambda}^{-m+\kappa-\mu}\ \vrac{\laps{\tau} \varphi}_{(\frac{m}{\kappa+\tau-\mu},\infty)}\ r^{\mu-\frac{m}{2}}\ \vrac{\laps{\mu}P}_{2}.
 \end{ma} 
\]
Consequently,
\[
\begin{ma}
  \abs{II_k} &\aleq& \vrac{w\chi_{A_k}}_{({p_\kappa},\infty)}\ \brac{2^k \Lambda}^{-m+\kappa -\mu}\ \vrac{\laps{\tau} \varphi}_{(\frac{m}{\kappa+\tau-\mu},\infty)}\ r^{\mu-\frac{m}{2}}\ \vrac{\laps{\mu}P}_{2}\\
  &\overset{\eqref{eq:mm2mulambda}}{\aleq}& (2^k \Lambda r)^{m-2\mu}\ [w\chi_{A_k}]_{({p_\kappa},\infty)_{\lambda_\kappa}}\ \brac{2^k \Lambda}^{-m+\kappa -\mu}\ \vrac{\laps{\tau} \varphi}_{(\frac{m}{\kappa+\tau-\mu},\infty)}\ r^{\mu-\frac{m}{2}}\ \vrac{\laps{\mu}P}_{2}\\ 
  &\aleq& r^{m-2\mu}\ (2^k \Lambda)^{\kappa - 3\mu}\ \vrac{\laps{\tau} \varphi}_{(\frac{m}{\kappa+\tau-\mu},\infty)}\ r^{\mu-\frac{m}{2}}\ \vrac{\laps{\mu}P}_{2}\ [w\chi_{A_k}]_{({p_\kappa},\infty)_{\lambda_\kappa}}.
 \end{ma}
\]
\underline{As for $I$}, set $\tilde{w} := \chi_{B_\Lambda r} w$ and write
\[
 \int \tilde{w}\ H_{\mu}(P ,\varphi) = \int \lapms{\beta} \tilde{w}\ \laps{\beta}H_{\mu}(P ,\varphi) 
\]
Actually, the claim follows quite straight forward from \eqref{eq:c2:laphH1inhardy} for $\mu \leq 1$, $\beta := \mu$, but the pointwise estimates on $H$, Lemma \ref{la:lapsbetaHalpha}, are strong enough to deal with our situation, and they do not make use of para-products which were necessary for the proof of \eqref{eq:c2:laphH1inhardy}: By Lemma \ref{la:adams}
\[
 \vrac{\lapms{\beta} \tilde{w}}_{(p_1,\infty)_{{\lambda_\kappa}}} \aleq \vrac{\tilde{w}}_{({p_\kappa},\infty)_{{\lambda_\kappa}}}
\]
where for $\beta < \min (2\mu - \kappa,1)$,
\[
 \fracm{p_1} = \frac{m-\kappa}{m}\ \frac{2\mu - \kappa- \beta}{2\mu - \kappa} \in (0,1).
\]
If $\mu = \frac{m}{2}$, we set $\beta = 0$, if $\mu < \frac{m}{2}$, let $\epsilon > 0$ such that $\mu +\epsilon < \frac{m}{2}$. Now we estimate $\sabs{\laps{\beta} H_\mu(P,\varphi)}$, applying Lemma \ref{la:lapsbetaHalpha} for any $\tau \in (\max\{\beta,\mu+\beta-1\},\mu]$,
we have to control terms of the form (for $s \in (0,\mu)$, $t \in (0,\tau)$, $\tau-\beta -s-t \in [0,\epsilon)$)
\[
  \lapms{\tau-\beta -s-t} \brac{\lapms{s} \abs{\laps{\mu} P}\ \lapms{t}\abs{\laps{\tau} \varphi}}.
\]
We have
\[
 \vrac{\lapms{s} \abs{\laps{\mu} P}}_{(p_2,2)}  \aleq \vrac{\laps{\mu} P}_{2}, \quad \fracm{p_2} = \fracm{2} - \frac{s}{m} \in (0,1),
\]
\[
 \vrac{\lapms{t}\abs{\laps{\tau} \varphi}}_{(p_3,2)} \aleq \vrac{\laps{\tau} \varphi}_{(\frac{m}{\kappa+\tau-\mu},2)}, \quad \fracm{p_3} = \frac{\kappa +\tau-\mu}{m} - \frac{t}{m} \in (0,1).
\]
Note that
\[
 0 < \fracm{p_2} + \fracm{p_3} = \fracm{2} + \frac{\kappa +\tau-\mu-s-t}{m} < \fracm{2} + \frac{\kappa +\tau-\mu+\epsilon-\tau+\beta}{m} 
 < \fracm{2} + \frac{\epsilon+\mu}{m}<1,
\]
consequently,
\[
\vrac{\lapms{\tau-\beta -s-t} \brac{\lapms{s} \abs{\laps{\mu} P}\ \lapms{t}\abs{\laps{\tau} \varphi}}}_{(p_4,1)} \aleq 
\vrac{\laps{\mu} P}_{2}\ \vrac{\laps{\tau} \varphi}_{(\frac{m}{\kappa+\tau-\mu},2)},
\]
where
\[
\fracm{p_4} = \fracm{p_2}+\fracm{p_3}  - \frac{\tau-\beta -s-t}{m} = 
\fracm{2}+ \frac{\kappa +\beta-\mu }{m} \in (0,1).
\]
%
Now we have to ensure that the $f(\beta) \leq 1$ for admissible $\beta$ (and admissible $\tau$):
\[
f(\beta) := \fracm{p_1} + \fracm{p_4} 
=  \frac{3}{2} - \frac{\mu}{m} - \beta \frac{m-2\mu}{m(2\mu - \kappa)} > 0.
\]
Obviously, $f(0) = 1$ holds, if $\mu = \frac{m}{2}$ (so $\beta = 0$, and $\tau$ arbitrarily between $(\mu-1,\mu]$). As for the case $\mu < \frac{m}{2}$, $\mu \leq 1$, We have $2\mu - \kappa \leq 1$ for $\kappa \in [\mu,2\mu)$, then 
\[
f(2\mu-\kappa) = \frac{1}{2}  + \frac{\mu}{m} < 1.
 \]
so we can take $\beta < 1$ sufficiently close to $2\mu-\kappa$, so that $f(\beta) < 1$, and take $\tau \in (\beta,\mu)$ sufficiently close to or greater than $2\mu - \kappa$. Consequently,
\[
\begin{ma}
 \abs{I} &\aleq& \int_{B_{4r}} \lapms{\beta} \tilde{w}\ \laps{\beta}H_{\mu}(P,\varphi) + \sum_{k=1}^\infty \int_{A_{4r}^k} \lapms{\beta} \tilde{w}\ \laps{\beta}H_{\mu}(P,\varphi)\\
 &\aleq& \vrac{\lapms{\beta} \tilde{w}}_{(p_1,\infty),B_{4r}}\ \vrac{\laps{\beta}H_{\mu}(P,\varphi)}_{(p_4,1)}\ r^{m-\frac{m}{p_1}-\frac{m}{p_4}}\\
 && + \sum_{k=1}^\infty \vrac{\lapms{\beta} \tilde{w}}_{(p_1,\infty),A_{4r}^k}\ \vrac{\laps{\beta}H_{\mu}(P,\varphi)}_{(p_4,1),A_{4r}^k}\ (2^k r)^{m-\frac{m}{p_1}-\frac{m}{p_4}}\\
 &\aleq& r^{\frac{m-{\lambda_\kappa}}{p_1}} \vrac{\lapms{\beta} \tilde{w}}_{(p_1,\infty)_{\lambda_\kappa}}\ \vrac{\laps{\beta}H_{\mu}(P,\varphi)}_{(p_4,1)}\ r^{m-\frac{m}{p_1}-\frac{m}{p_4}}\\
 && + \sum_{k=1}^\infty (2^k r)^{\frac{m-{\lambda_\kappa}}{p_1}} \vrac{\lapms{\beta} \tilde{w}}_{(p_1,\infty)_{\lambda_\kappa}}\ \vrac{\laps{\beta}H_{\mu}(P,\varphi)}_{(p_4,1),A_{4r}^k}\ (2^k r)^{m-\frac{m}{p_1}-\frac{m}{p_4}}\\
 &\aleq& r^{\frac{m-{\lambda_\kappa}}{p_1}} \vrac{\tilde{w}}_{({p_\kappa},\infty)_{\lambda_\kappa}}\ \vrac{\laps{\beta}H_{\mu}(P,\varphi)}_{(p_4,1)}\ r^{m-\frac{m}{p_1}-\frac{m}{p_4}}\\
 && + \sum_{k=1}^\infty (2^k r)^{\frac{m-{\lambda_\kappa}}{p_1}} \vrac{\tilde{w}}_{({p_\kappa},\infty)_{\lambda_\kappa},A_{4r}^k}\ \vrac{\laps{\beta}H_{\mu}(P,\varphi)}_{(p_4,1),A_{4r}^k}\ (2^k r)^{m-\frac{m}{p_1}-\frac{m}{p_4}}.\\
\end{ma}
\]
By Proposition~\ref{pr:HestRm}, for the same $\tau$ as above,
\[
 \vrac{\laps{\beta}H_{\mu}(P,\varphi)}_{(p_4,1)} \aleq \vrac{\laps{\mu} P}_{2}\ \vrac{\laps{\tau} \varphi}_{\frac{m}{\kappa-\tau-\mu},2}.
\]

Now we apply Proposition~\ref{pr:HestArk} (using that $\varphi$ and $P-I$ have support in $B_r$), and using
\[ 
\begin{ma}
&&\frac{m-\frac{m(2\mu-\kappa)}{m-\kappa}}{p_1} + m-\frac{m}{p_1}-\frac{m}{p_4} -m-\beta + \frac{m}{p_4}\\
&=& -2\mu + \kappa,
\end{ma}
\]
and
\[ 
 \frac{m-{\lambda_\kappa}}{p_1} + m-\frac{m}{p_1}-\frac{m}{p_4}+\frac{m}{2}-\mu = m -2\mu
\]
we conclude 
\[
\begin{ma}
 \abs{I} &\aleq& r^{m-2\mu}\ \vrac{\tilde{w}}_{({p_\kappa},\infty)_{\lambda_\kappa}}\ 
 r^{\mu-\frac{m}{2}}\vrac{\laps{\mu} P}_{2}\ \vrac{\laps{\tau} \varphi}_{\frac{m}{\kappa-\tau-\mu},2}
 \\
 && + r^{m-2\mu}\ \sum_{k=1}^\infty 2^{k (-2\mu + \kappa)}
 \vrac{\tilde{w}}_{({p_\kappa},\infty)_{\lambda_\kappa},A_{4r}^k}\  \vrac{\laps{\tau} \varphi}_{(\frac{m}{\kappa+\tau-\mu},\infty)}\ r^{\mu-\frac{m}{2}}\vrac{\laps{\mu} P}_{(2,\infty)}\\ 
 &\aleq& C_{\kappa} r^{m-2\mu}\ \vrac{w \chi_{B_{\Lambda r}}}_{({p_\kappa},\infty)_{\lambda_\kappa}}\ 
 r^{\mu-\frac{m}{2}}\vrac{\laps{\mu} P}_{2}\ \vrac{\laps{\tau} \varphi}_{\frac{m}{\kappa-\tau-\mu},2}.\\
 \end{ma}
\]
\end{proof}

\newpage
\subsection{Better integrability for transformed potential: Proof of Theorem \ref{th:rhsest} (II)}\label{s:omegapest}
This section is devoted to the proof of the following Lemma:
\begin{lemma} \label{la:rhs:OmegaPest}
Let $B_r \subset \R^m$, $\Omega$ as in \eqref{eq:OmegaisRiesztrafo}, $\Lambda > 2$. There exists $P: \R^m \to SO(N)$, $P \equiv I$ on $\R^m \backslash B_{\Lambda^{-1} r}$, with the estimate
\begin{equation}\label{eq:inttransp:smallPscaled}
 (\Lambda^{-1} r)^{\frac{2\mu-m}{2}}\ \Vert \laps{\mu}P \Vert_{2,\R^m} \aleq \theta,
\end{equation}
such that for any $\tau \in (0,\mu]$ sufficiently close or greater than $2\mu-\kappa$, $\kappa \in [\mu,2\mu)$, $\theta > 0$ from \eqref{eq:regp} in $D = B_{r}$, and for any $\varphi \in C_0^\infty(B_{\Lambda^{-1} r})$, if $\mu \in (0,1]$, or $\mu = \frac{m}{2}$,
\[
\begin{ma}
 (\Lambda^{-1} r)^{2\mu-m} \int \brac{(\laps{\mu}P) P^T w + P \Omega [P^T w]}\ \varphi  
&\leq&C_{\kappa,\mu}\ \theta\ \vrac{\laps{\tau} \varphi}_{(\frac{m}{\kappa+\tau-\mu},1)}\ \vrac{w}_{\brac{{p_\kappa},\infty}_{\lambda_\kappa},B_{r} }  \\ 
 &&+ C_{\kappa,\mu}\ \theta\ \vrac{\laps{\tau} \varphi}_{(\frac{m}{\kappa+\tau-\mu},2)}\ \sum_{k=1}^\infty (2^k \Lambda)^{\kappa-3\mu}\ [ w]_{\brac{{p_\kappa},\infty}_{{\lambda_\kappa} },A_{r}^k}.
\end{ma} 
\]
where we recall the definition $A_r^k$ from \eqref{eq:def:Alambdark}, $\lambda_\kappa$ from \eqref{eq:def:lambdakappa}, and $p_\kappa$ from \eqref{eq:def:pkappa}.\\
\end{lemma}
As in the proof of Lemma~\ref{la:rhs:Htermest}, we prove the scaled claim for replacing $r$ by $\Lambda r$ which makes the presentation of the proof somewhat lighter: We are going to show the existence of $P$ such that for $\varphi \in C_0^\infty(B_r)$
\begin{equation}\label{eq:higheromega:proofclaim}
\begin{ma}
 r^{2\mu-m} \int \brac{(\laps{\mu}P) P^T w + P \Omega [P^T w]}\ \varphi  
&\leq&C_{\kappa,\mu}\ \theta\ \vrac{\laps{\tau} \varphi}_{(\frac{m}{\kappa+\tau-\mu},1)}\ \vrac{w}_{\brac{{p_\kappa},\infty}_{\lambda_\kappa},B_{\Lambda r} }  \\ 
 &&+ C_{\kappa,\mu}\ \theta\ \vrac{\laps{\tau} \varphi}_{(\frac{m}{\kappa+\tau-\mu},2)}\ \sum_{k=1}^\infty (2^k \Lambda)^{\kappa-3\mu}\ [ w]_{\brac{{p_\kappa},\infty}_{{\lambda_\kappa} },A_{\Lambda, r}^k},
\end{ma} 
\end{equation}
Fix $B_r \subset \R^m$. In order to prove this claim, note that
\[
 \int \brac{(\laps{\mu}P) P^T w + P \Omega [P^T w]}\ \varphi = \int \brac{(\laps{\mu}P) P^T w + P \chi_{B_r}\Omega [P^T w]}\ \varphi,
\]
so we are going to assume that the $A_l$ in \eqref{eq:OmegaisRiesztrafo}
\begin{equation}\label{eq:omegapest:Omegalocalguy}
 \supp A_l \subset B_r, \quad \Omega[] = \chi_{B_r} \Omega[]
\end{equation}
and consequently
assuming (from \eqref{eq:regp}) that
\begin{equation}\label{eq:higheromega:ourregp}
r^{\frac{2\mu-m}{2}}\ \vrac{A_l}_{2,\R^m}\ \vrac{f}_2 + \sup_{\rho \in (0,\Lambda r)} \rho^{\frac{2\mu-m}{2}}\ \vrac{\Omega [f]}_{1,B_{\rho}} \aleq \theta\ \vrac{f}_2
\end{equation}

Let  $P: \R^m \to SO(N)$ be the minimizer, $P \equiv I$ on $\R^m \backslash B_r$, of $E(\cdot) \equiv E_{r,x,\Lambda_\mu,1,2}(\cdot)$, where $\Lambda_\mu$ is from Lemma \ref{la:energy:minexists}.
Using \eqref{eq:energy:minex:Pest}, \eqref{eq:omegapest:Omegalocalguy}, we have the estimates (for from now on fixed $\Lambda > 2$),
\begin{equation}\label{eq:inttransp:smallP}
r^{\frac{2\mu-m}{2}}\ \Vert \laps{\mu}P \Vert_{2,\R^m} \aleq \theta,
\end{equation}
which after rescaling amounts to \eqref{eq:inttransp:smallPscaled}, and with the help of \eqref{eq:higheromega:ourregp},
\begin{equation}\label{eq:inttransp:smallOmegaP}
 r^{\frac{2\mu-m}{2}}\ \Vrac{ (\laps{\mu}P) P^T f + P \Omega [P^T f]}_{1,B_{\Lambda r}} \aleq \theta\ \Vert f \Vert_{2,\R^m}.
\end{equation}
Let
\[
 w = w \chi_{B_{\Lambda r}} + \sum_{k=1}^\infty w \chi_{A^{k}_{ \Lambda r}} =: w_0 + \sum_{k=1}^\infty w_k.
\]
Then,
\[
\begin{ma}
&&\int \brac{(\laps{\mu}P) P^T w + P \Omega [P^T w]}\ \varphi\\
&=& \int (\laps{\mu}P) P^T w_0 \varphi + P \Omega [P^T w_0 \varphi]\
- \int P \mathcal{C}(\varphi,\Omega)[P^T w_0]
+ \sum_{k=1}^\infty \int P \Omega [P^T w_k]\ \varphi\\
&=:& I - II + III.
\end{ma}
\]
\subsection*{The disjoint support part (III)}
Since $\mu \leq \kappa < 2\mu$,	
\[
\begin{ma}
 \int P \Omega [P^T w_k]\ \varphi 
 &\overset{\eqref{eq:OmegaisRiesztrafo}}{\aleq}& \vrac{A}_{2,B_r}\ \vrac{\varphi}_2\ \vrac{\Rz[P^T w_k]}_{\infty,B_r}\\
&\overset{\sref{P}{la:QuasiLocality}}{\aleq}&  \vrac{A}_{2,B_r}\ r^{\frac{m}{2}-\kappa+\mu}\ \vrac{\laps{\tau} \varphi}_{\frac{m}{\kappa+\tau-\mu}}\ (2^k \Lambda r)^{-m+\kappa}\ \vrac{w_k}_{\brac{{p_\kappa},\infty}}\\
&\overset{\eqref{eq:mm2mulambda}}{\aleq}&  r^{\mu-\frac{m}{2}}\ \vrac{A}_{2,B_r}\ \vrac{\laps{\tau}\varphi}_{\frac{m}{\kappa+\tau-\mu}}\ (2^k \Lambda)^{-m+\kappa}\ (2^k \Lambda r)^{m-2\mu}\ [w]_{\brac{{p_\kappa},\infty}_{{\lambda_\kappa}},A_{\Lambda,r}^k }\\
&\overset{\eqref{eq:higheromega:ourregp}}{\aleq}&   \theta\ r^{m-2\mu}\ (2^k \Lambda)^{\kappa-2\mu}\ \vrac{\laps{\tau}\varphi}_{\frac{m}{\kappa+\tau-\mu}}\ [w]_{\brac{{p_\kappa},\infty}_{{\lambda_\kappa}},A_{\Lambda,r}^k }.
\end{ma}
\]


\subsection*{The same-support/commutator part (II)}
We have
\[
 \abs{II} \aleq \vrac{A}_2\ \vrac{\mathcal{C}(\varphi,\Rz)[P^T w_0]}_{2,B_r} \overset{\eqref{eq:higheromega:ourregp}}{\aleq}  r^{\frac{m-2\mu}{2}}\ \theta\ \vrac{\mathcal{C}(\varphi,\Rz)[P^T w_0]}_{2,B_r}.
\]
Now we apply Lemma \ref{la:commRzTransf}, and have for arbitrary $\delta \in (0,1)$, $\gamma_{1,2} \in (0,\delta)$,
\[
 \abs{\mathcal{C}(\varphi,\Rz)[P^T w_0]} \aleq \lapms{\delta-\gamma_1} \sabs{\laps{\delta} \varphi}\ \lapms{\gamma_1} \abs{w_0}
+ C_{\Rz,\delta,\gamma_2}\ \lapms{\gamma_2}\brac{\sabs{\laps{\delta} \varphi}\ \lapms{\delta-\gamma_2}\abs{w_0}}
\]
Now, if we choose $\delta < \tau$
\[
 \vrac{\lapms{\delta_1-\gamma_1} \sabs{\laps{\delta_1} \varphi}}_{(\frac{m}{\gamma_1+\kappa-\mu},q)} \aleq \vrac{\laps{\tau} \varphi}_{(\frac{m}{\tau+\kappa-\mu},q)},
\]
and for $\beta < 2\mu - \kappa$, using \cite{Adams75}, see Lemma \ref{la:adams},
\[
 r^{\frac{{\lambda_\kappa}-m}{p_{\gamma_1}}} \vrac{\lapms{\beta} w_0}_{(p_\beta,\infty),B_r}\ \aleq \vrac{\lapms{\beta} w_0}_{(p_\beta,\infty)_{{\lambda_\kappa}}} \aleq \vrac{w_0}_{({p_\kappa},\infty)_{{\lambda_\kappa}}}
\]
where
\[
 \fracm{p_\beta} = \frac{m-\kappa}{m}\ \frac{2\mu - \kappa- \beta}{2\mu - \kappa} \in (0,1).
\]
Now,
\begin{equation}\label{eq:highomega:pgamma1}
\begin{ma}
  &&\fracm{p_{\gamma_1}} + \frac{\gamma_1+\kappa-\mu}{m} \\
  &=&  \frac{\mu}{m} +(m - 2\mu)\frac{(2\mu-\kappa)-\gamma_1}{m(2\mu-\kappa)}\\
  &\leq& \fracm{2},
 \end{ma} \end{equation} if we choose $\gamma_1 \in (0,2\mu - \kappa)$ as follows: If $\mu = \frac{m}{2}$ we can choose $\gamma$ arbitrarily. If $\mu < \frac{m}{2}$ and $\mu 
\leq 1$, then we pick $\gamma_1$ sufficiently close to $2\mu -\kappa \leq 1$. That is, for any $\tau < \mu$ sufficiently close or greater than $2\mu - \kappa$ such that there 
is a $\gamma_1 < \delta < \tau$, $\delta < 2\mu - \kappa$, satisfying the above equation, we have \[
 \vrac{ \lapms{\delta-\gamma_1} \sabs{\laps{\delta_1} \varphi}\ \lapms{\gamma_1} 
\abs{w_0}}_{2,B_r} \aleq r^{\frac{m}{2}-\frac{m}{p_{\gamma_1}} - \brac{\gamma_1+\kappa-\mu}}\ r^{\frac{m-{\lambda_\kappa}}{p_{\gamma_1}}} \ \vrac{\laps{\tau} 
\varphi}_{(\frac{m}{\tau+\kappa-\mu},2)} \ \vrac{w_0}_{({p_\kappa},\infty)_{{\lambda_\kappa}}}, \] and \[
 \frac{m}{2}-\frac{m}{p_{\gamma_1}} - \brac{\gamma_1+\kappa-\mu} + \frac{m-{\lambda_\kappa}}{p_{\gamma_1}} = 
 \frac{m}{2} - \brac{\gamma_1+\kappa-\mu}  -(2\mu-\kappa-\gamma_1) = \frac{m}{2}-\mu.
\]
As for the second term, for $\delta-\gamma_2 < 2\mu - \kappa$, using the formula \eqref{eq:highomega:pgamma1} with $\delta$ instead of $\gamma_1$,
\[
\begin{ma}
 \fracm{p_2} &:=& \frac{\delta+\kappa-\mu}{m} + \fracm{p_{\delta-\gamma_2}}\\
 &=& \frac{\delta+\kappa-\mu}{m} + \fracm{p_{\delta}} + \gamma_2\frac{m-\kappa}{m(2\mu-\kappa)} \leq \fracm{2} + \gamma_2\frac{m-\kappa}{m(2\mu-\kappa)} < 1,\\
\end{ma}
\]
if we choose $\gamma_1 < \delta$ (as above $\gamma_1$) close enough $2\mu - \kappa$, and $\gamma_2$ very small. Consequently, if we set
\[
 \lambda :={\lambda_\kappa},
\]
and $\tilde{\lambda} \in (0,m)$ such that $\frac{\tilde{\lambda}-m}{p_2} = \frac{\lambda-m}{p_{\delta-\gamma_2}}$, that is
\begin{equation}\label{eq:highomega:tildelambda}
\begin{ma}
 &&\frac{\tilde{\lambda}}{p_2}= \frac{\lambda-m}{p_{\delta-\gamma_2}}+\frac{m}{p_2}
 = \frac{{\lambda_\kappa}-m}{p_{\delta-\gamma_2}}+\delta+\kappa-\mu + \frac{m}{p_{\delta-\gamma}}\\
 &=& \brac{{\lambda_\kappa}}\frac{m-\kappa}{m}\ \frac{2\mu - \kappa- (\delta-\gamma_2)}{2\mu - \kappa}+\delta+\kappa-\mu

 \\
 &=& \mu +\gamma_2
 \end{ma}
\end{equation}
then
\[
\begin{ma}
\vrac{ \laps{\delta_2} \varphi\ \lapms{\delta_2-\gamma_2}\abs{w_0} }_{(p_2,2)_{\tilde{\lambda}}}
&\aeq& \sup_{B_\rho} \rho^{\frac{\tilde{\lambda}-m}{p_2}}\vrac{\sabs{\laps{\delta_2} \varphi}\ \lapms{\delta_2-\gamma_2}\abs{w_0} }_{(p_2,2),B_\rho}\\
&\aleq& \vrac{\laps{\tau} \varphi}_{(\frac{m}{\tau+\kappa-\mu},2)}\ \sup_{B_\rho} \rho^{\frac{\tilde{\lambda}-m}{p_2}} \vrac{\lapms{\delta_2-\gamma_2}\abs{w_0} }_{(p_{\delta_2-\gamma_2},\infty),B_\rho}\\
&\aeq& \vrac{\laps{\tau} \varphi}_{(\frac{m}{\tau+\kappa-\mu},2)}\ \vrac{\lapms{\delta_2-\gamma_2}\abs{w_0} }_{(p_{\delta_2-\gamma_2},\infty)_{\lambda},B_\rho}\\
&\aleq&\vrac{\laps{\tau} \varphi}_{(\frac{m}{\tau+\kappa-\mu},2)}\ \vrac{w_0}_{({p_\kappa},\infty)_{{\lambda_\kappa}}}.
\end{ma}
\]
Now observe
\[
\begin{ma}
 &&\fracm{2} - \brac{\fracm{p_2} - \frac{\gamma_2}{\tilde{\lambda}}} \overset{\eqref{eq:highomega:tildelambda}}{=}  \fracm{2} - \frac{\mu}{p_2(\mu + \gamma_2)}\\
 &=& \fracm{2} - \frac{\mu}{\mu + \gamma_2} \brac{\frac{\delta_2+\kappa-\mu}{m} + \frac{m-\kappa}{m}\ \frac{2\mu - \kappa- (\delta_2-\gamma_2)}{2\mu - \kappa}}\\
&=& \fracm{2} - \frac{\mu}{\mu + \gamma_2} 
 \brac{
 \brac{(2\mu-\kappa)-\delta_2}\frac{m-2\mu}{m(2\mu - \kappa)}
 + \frac{\mu}{m} + \frac{m-\kappa}{m}\ \frac{\gamma_2}{2\mu - \kappa}} \geq 0,\\
\ \end{ma}
\]
for sufficiently small $\gamma_2$ and $\delta_2$ sufficiently close to $2\mu-\kappa$. In fact, this holds obviously, if $\frac{\mu}{m} < \fracm{2}$. If $\frac{\mu}{m} = \fracm{2}$, we have
\[
 \frac{\mu}{\mu + \gamma_2} 
 \brac{
 \brac{(2\mu-\kappa)-\delta_2}\frac{m-2\mu}{m(2\mu - \kappa)}
 + \frac{\mu}{m} + \frac{m-\kappa}{m}\ \frac{\gamma_2}{2\mu - \kappa}}
 = \frac{\mu}{\mu + \gamma_2} \brac{\fracm{2} + \frac{\gamma_2}{2\mu}} = \fracm{2}
\]
Moreover, one checks
\[
\frac{m}{2}-\frac{\mu}{\mu+\gamma_2}\frac{m}{p_2}+\frac{\mu}{\mu+\gamma_2}\frac{m-\tilde{\lambda}}{p_2}
= \frac{m}{2}-\frac{\mu}{\mu+\gamma_2}\frac{\tilde{\lambda}}{p_2}
\overset{\eqref{eq:highomega:tildelambda}}{=} \frac{m}{2}-\mu.
\]
Thus,
\[
\begin{ma}
\vrac{ \lapms{\gamma_2} (\laps{\delta_2} \varphi\ \lapms{\delta_2-\gamma_2}\abs{w_0} )}_{2,B_r} &\aleq&r^{\frac{m}{2}-\mu}\ \vrac{ \lapms{\gamma_2} (\laps{\delta_2} \varphi\ \lapms{\delta_2-\gamma_2}\abs{w_0} )}_{(\frac{p_2 (\mu + \gamma_2)}{\mu},2)_{\tilde{\lambda}}}\\
&\aleq&r^{\frac{m}{2}-\mu}\ \vrac{ \laps{\delta_2} \varphi\ \lapms{\delta_2-\gamma_2}\abs{w_0} }_{(p_2,2)_{\tilde{\lambda}}}\\
&\aleq&r^{\frac{m}{2}-\mu}\ \vrac{\laps{\tau} \varphi}_{(\frac{m}{\tau+\kappa-\mu},2)}\ \vrac{w_0}_{({p_\kappa},\infty)_{{\lambda_\kappa}}}.
\end{ma}
\]

\subsection*{The same-support/commutator part (I)}
Here, we decompose
\[
w_0 \varphi =  \laps{\mu}\brac{\eta_{\Lambda r} \brac{\lapms{\mu}(w_0 \varphi)}} + \laps{\mu}\brac{\brac{1-\eta_{\Lambda r}} \brac{\lapms{\mu}(w_0 \varphi)}} =: \laps{\mu}g_1 + \laps{\mu}g_2
\]
and
\[
\begin{ma}
 I &=& \int (\laps{\mu}P)P^T \laps{\mu}g_1 + P \Omega [P^T \laps{\mu}g_1] + \int (\laps{\mu}P)P^T \laps{\mu}g_2 + P \Omega [P^T \laps{\mu}g_2]\\
 &=:& I_1 + I_2.
\end{ma}
\]
For $I_1$ we use Theorem \ref{th:intro:energy} in the form of Lemma~\ref{la:betterOmegaP}, 
\[
 I_1 = \int \Omega_P[\laps{\mu}g_1] \aleq \theta\ r^{m-2\mu}\ \begin{cases}
                                           [g_1]_{\BMO} \quad &\mbox{if $\mu \leq 1$},\\
					    r^{\mu-\frac{m}{2}}\vrac{\laps{\mu} g_1}_{(2,\infty)} \quad &\mbox{if $\mu > 1$}.
                                          \end{cases}
\]
Note that \[
 \supp (\varphi w_0) \subset B_r,
\]
and moreover for $q_\mu = \infty$, for $\kappa > \mu$, and $q_\mu = 1$ for $\kappa = \mu$, (for arbitrary $\tau > 0$)
\begin{equation}\label{eq:highomega:varphiw0crit}
 \vrac{\varphi w_0}_{(\frac{m}{m-\mu},\infty)_{\frac{\mu m}{m-\mu}}} \aleq \vrac{\varphi}_{(\frac{m}{\kappa-\mu},\infty)} \vrac{w \chi_{B_r}}_{\brac{{p_\kappa},\infty}_{{\lambda_\kappa} }} \aleq  \vrac{\laps{\tau} \varphi}_{(\frac{m}{\kappa+\tau-\mu},q_\mu)} \vrac{w}_{\brac{{p_\kappa},\infty}_{{\lambda_\kappa} },B_{2r}}.
\end{equation}
Then, the claim for $I_1$ follows from
\begin{proposition}
Let $\mu \leq 1$, $g := \eta_{\Lambda r} \lapms{\mu}(f)$, $\supp f \subset \overline{B_r}$, then for any $\kappa \in [\mu,2\mu)$,
\[
 [g]_{\BMO} \aleq \brac{1+\Lambda^{\mu-m}}\ \vrac{f}_{(\frac{m}{m-\mu},\infty)_{\frac{\mu m}{m-\mu}}}.
\]
\end{proposition}
\begin{proof}
From \cite[Proposition 3.3.]{Adams75}
\[
[g]_{\BMO} \aleq \vrac{\laps{\mu} g}_{(1)_\mu}.
\]
Since,
\[
 \laps{\mu} g = f + \laps{\mu} \brac{(1-\eta_{\Lambda r}) \lapms{\mu} f},
\]
we have,
\[
[g]_{\BMO} \aleq \vrac{f}_{(1)_\mu} + \vrac{\laps{\mu} \brac{(1-\eta_{\Lambda r}) \lapms{\mu} f}}_{(1)_\mu} \aleq \vrac{f}_{(\frac{m}{m-\mu},\infty)_{\frac{\mu m}{m-\mu}}} + \vrac{\laps{\mu} \brac{(1-\eta_{\Lambda r}) \lapms{\mu} f}}_{(\frac{m}{\mu},\infty)},
\]
and by Proposition \ref{la:quasilocIII} 
\begin{equation}\label{eq:highomega:lapmsmulapmsmufest}
 \vrac{\laps{\mu} \brac{(1-\eta_{\Lambda r}) \lapms{\mu} f}}_{\frac{m}{\mu}}  \aleq  
 \sup_{\alpha \in [0,\mu]} (\Lambda r)^{-m+\mu-\alpha}\ \vrac{f}_{1}\ \vrac{\sabs{\laps{\mu-\alpha}((1-\eta_{\Lambda r}))} }_{\frac{m}{\mu}}
 \aleq (\Lambda r)^{\mu-m}\ \vrac{f}_{1}.
\end{equation}
Since $\supp f \subset B_r$, 
\begin{equation}\label{eq:highomega:f1hoelder}
 r^{\mu-m}\ \vrac{f}_{1} \aleq \vrac{f}_{(1)_\mu} \aleq \vrac{f}_{(\frac{m}{m-\mu},\infty)_{\frac{\mu m}{m-\mu}}}.
\end{equation}
\end{proof}

Moreover, as in \eqref{eq:highomega:lapmsmulapmsmufest}, from Proposition \ref{la:quasilocIII} and \eqref{eq:highomega:varphiw0crit},
\[
 \vrac{\laps{\mu} g_2}_{2}  \aleq (\Lambda r)^{-\frac{m}{2}}\ 
 \vrac{\varphi w_0}_1
 \overset{\eqref{eq:highomega:f1hoelder}}{\aleq} r^{\frac{m}{2}-\mu} \vrac{\laps{\tau} \varphi}_{(\frac{m}{\kappa+\tau-\mu},q_\mu)} \vrac{w}_{\brac{{p_\kappa},\infty}_{{\lambda_\kappa} },B_{2r}},
\]
implying
\[
\begin{ma}
 \abs{I_2} &\aleq& \vrac{\Omega_P}_{2 \to 1}\ (\Lambda r)^{\frac{m}{2}-\mu}\ \vrac{\laps{\tau} \varphi}_{(\frac{m}{\kappa+\tau-\mu},q_\mu)} \vrac{w}_{\brac{{p_\kappa},\infty}_{{\lambda_\kappa},B_{2r} }}\\
 &\overset{\eqref{eq:inttransp:smallOmegaP}}{\aleq}& \theta\ r^{m-2\mu}\ \vrac{\laps{\tau} \varphi}_{(\frac{m}{\kappa+\tau-\mu},q_\mu)} \vrac{w}_{\brac{{p_\kappa},\infty}_{{\lambda_\kappa},B_{2r} }}.
\end{ma}
\]
This proves the claim \eqref{eq:higheromega:proofclaim} and thus Lemma~\ref{la:rhs:OmegaPest}

\newpage
\section{Higher Integrability: Proof of Theorem \ref{th:uptoinfty}}
Let $w \in L^{p,\lambda}_{loc}(D) \cap L^2(\R^m)$ be a solution to
\[
 \laps{\mu}w = \Omega[w] \quad \mbox{in $D \subsubset \R^m$}.
\]
Choosing for any domain $\tilde{D} \subsubset D$, we can choose a domain $D_2$, $\tilde{D} \subsubset D_2 \subsubset D$ and a cutoff function $\eta_{\tilde{D}} \in C_0^\infty(D_2)$, $\eta_{\tilde{D}} \equiv 1$ in $\tilde{D}$. Then $w_{\tilde{D}} := \eta_{\tilde{D}} w \in L^{p,\lambda}(\R^n)$ is a solution to
\[
 \laps{\mu}w_{\tilde{D}} = \Omega[w_{\tilde{D}}] + \Omega[w-w_{\tilde{D}}]+\laps{\mu}(w_{\tilde{D}}-w)\quad \mbox{in $\tilde{D}$},
\]
and in $\tilde{D}$,
\[
 \vrac{\Omega[w-w_{\tilde{D}}]+\laps{\mu}(w_{\tilde{D}}-w)}_{\infty,\tilde{D}} \leq C_{\tilde{D},D,D_2,\eta,\vrac{w}_2}.
\]
So Theorem~\ref{th:uptoinfty} follows from the following argument.

\begin{lemma}\label{la:higherit}
Let $p > 2$, and $0 < \mu \leq \frac{m}{2}$, $\lambda \leq 2\mu$, and let $w \in L^{p,\lambda}$ be a solution to 
\begin{equation}\label{eq:highit:lapsmuw.new}
 \laps{\mu}w = \Omega[w] + f \quad \mbox{in $D \subsubset \R^m$},
\end{equation}
where $f \in L^\infty(D)$. Then, for any $\tilde{p} \in [p,\infty)$ there exists $\varepsilon \in (0,1)$ such that if $\theta$ from \eqref{eq:regp} satisfies $\theta < \varepsilon$, then $w \in L^{\tilde{p}}_{loc}(D)$.
\end{lemma}
\begin{proof}
In order to keep the presentation short, we are going to assume that $\Omega[] = \tilde{\Omega} \Rz[]$. Also note that if $w \in L^{p,\lambda}$ for some $p > 2$, than for some $\tilde{p} \in (2,p)$, $w \in L^{\tilde{p},\tilde{\lambda}}$, for some $\tilde{\lambda} < \lambda$, so we can assume w.l.o.g. that $\lambda < 2\mu$.

From \eqref{eq:highit:lapsmuw.new} we have for any $B_{r} \subset B_R \subset \tilde{D}$,
\[
\begin{ma}
\vrac{\laps{\mu} w}_{\frac{2p}{p+2},B_r} &\aleq& \vrac{\Omega}_{2,B_r} \vrac{\Rz[w]}_{p,B_r} + \vrac{f}_{\infty}\ r^{m\frac{p+2}{2p}}\\
&\overset{\eqref{eq:regp}}{\aleq}& r^{\frac{m-2\mu}{2}}\  \theta\ \vrac{w}_{p,B_{2r}} + r^{\frac{m-2\mu}{2}} \ \theta \sum_{k=2}^\infty 2^{-k\frac{m}{p}} \vrac{w}_{p,A^k_r} + \vrac{f}_{\infty}\ r^{m\frac{p+2}{2p}}\\
&\aleq& r^{\frac{m-2\mu}{2}}\ r^{\frac{m-\lambda}{p}}\ \theta\ \vrac{w}_{(p)_\lambda,B_{R}} + r^{\frac{m-2\mu}{2}}\ r^{\frac{m-\lambda}{p}} \theta \sum_{k=2}^\infty 2^{-k\frac{\lambda}{p}} [w]_{(p)_\lambda,B_{2^{k+1} R}} + \vrac{f}_{\infty}\ r^{m\frac{p+2}{2p}}\\
\end{ma}
\]
That is, for 
\begin{equation}\label{eq:higherit:lambdaNfromplambda}
 \lambda_N := \lambda \frac{2}{2+p} + 2\mu \frac{p}{2+p} \in (\lambda, 2\mu),
\end{equation}
\[
\begin{ma}
\vrac{\laps{\mu} w}_{(\frac{2p}{p+2})_{\lambda_N},B_{\frac{R}{2}}}
&\aleq& \theta\ \vrac{w}_{(p)_\lambda,B_{R}} + \theta \sum_{k=2}^\infty 2^{-k\frac{\lambda}{p}} [w]_{(p)_\lambda,B_{2^{k+1} R}} + \vrac{f}_{\infty}\ R^{\lambda_N\frac{p+2}{2p}}\\
\end{ma}
\]
Consequently, by Proposition \ref{pr:highit:fmorrey} (note that $\frac{2p}{p+2} > 1$), for $p_2 = 2p/(p+2)$ and $p_1 > p$ (since $\lambda_N < 2\mu$) defined by
\begin{equation}\label{eq:higherit:p1fromplambda}
\fracm{p_1} = \fracm{p} + \fracm{2} - \frac{\mu}{\lambda_N}
\end{equation}
\[
\begin{ma}
\vrac{w}_{p_1,B_{\Lambda^{-1}r}} 
&\aleq& (\Lambda^{-1} r)^{-\frac{\lambda_N}{p}-\frac{\lambda_N}{2}+\mu+\frac{m}{p_1}} \vrac{\laps{\mu} w}_{(\frac{2p}{p+2})_{\lambda_N},B_{r}} + \Lambda^{-\frac{m}{p_1}}\ \sum_{k=1}^\infty 2^{-km}\  r^{\frac{m}{p_1}-m}\ \vrac{w}_{1,A^k_r}\\

&\aleq& (\Lambda^{-1} r)^{-\frac{\lambda_N}{p}-\frac{\lambda_N}{2}+\mu+\frac{m}{p_1}} 
\theta\ \vrac{w}_{(p)_\lambda,B_{2r}} + (\Lambda^{-1} r)^{-\frac{\lambda_N}{p}-\frac{\lambda_N}{2}+\mu+\frac{m}{p_1}}\ \theta \sum_{k=2}^\infty 2^{-k\frac{\lambda}{p}} [w]_{(p)_\lambda,B_{2^{k+2} r}}\\
&&+ (\Lambda^{-1} r)^{-\frac{\lambda_N}{p}-\frac{\lambda_N}{2}+\mu+\frac{m}{p_1}} \vrac{f}_{\infty}\ r^{\lambda_N\frac{p+2}{2p}}\\
&&+ (\Lambda^{-1} r)^{\frac{m}{p_1}-\frac{\lambda}{p}}\  \Lambda^{-\frac{\lambda}{p}} \sum_{k=1}^\infty 2^{-k\frac{\lambda}{p}}\  [w]_{(p)_\lambda,B_{2^{k+1}r}}\\
&\aleq& (\Lambda^{-1} r)^{-\frac{\lambda_N}{p}-\frac{\lambda_N}{2}+\mu+\frac{m}{p_1}} 
\theta\ \vrac{w}_{(p)_\lambda,B_{2r}}\\
&&+ (\Lambda^{-1} r)^{-\frac{\lambda_N}{p}-\frac{\lambda_N}{2}+\mu+\frac{m}{p_1}}\ (\theta + \Lambda^{-\frac{\lambda}{p}} )\ \sum_{k=1}^\infty 2^{-k\frac{\lambda}{p}} [w]_{(p)_\lambda,B_{2^{k+1} r}}\\
&&+ (\Lambda^{-1} r)^{-\frac{\lambda_N}{p}-\frac{\lambda_N}{2}+\mu+\frac{m}{p_1}} \vrac{f}_{\infty}\ r^{\lambda_N\frac{p+2}{2p}}\\
\end{ma}
\]
Consequently,
\[
\begin{ma}
\vrac{w}_{p,B_{\Lambda^{-1}r}} 
&\aleq& (\Lambda^{-1} r)^{\frac{m}{p}-\frac{m}{p_1}} \vrac{w}_{p_1,B_{\Lambda^{-1}r}} \\
&\aleq& (\Lambda^{-1} r)^{\frac{m}{p}-\frac{\lambda_N}{p}-\frac{\lambda_N}{2}+\mu}
\theta\ \vrac{w}_{(p)_\lambda,B_{2r}}\\
&&+ (\Lambda^{-1} r)^{\frac{m}{p}-\frac{\lambda_N}{p}-\frac{\lambda_N}{2}+\mu}\ (\theta + \Lambda^{-\frac{\lambda}{p}} )\ \sum_{k=1}^\infty 2^{-k\frac{\lambda}{p}} [w]_{(p)_\lambda,B_{2^{k+1} r}}\\
&&+ (\Lambda^{-1} r)^{\frac{m}{p}-\frac{\lambda_N}{p}-\frac{\lambda_N}{2}+\mu} \vrac{f}_{\infty}\ r^{\lambda_N\frac{p+2}{2p}}.
%
\end{ma}
\]
Now
\[
\frac{m}{p}-\frac{\lambda_N}{p}-\frac{\lambda_N}{2}+\mu = \frac{m-\lambda}{p},
\]
which implies finally, for any $B_{2r} \subset D$, 
\[
\begin{ma}
\vrac{w}_{(p)_\lambda,B_{\Lambda^{-1}r}}
&\aleq& \theta\ \vrac{w}_{(p)_\lambda,B_{2r}}\\
&&+  (\theta + \Lambda^{-\frac{\lambda}{p}} )\ \sum_{k=1}^\infty 2^{-k\frac{\lambda}{p}} [w]_{(p)_\lambda,B_{2^{k+1} r}}\\
&&+ \vrac{f}_{\infty}\ r^{\lambda_N\frac{p+2}{2p}}.
\end{ma}
\]
Now we argue similar to the iteration in Section~\ref{s:upto2}: Choose $\Lambda_\lambda := 2^{C_{p,\mu} \lambda^{-4}}$, assume that $\theta < \Lambda_\lambda^{-\frac{\lambda}{p}}$, and choose $C_{p,\mu}$ so that \eqref{eq:it:smallness} is satisfied. Then we can choose a new $\lambda_1 = \lambda-c\lambda^{4}$ for which the above estimate holds and the right-hand side is finite. Repeating this argument (for smaller and smaller $\theta$), we obtain a monotone decreasing sequence of $\lambda_{i+1} = \lambda_i - c \lambda_i^4 \geq 0$, which has as only fixed point $0$. Thus, for any $\lambda > 0$ there exists $\theta > 0$ such that for any $\tilde{D} \subsubset D$,
\[
 \vrac{w}_{(p)_\lambda,\tilde{D}} \leq C_{\tilde{D},D,\lambda, w}.
\]
Note that for $\lambda \to 0$, $\lambda_N \to \mu \frac{2p}{p+2}$ and thus $p_1$ in \eqref{eq:higherit:p1fromplambda} tends to infinity. Thus, we have obtain for any $\tilde{p} > 1$ a $\lambda_{\tilde{p}} > 0$ such that $p_1 \equiv p_1(\lambda_{\tilde{p}}) > \tilde{p}$, and if $\theta$ is small enough, we have to iterate the above argument finitely many steps to obtain that $w \in L^{p_1}_{loc}(\tilde{D})$.
%
\end{proof}

\begin{proposition}\label{pr:highit:fmorrey}
For any $f$, $\mu \in (0,m)$ we have for $p_1 \in (1,\infty)$, $p_2 \in (1,\infty)$, $\lambda \in (0,m)$ such that
\[
 \fracm{p_1} = \fracm{p_2} - \frac{\mu}{\lambda},
\]
the following estimate for any $\Lambda > 2$
\[
 \vrac{f}_{p_1,B_{\Lambda^{-1}r}} \aleq (\Lambda^{-1} r)^{-\frac{\lambda}{p_2}+\mu+\frac{m}{p_1}} \vrac{\laps{\mu} f}_{(p_2)_\lambda, B_r} + \sum_{k=1}^\infty 2^{-km}\  \Lambda^{-\frac{m}{p_1}}\ r^{\frac{m}{p_1}-m}\ \vrac{f}_{1,A^k_r}
\]
\end{proposition}
\begin{proof}
Let $1 < p_4 \leq p_1'$,
\[
 \fracm{p_3} + \fracm{p_4} = 1.
\]
\[
 \fracm{p_3} = \fracm{p_2} - \frac{\mu}{\lambda} \in (0,1).
\]
There exists $\varphi \in C_0^\infty(B_{\Lambda^{-1}r})$, $\vrac{\varphi}_{p_1'} \leq 1$, such that
\[
 \begin{ma}
  \vrac{f}_{p_1,B_{\Lambda^{-1}r}}  &\aleq& \int f \varphi = \int \lapms{\mu} \brac{\eta_{B_r} \laps{\mu} f}\ \varphi + \sum_{k=1}^\infty \int f\ \laps{\mu} \brac{\eta_{A_r^k}\lapms{\mu} \varphi}\\
  &\aleq& \vrac{\lapms{\mu} (\eta_r \laps{\mu} f)}_{p_3,B_{\Lambda^{-1} r}}\ \vrac{\varphi}_{p_4} + \sum_{k=1}^\infty \vrac{f}_{p_1,A^k_r}\  \vrac{\laps{\mu}\brac{\eta_{A^k_r} \lapms{\mu} \varphi}}_{p_1'} 
 \end{ma}
\]
By Lemma~\ref{la:quasilocIII},
\[
\vrac{\laps{\mu}\brac{\eta_{A^k_r} \lapms{\mu} \varphi}}_{\infty} \aleq 2^{-km}\  \Lambda^{-\frac{m}{p_1}}\ r^{\frac{m}{p_1}-m}
\]
Since $p_4 \leq p_1'$,
\[
 \vrac{\varphi}_{p_4,B_r} \aleq (\Lambda^{-1} r)^{\frac{m}{p_4}-\frac{m}{p_1'}}.
\]
And using Lemma~\ref{la:adams}
\[
 \vrac{\lapms{\mu} (\eta_r \laps{\mu} f)}_{p_3,B_{\Lambda^{-1} r}} \aleq (\Lambda^{-1} r)^{\frac{m-\lambda}{p_3}} \vrac{\lapms{\mu} (\eta_r \laps{\mu} f)}_{(p_3)_\lambda} \aleq (\Lambda^{-1} r)^{\frac{m-\lambda}{p_3}} \vrac{\laps{\mu} f}_{(p_2)_\lambda,B_r}
\]
Consequently, we have shown the claim.

%
%
\end{proof}

\newpage
\section{Commutators and fractional product rules: Proof of Theorem~\ref{th:commutators}}
In this section we repeat and refine some commutator estimates and non-local expansion rules which were introduced in \cite{Sfracenergy}, motivated by the results in \cite{DR1dSphere,SNHarmS10}. The for us most important commutators are
\[
 H_\alpha (a,b) :=\laps{\alpha} (ab) - a \laps{\alpha} b - - b \laps{\alpha} a,
\]
and for a linear operator $T$
\[
\mathcal{C}(a,T)[b] := a T[b] - T[ab].
\]
The commutator $H_\alpha (a,b)$ was introduced by Da Lio and Rivi\`{e}re in \cite{DR1dSphere}, where Hardy-space $\mathcal{H}$ and $BMO$-estimates where shown, making use of the Hardy-Littlewood decomposition and paraproducts. This is also somewhat related to the techniques of the T1-Theorem cf. \cite{KP88}. 
If one is interested in $L^2$-estimates only (e.g., in the sphere case) then there is an extremely elementary argument \cite{SNHarmS10} somewhat inspired by Tartar's proof of Wente's inequality \cite{Tartar85}. 
For general Lorentz space estimates there is also an argument using potential arguments, which even gives pointwise estimates, and was introduced in \cite{Sfracenergy}. As it is a direct, pointwise argument not involving the Fourier transform, it is easier to apply in non-linear situations, cf. \cite{BPSknot12}.

The commutator $\mathcal{C}(a,T)[b]$ and its Hardy-BMO estimates were introduced in \cite{CRW76} for the Riesz transform $\Rz$, and later generalized to the Riesz potential $\lapms{\alpha}$ in \cite{Chanillo82}. Again for pointwise estimates the arguments in \cite{SNHarmS10} can be adapted.

Here, we are going to treat in Subsection~\ref{ss:pointwisefracprodrule} pointwise estimates on $H_\alpha(a,b)$, and in Subsection~\ref{ss:pointwisecommest} pointwise estimates on $\mathcal{C}(a,T)[b]$ using and extending the techniques from \cite{SNHarmS10}. For Hardy-BMO estimates, we will use in Subsection~\ref{ss:drprodrules} the techniques in \cite{DR1dSphere}, and extend them to the limiting case $\alpha = 1$.

Let us shortly recall the notion for Hardy space $\mathcal{H}$ and $BMO$. The latter space $BMO$ is defined as
\[
 g \in BMO \quad :\Leftrightarrow [g]_{BMO} := \abs{B_r}^{-1}\ \int_{B_r} \abs{g-\abs{B_r}^{-1}\int_{B_r} g} < \infty.
\]
Our interest in $BMO$ stems from the fact, that it is a bigger space than $L^\infty$, and we have the nice embedding
\begin{equation}\label{eq:commis:BMOembedding}
 [g]_{BMO} \aleq \sup_{r > 0}  r^{\frac{p\tau-m}{p}} \vrac{\laps{\tau} g}_{(p,\infty), B_r} \quad \mbox{for $\tau > 0$, $p > 1$},
\end{equation}
wheras for $L^\infty$ we only have the following embedding which is more difficult to control
\begin{equation}\label{eq:commis:Linftyembedding}
 \vrac{g}_{\infty} \aleq \vrac{\laps{\tau} g}_{(\frac{m}{\tau},1)} \quad \mbox{for $\tau \in (0,m)$}.
\end{equation}
The Hardy space $\mathcal{H}$, on the other hand, is a slightly smaller space than $L^1$, with the (for us) most important property that
\begin{equation}\label{eq:hardybmo}
 \int f\ g \aleq \vrac{f}_{\mathcal{H}}\ [g]_{BMO}.
\end{equation}
That is, if we know that a quantity belongs to the Hardy space, it allows us to control the integral of \eqref{eq:hardybmo} in terms of the right-hand side of \eqref{eq:commis:BMOembedding}, instead of having to deal with the terms on the right-hand side of \eqref{eq:commis:Linftyembedding}.

The norm of the Hardy space $\mathcal{H}$ is usually defined via
\[
 \vrac{f}_{\mathcal{H}} := \|\sup_{t > 0} \phi_t \ast f \|_1,
\]
where $\phi \in C_0^\infty(B_1)$, $\int \phi = 1$, and $\phi_t(x) := t^{-m}\phi(x/t)$, cf. \cite{Stein93, FS72}, another very readable overview in the context with Partial Differential Equations is given in \cite{Semmes94}. We are never using the above definition, though, but rather use a characterization in terms of Triebel-spaces, and employ the duality \eqref{eq:hardybmo}.

\subsection{Pointwise fractional product rules via potentials}\label{ss:pointwisefracprodrule}
\begin{lemma}\label{la:lapsbetaHalpha}
For any $\alpha \in (0,m)$ there is $L \in \N$ such that the following holds: For any $\beta \in [0,\min (\alpha,1))$, $\beta \leq m-\alpha$, $\tau \in (\max\{\beta,\alpha+\beta-1\},\alpha]$, $\epsilon > 0$, there are, $s_k \in (0,\alpha)$, $t_k \in (0,\tau)$, where $\tau-\beta -s_k-t_k \in [0,\epsilon)$, such that the following holds
\[
 \abs{\laps{\beta} H_\alpha(a,b)} \aleq \sum_{k=1}^L \lapms{\tau-\beta  -s_k-t_k} \brac{\lapms{s_k} \abs{\laps{\alpha} a}\ \lapms{t_k}\abs{\laps{\tau} b}}.
\]
\end{lemma}
\begin{proof}
 Notice that is suffices to prove the for $\tau < \alpha$, given that
\begin{equation}
\label{eq:taumax}
 \abs{\laps{\tau} b} = \abs{\lapms{\alpha-\tau} \laps{\alpha} b} \aleq \lapms{\alpha-\tau}\abs{\laps{\alpha} b}.
\end{equation}
%
\[
 \begin{ma}
 \laps{\beta} H_{\alpha}(a,b) &=& \laps{\beta} \brac{\lapa(ab) - a \lapa b - b \lapa a}\\
&=& \laps{\alpha+\beta} (ab) - a\ \laps{\alpha+\beta} b - b\ \laps{\alpha+\beta} a \\
&&+  a\laps{\beta} \lapa b - \laps{\beta}\brac{a \lapa b}\\
&&+ b\laps{\beta} \lapa a - \laps{\beta}\brac{b \lapa a}\\
&=& H_{\alpha+\beta}(a,b) +  \mathcal{C}(a,\laps{\beta})[\laps{\alpha}b] + \mathcal{C}(b,\laps{\beta})[\laps{\alpha}a] 

 \end{ma}
\]
The claim for the term $H_{\alpha+\beta}(a,b)$ then comes from Lemma \ref{la:comm:Hsest}.\\
For the remaining terms, we apply Lemma~\ref{la:comm:Clapmslapslapsd}: For $\mathcal{C}(a,\laps{\beta})[\laps{\alpha}b]$, we take $\delta \hat{=} \alpha-\tau$, $\tau \hat{=} \alpha$, $\beta \hat{=} \beta$ and set $A := \laps{\alpha} a$, $B := \laps{\tau} b$. Then
\[
 \mathcal{C}(a,\laps{\beta})[\laps{\alpha}b] = \mathcal{C}(\lapms{\alpha} A,\laps{\beta})[\laps{\alpha-\tau} B].
\]
For $\mathcal{C}(b,\laps{\beta})[\laps{\alpha}a]$, set for very small $\delta < \min\{\tau-\beta,1-\beta\}$, $A := \laps{\alpha-\delta} a$, $B := \laps{\tau} b$. Then 
\[
 \mathcal{C}(b,\laps{\beta})[\laps{\alpha}a] = \mathcal{C}(\lapms{\tau} B,\laps{\beta})[\laps{\delta} A].
\]
\end{proof}

\begin{lemma}\label{la:comm:Hsest}
Let $\alpha \in (0,m)$, $\epsilon > 0$ and assume that $\tau_1, \tau_2 (\max\{\alpha-1,0\},\alpha]$, $\tau_1+ \tau_2 > \alpha$. Then for some $L\in \N$, there are $s_k \in (0,\tau_1)$, $t_k \in (0,{\tau_2})$, $\tau_1+ \tau_2-s_k- t_k-\alpha \in [0,\epsilon)$ such that
\begin{equation}\label{eq:comm:Hsestclaim}
 \abs{H_\alpha(a,b)} \aleq \sum_{k=1}^L \lapms{\tau_1+ \tau_2-s_k- t_k-\alpha} \brac{\lapms{s_k} \abs{\laps{\tau_1} a}\ \lapms{t_k}\abs{\laps{{\tau_2}} b}}.
\end{equation}
\end{lemma}
For the convenience of the reader, we give the proof, which essentially follows the argument in \cite{Sfracenergy}; For a presentation closer to this one see \cite{DLSsphere}.
\begin{proof}
For $\alpha \in 2\N$ it is easy to obtain \eqref{eq:comm:Hsestclaim}, since for any $l \leq 2K-1$, for ${\tau} > 2K-1$,
\begin{equation}\label{eq:commi:nablakfvslapms}
 \abs{\nabla^l f} = \abs{\nabla^l\lapms{{\tau}} \laps{{\tau}}f} \aleq \lapms{{\tau_2}-l}\abs{\laps{{\tau}}f}.
\end{equation}
So we can assume that $\alpha = 2K+s$, for some $s \in (0,2)$, $K \in \N\cup\{0\}$. Assume at first that $K = 0$. Set
$A := \sabs{\laps{\tau_1} a}$, $B := \sabs{\laps{{\tau_2}} b}$, we have 
\[
 \abs{H_\alpha(a,b)} \aleq \lim_{\varepsilon \to 0} \int_{\abs{x-y} > \varepsilon} \int\int k(x,y,z_1,z_2)\ A(z_1)\ B(z_2)\ dz_2\ dz_1\ dy,
\]
where
\[
 k(x,y,z_1,z_2) = \frac{m_{\tau_1}(x,y,z_1)\ m_{{\tau_2}}(x,y,z_2)}{\abs{x-y}^{m+\alpha}},
\]
and for $s > 0$,
\[
 m_s(x,y,z) = \abs{\abs{x-z}^{-m+s} - \abs{z-y}^{-m+s}}.
\]
Let moreover
\[
 1 \leq \chi_1(x,y,z) + \chi_2(x,y,z) + \chi_3(x,y,z) \quad \mbox{for
$x,y,z \in \R^m$},
\]
where
\[
 \chi_1 := \chi_{\abs{x-y} \leq  2\abs{z-y}}\ \chi_{\abs{x-y} \leq 2
\abs{x-z}},
\]
\[
 \chi_2 := \chi_{\abs{x-y} \leq  2\abs{z-y}}\ \chi_{\abs{x-y} > 2
\abs{x-z}},
\]
\[
 \chi_3 := \chi_{\abs{x-y} >  2\abs{z-y}}\ \chi_{\abs{x-y} \leq 2
\abs{x-z}}.
\]
One then checks, using for $m_s(x,y,z)\chi_1$ a one-step Taylor expansion, for any $\delta \in (0,\min(s,1))$
\[
 m_s(x,y,z) \chi_1 \aleq \abs{x-z}^{-n+\alpha -\delta} \abs{x-y}^\delta\ \chi_1 \approx
\abs{z-y}^{-n+\alpha -\delta} \abs{x-y}^\delta\ \chi_1.
\]
\[
 m_s(x,y,z) \chi_2 \aleq \abs{x-z}^{-n+\alpha} \chi_2 \aleq
\abs{x-z}^{-n+\alpha-\delta}\ \abs{x-y}^{\delta},
\]
\[
 m_s(x,y,z) \chi_3 \aleq \abs{z-y}^{-n+\alpha} \chi_3 \aleq
\abs{z-y}^{-n+\alpha-\delta}\ \abs{x-y}^{\delta}.
\]
Hence, for $\delta_1 \in (0,\min(\tau_1,1))$, $\delta_2 \in (0,\min({\tau_2},1))$
\[
\begin{ma}
 k(x,y,z_1,z_2) \aleq \abs{x-y}^{-m-\alpha+\delta_1 + \delta_2}\ \big (&&  \abs{x-z_1}^{-m+{\tau_1}-\delta_1}\ \abs{y-z_2}^{-m+{\tau_2}-\delta_2}\\
 &&+ \abs{y-z_1}^{-m+{\tau_1}-\delta_1}\ \abs{x-z_2}^{-m+{\tau_2}-\delta_2} \\
 &&+ \abs{y-z_1}^{-m+{\tau_1}-\delta_1}\ \abs{y-z_2}^{-m+{\tau_2}-\delta_2} \\
 && + \abs{x-y}^{-\delta_1-\delta_2} \abs{x-z_1}^{-m+{\tau_1}}\ \abs{x-z_2}^{-m+{\tau_2}} \chi_{\abs{x-y} > 2\max\{\abs{x-z_1},\abs{x-z_2}\}}\
  \big ).
\end{ma}
\]
We choose $\delta_1 \in (0,\min(\tau_1,1))$, $\delta_2 \in (0,\min(\tau_2,1))$ such that $\delta_1+\delta_2 -\alpha \in (0,\epsilon)$. This is possible, since $\tau_1 + \tau_2 > \alpha$: If $\alpha \in (0,1]$, so are $\tau_1,\tau_2$, and we can choose $\delta_1,\delta_2$ arbitrarily close to $\tau_1,\tau_2$, so that this inequality is satisfied. If $\alpha \in [1,2)$ and (say) $\tau_1 > 1$, choose $\delta_1$ close enough to $1$, and $\delta_2 \in (\alpha-1,{\tau_2})$. Using that
\[
 \int_{\abs{x-y} > \varepsilon} \abs{x-y}^{-m-\alpha} \chi_{\abs{x-y} > 2\max\{\abs{x-z_1},\abs{x-z_2}\}} dy \aleq \max\{\abs{x-z_1},\abs{x-z_2}\}^{-\alpha} \aleq \abs{x-z_1}^{-\delta_1}\ \abs{x-z_2}^{-\alpha+\delta_1},
\]
and consequently, 
\[
\begin{ma}
H_\alpha(a,b) &\aleq& \lapms{\alpha-\delta_1}A \  (\lapms{\delta_1+\delta_2-\alpha}\lapms{{\tau_2}-\delta_2}B)\\
&&+ (\lapms{\delta_1+\delta_2-\alpha}\lapms{\alpha-\delta_1}A) \  \lapms{{\tau_2}-\delta_2}B\\
&&+ \lapms{\delta_1+\delta_2-\alpha}(\lapms{\alpha-\delta_1}A \  \lapms{{\tau_2}-\delta_2}B)\\
&&+ \lapms{\alpha-\delta_1}A\ \lapms{{\tau_2}+\delta_1-\alpha} B\\
&\aeq& \lapms{\alpha-\delta_1}A \  \lapms{\delta_1-\alpha+{\tau_2}}B\\
&&+ \lapms{\delta_1+\delta_2-\alpha}(\lapms{\alpha-\delta_1}A \  \lapms{{\tau_2}-\delta_2}B)\\
&&+ \lapms{\delta_2}A \  \lapms{{\tau_2}-\delta_2}B.
\end{ma}
\]
This shows \eqref{eq:comm:Hsestclaim} for $\alpha \in (0,2)$.\\
\\

If $K \geq 1$, $s \in (0,2)$, $\alpha = 2K+s > 2$
\begin{align}
 H_{2K+s}(a,b) &=  \lap^{K}\laps{s}(ab) - a\ \lap^{K}\laps{s}b- b\ \lap^{K}\laps{s}a \nonumber\\
 &= \laps{s}(\lap^{K}(ab)-a\lap^{K}b-b\lap^{K}a) \label{eq:commi:H2KpslapsHK}\\
  & \quad +\laps{s}(b\lap^{K}a) - b\ \lap^{K}\laps{s}a\label{eq:commi:H2Kpslapsblapka}\\
  & \quad + \laps{s}(a\lap^{K}b)- a\ \lap^{K}\laps{s}b\label{eq:commi:H2Kpslapsalapkb}.
\end{align}
Let $\nabla^l$, $\nabla^{2K-l}$, $l \in {1,\ldots,2K-1}$, be arbitrary combinations of gradients which sum up to differential order of $l$ and $2K-l$, respectively. Then
\[
\begin{ma}
 \laps{s} (\nabla^l a\ \nabla^{2K-l} b) &=& H_s(\nabla^l a, \nabla^{2K-l}b) + \laps{s}\nabla^l a\ \nabla^{2K-l}b+ \nabla^l a\ \laps{s}\nabla^{2K-l}b.
\end{ma}
\]
Recall $\alpha = 2K+s$. Applying \eqref{eq:comm:Hsestclaim} to $H_s(\cdot,\cdot)$, noting that $\alpha -s-l = 2K-l > 0$ and ${\tau_2}-s-2K+l > l-1 \geq 0$
\[
  \begin{ma}
  \laps{s} (\nabla^k a\ \nabla^{2K-k} b) 
  &\aleq& \sum_{k=1}^L \lapms{s-s_k-t_k} \brac{\lapms{s_k} \abs{\laps{s} \nabla^l a}\ \lapms{t_k}\abs{\laps{s} \nabla^{2K-l} b}}\\
  && + \abs{\laps{s}\nabla^l a}\ \abs{\nabla^{2K-l}b}+ \abs{\nabla^l a}\ \abs{\laps{s}\nabla^{2K-l}b}\\
  &\overset{\eqref{eq:commi:nablakfvslapms}}{\aleq}& \sum_{k=1}^L \lapms{s-s_k-t_k} \brac{\lapms{s_k+\tau_1-s-l} \abs{ \laps{\tau_1}a}\ \lapms{t_k-\alpha+l+{\tau_2}}\abs{\laps{{\tau_2}} b}}\\
  && + \lapms{{\tau_1}-s-l}\abs{\laps{{\tau_1}}a}\ \lapms{{\tau_2}-2K+l} \abs{\laps{{\tau_2}}b}+ \lapms{{\tau_1}-l}\abs{\laps{{\tau_1}}a}\ \lapms{{\tau_2}-\alpha+l} \abs{\laps{{\tau_2}}b}.
  \end{ma}
\]
Thus, we can estimate \eqref{eq:commi:H2KpslapsHK}
\[
  \begin{ma}
 \laps{s}(\lap^{K}(ab)-a\lap^{K}b-b\lap^{K}a) 
 &\aleq& \sum_{l=1}^{2K-1}\sum_{k=1}^L \lapms{s-s_k-t_k} \brac{\lapms{s_k+\tau_1-s-l} \abs{ \laps{\tau_1}a}\ \lapms{t_k-\alpha+l+{\tau_2}}\abs{\laps{{\tau_2}} b}}\\
  && + \sum_{l=1}^{2K-1}\lapms{\tau_1-s-l}\abs{\laps{\tau_1}a}\ \lapms{{\tau_2}-2K+l} \abs{\laps{{\tau_2}}b}+ \sum_{l=1}^{2K-1}\lapms{\tau_1-l}\abs{\laps{\tau_1}a}\ \lapms{{\tau_2}-\alpha+l} \abs{\laps{{\tau_2}}b}.
  \end{ma}
\]
As for \eqref{eq:commi:H2Kpslapsblapka},
\[
\begin{ma}
  &&\laps{s}(b\lap^{K}a)- b\ \lap^{K}\laps{s}a = H_s(\lap^K a,b) + \lap^{K} a\ \laps{s}b.
\end{ma}
\]
We have $\tau_1, \tau_2 > \alpha -1 = 2K+s-1$, and consequently ($K \geq 1$) we know $\tau_2 > s$. Assume that moreover $\tau_1 > 2K$, then
\[
 \abs{\laps{2K} a} \aleq \lapms{{\tau_2}-2K} \abs{\laps{{\tau_1}}a},\quad \abs{\laps{s} b} \aleq \lapms{{\tau_2}-s} \abs{\laps{{\tau_2}}b},
\]
and applying our estimates on $H_s(\cdot,\cdot)$ for $\tilde{\tau}_1 = \tau_1 - 2K > s-1$ (since $\tau_1 > \alpha-1$) and $\tilde{\tau}_2 = s$ we have the claim.\\
If this is not the case and $\tau_1 \leq 2K$, then $s < 1$. Then we need to apply Lemma \ref{la:comm:Clapmslapslapsd}: 
\[
\begin{ma}
  \laps{s}(b\lap^{K}a)- b\ \lap^{K}\laps{s}a
  &=&\laps{s}(b\laps{\delta}\lapms{{\tau_1}-2K+\delta}\laps{{\tau_1}}a)- b\ \laps{s}\laps{\delta}\lapms{{\tau_1}-2K+\delta}\laps{{\tau_1}}a\\
  &=& \mathcal{C}(\lapms{\tau_2}B,\laps{s})[\laps{\delta}A],
\end{ma}
\]
for $B := \laps{\tau_2} b$, $A := \lapms{{\tau_1}-2K+\delta}\laps{{\tau_1}}a$, for some $\delta \in (\max\{2K-{\tau_2},0\},1-s)$. Then Lemma~\ref{la:comm:Clapmslapslapsd} is applicable, and we have the claim.\\
We apply the same argument for \eqref{eq:commi:H2Kpslapsalapkb}. 
%
%
%
%
This concludes the proof of Lemma \ref{la:comm:Hsest}.
\end{proof}

\begin{proposition}\label{pr:HestRm}
Let $f,g \in \Sw(\R^m)$, Then
\[
 \vrac{\laps{\beta}H_{\mu}(f,g)}_{(p_0,1),\R^m} \aleq 
 \vrac{\laps{\tau} f}_{\frac{m}{\kappa+\tau-\mu},2}\ \vrac{\laps{\mu} g}_{2},
\]
where $\tau$ is chosen as in Lemma \ref{la:lapsbetaHalpha}
\[
 \fracm{p_0} = \fracm{2}+ \frac{\kappa +\beta-\mu}{m}.
\]
\end{proposition}

\begin{proposition}\label{pr:HestArk}
Let $f,g \in \Sw(\R^m)$, $\supp f \subset \overline{B_r}$. Then for any $k \geq 2$,
\[
 \vrac{\laps{\beta}H_{\mu}(f,g)}_{(p_0,1),A_{r}^k} \aleq 
 2^{k(-\frac{m}{2}+ \kappa -\mu)}\ \vrac{\laps{\tau} f}_{\frac{m}{\kappa+\tau-\mu}}\ \vrac{\laps{\mu} g}_{2},
\]
where
\[
 \fracm{p_0} = \fracm{2}+ \frac{\kappa +\beta-\mu}{m}.
\]
\end{proposition}
\begin{proof}
Pick $\psi \in C_0^\infty(A_{r}^k)$, $\vrac{\psi}_{(p_0',\infty)} \leq 1$, such that
\[
\begin{ma}
 &&\vrac{\laps{\beta}H_{\mu}(g,f)}_{(p_0,1),A_{4r}^k}\\
 &\aleq& \int \psi\ \laps{\beta}H_{\mu}(g,f)\\
  &=& \int \laps{\mu + \beta}\psi\ g f + \int \laps{\beta}\psi\ \laps{\mu}g\ f + \int \laps{\beta}\psi\ \laps{\mu} f g\\
  &\aleq& \vrac{\laps{\mu + \beta}\psi}_{\infty,B_r}\ \vrac{g f}_{1} + 
  \vrac{f}_{2} \vrac{\laps{\beta}\psi}_{\infty}\ \vrac{\laps{\mu}g}_2
  + \vrac{\laps{\mu-\tau} (g \laps{\beta}\psi)}_{\frac{m}{m-\kappa - \tau +\mu}}\ \vrac{\laps{\tau}f}_{(\frac{m}{\kappa + \tau -\mu},\infty)}\\
  
  &\aleq& \vrac{\laps{\mu + \beta}\psi}_{\infty,B_r}\ \vrac{\laps{\mu}g}_{2} \vrac{\laps{\tau}f}_{(\frac{m}{\kappa + \tau -\mu})} r^{\frac{m}{2}-\kappa+2\mu}\\
  &&+   \vrac{\laps{\beta}\psi}_{\infty,B_r}\ \vrac{\laps{\mu}g}_{2} \vrac{\laps{\tau}f}_{(\frac{m}{\kappa + \tau -\mu})} r^{\frac{m}{2}-\kappa+\mu}\\
  &&+ \vrac{\laps{\mu-\tau} (g \laps{\beta}\psi)}_{\frac{m}{m-\kappa - \tau +\mu}}\ \vrac{\laps{\tau}f}_{(\frac{m}{\kappa + \tau -\mu})}
  \end{ma}
\]
Lemma \ref{la:QuasiLocality} gives that
\[
 \vrac{\laps{\mu + \beta}\psi}_{\infty,B_r} \aleq (2^k r)^{-m-\mu-\beta+\frac{m}{p_0}}\ \vrac{\psi}_{p_0'} \aleq (2^k r)^{\kappa-2\mu-\frac{m}{2}},
\]
\[
 \vrac{\laps{\beta}\psi}_{\infty,B_r} \aleq (2^k r)^{\kappa-\mu-\frac{m}{2}}.
\]
Next, according to Lemma \ref{la:quasilocIII}
\[
\begin{ma}
&&\vrac{\laps{\mu-\tau} (g \laps{\beta}\psi)}_{\frac{m}{m-\kappa - \tau +\mu}} \\
&\aleq& \sup_{\alpha \in [0,\mu-\tau]}(2^l r)^{-m-\beta-\alpha} r^\alpha \vrac{\psi}_1\ \vrac{\laps{\mu-\tau}g}_{\frac{m}{m-\kappa - \tau +\mu}}\\
&\aleq& \sup_{\alpha \in [0,\mu-\tau]}(2^l r)^{-m-\beta-\alpha} r^\alpha (2^l r)^{\frac{m}{2}+ \kappa +\beta-\mu} \vrac{\psi}_{p_0',\infty}\ \vrac{\laps{\mu} g}_{2} r^{-\kappa +\mu+\frac{m}{2}}\\
&=& 2^{l(-\frac{m}{2}+ \kappa -\mu)} \vrac{\laps{\mu}g}_{2}.
\end{ma}
\]


\end{proof}

\subsection{Pointwise commutator estimates via potentials}\label{ss:pointwisecommest}
In this section, we discuss for commutators of which special cases have been appearing in \cite{CRW76,Chanillo82}. There, usually estimates in the Hardy-space and BMO were proven. In contrast, we are going to prove pointwise estimates adapting our arguments from \cite{Sfracenergy}, which might be of independent interest.
\begin{lemma}\label{la:comm:Clapmslapslapsd}
Let $\beta + \delta < \min(\tau,1)$, $\delta > 0$, $\epsilon > 0$. There exists a finite number $L$, and $s_k,\tilde{s}_k >0$, $t_k, \tilde{t}_k \in (0,\tau)$, $\tilde{s}_k+\tilde{t}_k=s_k+t_k = \tau-\beta-\delta$, $\tilde{s}_k < \epsilon$,
\begin{align}
&\mathcal{C}(\lapms{\tau} A,\laps{\beta})[\laps{\delta} B] \nonumber \\
&\aleq \sum_{k=1}^L \lapms{s_k}\abs{A}\ \lapms{t_k} \abs{B} + \sum_{k=1}^L \lapms{\tilde{s}_k}\brac{\lapms{\tilde{t}_k}\abs{A}\ \abs{B}}. \label{eq:comm:Clapmslapslapsd:rhs}
\end{align}
\end{lemma}
\begin{proof}
Since $\beta < 1$,
\[
 \begin{ma}
 \laps{\beta} \brac{\lapms{\tau}A\ \laps{\delta} B}(x)
&=&\int \frac{\lapms{\tau} A(x)\ \laps{\delta}B(x) - \lapms{\tau} A(y)\ \laps{\delta} B(y)}{\abs{x-y}^{m+\beta}}\ dy\\
&=&\int \frac{\lapms{\tau} A(x)\ \laps{\delta}B(x) - \lapms{\tau} A(y)\ \laps{\delta} B(y)}{\abs{x-y}^{m+\beta}}\ dy\\
&=&\lapms{\tau} A(x) \int \frac{\laps{\delta}B(x) - \ \laps{\delta} B(y)}{\abs{x-y}^{m+\beta}}\ dy\\
&&+ \int \frac{\brac{\lapms{\tau} A(x)  -\lapms{\tau} A(y)}\ \laps{\delta} B(y)}{\abs{x-y}^{m+\beta}}\ dy.
\end{ma}
\]
That is,
\[
 \begin{ma}
&& \mathcal{C}(\lapms{\tau} A,\laps{\beta})[\laps{\delta} B](x)\\
&=& \int \frac{\brac{\lapms{\tau} A(x)  -\lapms{\tau} A(y)}\ \laps{\delta} B(y)}{\abs{x-y}^{m+\beta}}\ dy\\
&=& \int\int \frac{\brac{\lapms{\tau} A(x)  -\lapms{\tau} A(y)}\ \brac{B(y+w)-B(y)}}{\abs{w}^{m+\delta} \abs{x-y}^{m+\beta}}\ dw\ dy\\
&=& \int\int \brac{\frac{\brac{\lapms{\tau} A(x)  -\lapms{\tau} A(y-w)}}{\abs{x-y+w}^{m+\beta}}-\frac{\brac{\lapms{\tau} A(x)  -\lapms{\tau} A(y)}}{\abs{x-y}^{m+\beta}}}\ B(y)\ \frac{dw}{\abs{w}^{m+\delta}}\ dy\\
&=& \int\int\int \brac{\frac{\brac{\abs{x-z}^{-m+\tau}  -\abs{z-y+w}^{-m+\tau}}}{\abs{x-y+w}^{m+\beta}}-\frac{\brac{\abs{x-z}^{-m+\tau}  -\abs{z-y}^{-m+\tau}}}{\abs{x-y}^{m+\beta}}}\ B(y)\ A(z)\ \frac{dw}{\abs{w}^{m+\delta}}\ dy\ dz.
\end{ma}
\]
So let us investigate the actual singularity of
\[
k(x,y,z,w) := \abs{w}^{-m-\delta}\abs{\frac{\brac{\abs{x-z}^{-m+\tau}  -\abs{z-y+w}^{-m+\tau}}}{\abs{x-y+w}^{m+\beta}}-\frac{\brac{\abs{x-z}^{-m+\tau}  -\abs{z-y}^{-m+\tau}}}{\abs{x-y}^{m+\beta}}}.
\]
We are going to show the following, for several sets $X \subset \R^{4m}$, which, as the union of these $X$ is $\R^{4m}$, gives the claim.
\begin{equation}\label{eq:setvaluedclaim}
 \int \int \int \chi_{X}(x,z,y,w)\ k(x,z,y,w)\ A(z)\ B(y)dw\ dx\ dy \aleq \eqref{eq:comm:Clapmslapslapsd:rhs}.
\end{equation}
We are denoting $k_1$, $k_2$,
\begin{equation}\label{eq:kest2}
\begin{ma}
k(x,y,z,w) &\leq&  k_1(x,y,z,w) + k_2(x,y,z,w),
\end{ma}
\end{equation}
where
\[
 k_1(x,y,z,w) := \abs{w}^{-m-\delta}\abs{\frac{\brac{\abs{x-z}^{-m+\tau}  -\abs{z-y+w}^{-m+\tau}}}{\abs{x-y+w}^{m+\beta}}},
\]
\[
 k_2(x,y,z,w) := \abs{w}^{-m-\delta}\abs{\frac{\brac{\abs{x-z}^{-m+\tau}  -\abs{z-y}^{-m+\tau}}}{\abs{x-y}^{m+\beta}}}.
\]
We split up the space $(x,y,z,w) \in \R^{4m}$ as follows
\[
\begin{ma}
 A_1 &:=& \{(x,y,z,w) \in \R^{4m}:\ \abs{x-y} \leq 2 \abs{z-y}\quad \wedge\quad \abs{x-y} \leq 2 \abs{x-z}\},\\
 A_2 &:=& \{(x,y,z,w) \in \R^{4m}:\ \abs{x-y} \leq 2 \abs{z-y}\quad \wedge\quad \abs{x-y} > 2 \abs{x-z}\},\\
 A_3 &:=& \{(x,y,z,w) \in \R^{4m}:\ \abs{x-y} > 2 \abs{z-y}\quad \wedge\quad \abs{x-y} \leq 2 \abs{x-z}\},
\end{ma}
\]
\[
\begin{ma}
 B_1 &:=& \{(x,y,z,w) \in \R^{4m}:\ \abs{x-y+w} \leq 2 \abs{z-y+w}\quad \wedge\quad \abs{x-y+w} \leq 2 \abs{x-z}\},\\
 B_2 &:=& \{(x,y,z,w) \in \R^{4m}:\ \abs{x-y+w} \leq 2 \abs{z-y+w}\quad \wedge\quad \abs{x-y+w} > 2 \abs{x-z}\},\\
 B_3 &:=& \{(x,y,z,w) \in \R^{4m}:\ \abs{x-y+w} > 2 \abs{z-y+w}\quad \wedge\quad \abs{x-y+w} \leq 2 \abs{x-z}\},
\end{ma}
\]
\[
\begin{ma}
 C_1 &:=& \{(x,y,z,w) \in \R^{4m}:\ \abs{w} \geq 4\abs{x-y}\},\\
 C_2 &:=& \{(x,y,z,w) \in \R^{4m}:\  \fracm{4}\abs{x-y} \leq \abs{w} < 4\abs{x-y}\},\\
 C_3 &:=& \{(x,y,z,w) \in \R^{4m}:\ 4\abs{w} < \abs{x-y}\}.\\
\end{ma}
\]
Note that
\[
 \bigcup_{i=1}^3 A_i = \bigcup_{i=1}^3 B_i = \bigcup_{i=1}^3 C_i = \R^{4m}.
\]
${}$\\[1em]
\paragraph*{\underline{Ad $\mathbf{(C_1 \cup C_2)\cap B_1}$:}}
First we observe that
\[
\int \chi_{C_1 \cup C_2}(x,z,y,w)\ k_2(x,y,z,w) \ dw \aleq \frac{\abs{\abs{x-z}^{-m+\tau}  -\abs{z-y}^{-m+\tau}}}{\abs{x-y}^{m+\beta+\delta}}.
\]
Following the argument as in the proof of Lemma \ref{la:comm:Hsest}, we consequently have for a finite number $L$, some $s_k,\tilde{s}_k >0$, $t_k, \tilde{t}_k \in (0,\tau)$, $\tilde{s}_k+\tilde{t}_k=s_k+t_k = \tau-\beta-\delta$, $\tilde{s}_k < \epsilon$, such that 
\begin{equation}\label{eq:k2nowest}
\int \int \int \chi_{C_1 \cup C_2}(x,z,y,w)\ k_2(x,z,y,w)\ A(z)\ B(y)dw\ dx\ dy \aleq \sum_{k=1}^L \lapms{s_k}\abs{A}(x)\ \lapms{t_k} \abs{B}(x) + \sum_{k=1}^L \lapms{\tilde{s}_k}\brac{\lapms{\tilde{t}_k}\abs{A}\ \abs{B}}(x).
\end{equation} 
Moreover, for any $\varepsilon \in [0,1]$, since on $B_1$, $\abs{x-z} \approx \abs{z-y+w}$, we can use the mean value theorem and have
\[
\begin{ma}
 \chi_{(C_1 \cup C_2)\cap B_1}\ k_1(x,z,y,w) 
 &\aleq& \abs{w}^{-m-\delta}\ \max\{\abs{x-z},\abs{z-y+w}\}^{-m+\tau-\varepsilon} \abs{x-y+w}^{-m-\beta+\varepsilon} \chi_{(C_1 \cup C_2)\cap B_1}\\
\end{ma}
\]
Now,
\[
 \abs{x-y+w}\chi_{C_1} \ageq \abs{w}\chi_{C_1} \ageq \abs{x-y}\chi_{C_1},
\]
\[
 \abs{x-y+w}\chi_{C_2} \aleq \abs{w}\chi_{C_2} \aeq \abs{x-y}\chi_{C_2},
\]
Consequently,
\[
\begin{ma}
 \chi_{(C_1 \cup C_2)\cap B_1}\ k_1(x,z,y,w) 
 &\aleq& \abs{w}^{-m-\delta}\ \chi_{\abs{w} \ageq \abs{x-y}}\ \abs{x-z}^{-m+\tau-\varepsilon} \abs{x-y}^{-m-\beta+\varepsilon}\\
&&+ \abs{x-y}^{-m-\delta}\ \abs{x-z}^{-m+\tau-\varepsilon} \abs{x-y+w}^{-m-\beta+\varepsilon} \chi_{\abs{x-y+w} \aleq \abs{x-y}}\\
\end{ma}
\]
Integrating in $w$ implies then for any $\delta > 0$, $\varepsilon \in (\beta,1)$
\[
\int \chi_{(C_1 \cup C_2)\cap B_1}\ k_1(x,z,y,w)\ dw \aleq \abs{x-y}^{-m-\delta-\beta+\varepsilon}\ \abs{x-z}^{-m+\tau-\varepsilon},
\]
thus if we choose $\varepsilon \in (\delta + \beta,\tau)$, $s := \tau -\varepsilon$, $t:=-\delta-\beta+\varepsilon$, together with \eqref{eq:k2nowest}, we have shown \eqref{eq:setvaluedclaim} for $X = (C_1 \cup C_2)\cap B_1$.\\[1em]
\paragraph*{\underline{Ad $\mathbf{(C_1 \cup C_2)\cap B_2}$:}}
Next, we consider $\chi_{(C_1 \cup C_2)\cap B_2}k$:
\[
 \chi_{(C_1 \cup C_2)\cap B_2} k_1(x,y,z,w) \aleq \abs{w}^{-m-\delta}\frac{\abs{x-z}^{-m+\tau}}{\abs{x-y+w}^{m+\beta}} \chi_{(C_1 \cup C_2)\cap B_2}  \aleq \abs{w}^{-m-\delta}\frac{\abs{x-z}^{-m+\tau-\varepsilon}}{\abs{x-y+w}^{m+\beta-\varepsilon}} \chi_{(C_1 \cup C_2)\cap B_2},
\]
Now one proceeds exactly as in the situation for $B_1$ above and we have \eqref{eq:setvaluedclaim} for $X = (C_1 \cup C_2)\cap (B_1 \cup B_2)$.\\[1em]
\paragraph*{\underline{Ad $\mathbf{(C_1 \cup C_2)\cap B_3}$:}}
Then, we have to consider $\chi_{(C_1 \cup C_2)\cap B_3}k$
\[
 \chi_{(C_1 \cup C_2)\cap B_3} k_1(x,y,z,w) \aleq \abs{w}^{-m-\delta}\frac{\abs{z-y+w}^{-m+\tau-\varepsilon}}{\abs{x-y+w}^{m+\beta-\varepsilon}} \chi_{(C_1 \cup C_2)\cap B_3},
\]
Using that
\[
 \chi_{(C_1 \cup C_2)}\abs{w} \ageq \max\{\abs{x-y+w},\abs{x-y}\}\ \chi_{(C_1 \cup C_2)},
\]
for any $\varepsilon_1 + \varepsilon_2 = \varepsilon < \tau$
\[
\begin{ma}
 \int \chi_{(C_1 \cup C_2)\cap B_3} k_1(x,y,z,w) A(z)\ dz &\aleq& \abs{x-y}^{-m-\delta + \varepsilon_1}\ {\abs{x-y+w}^{-m-\beta+\varepsilon_2}}\ \int \abs{z-y+w}^{-m+\tau-\varepsilon} A(z)\ dz\\
 &\aeq& \abs{x-y}^{-m-\delta + \varepsilon_1}\ {\abs{x-y+w}^{-m-\beta+\varepsilon_2}}  \lapms{\tau-\varepsilon} A(y-w),
 \end{ma}
\]
that is for any $\varepsilon_2 > \beta$,
\[
\begin{ma}
 \int \int \chi_{(C_1 \cup C_2)\cap B_3} k_1(x,y,z,w) A(z)\ dz\ dw 
 &\aleq& \abs{x-y}^{-m-\delta + \varepsilon_1}\ \int {\abs{x-y+w}^{-m-\beta+\varepsilon_2}} \ \lapms{\tau-\varepsilon} A(y-w) dw\\
 &\aeq& \abs{x-y}^{-m-\delta + \varepsilon_1}\ \int {\abs{x-\tilde{w}}^{-m-\beta+\varepsilon_2}} \ \lapms{\tau-\varepsilon} A(\tilde{w}) d\tilde{w}\\
 &\aeq& \abs{x-y}^{-m-\delta + \varepsilon_1}\ \ \lapms{\tau-\beta-\varepsilon_1} A(x) \\
 \end{ma}
\]
which gives for $\varepsilon_2 > \delta$ the claim \eqref{eq:setvaluedclaim} for $X = (C_1\cup C_2) \cap B_3$, where $s =\tau-\beta-\varepsilon_1$ and $t = \varepsilon_1 - \delta$. Together, \eqref{eq:setvaluedclaim} holds for $X \subseteq C_1 \cup C_2$.\\[1em] 
\paragraph*{\underline{Ad $\mathbf{C_3}$:}}It remains to show the claim for $C_3$:
\[
 \chi_{C_3}(x,y,z,w) \abs{x-y} \approx \chi_{C_3 }(x,y,z,w) \abs{x-y+w} 
\]
\[
 \chi_{C_3 \cap (A_1\cup A_2)}(x,y,z,w) \abs{z-y} \approx \chi_{C_3 \cap (A_1\cup A_2)}(x,y,z,w) \abs{z-y+w} 
\]
In this case, $k_1$ and $k_2$ should not be considered independently, but we rather use the following
\begin{equation}\label{eq:kest1}
\begin{ma}
k(x,y,z,w) &\leq&  \abs{w}^{-m-\delta}\frac{\abs{\abs{x-z}^{-m+\tau}  -\abs{z-y}^{-m+\tau}}}{\abs{x-y}^{m+\beta} \abs{x-y+w}^{m+\beta}} \abs{\abs{x-y}^{m+\beta}  -\abs{x-y+w}^{m+\beta}}\\
&&+ \abs{w}^{-m-\delta}\frac{\abs{\abs{z-y}^{-m+\tau}  -\abs{z-y+w}^{-m+\tau}}}{\abs{x-y+w}^{m+\beta}}.
\end{ma}
\end{equation}
Note that
\[
 \chi_{C_3} \abs{x-y+w} \ageq \chi_{C_3} \abs{x-y}.
 \]
\[
 \chi_{C_3\cap (A_1 \cup A_2)} \abs{y-z+w} \aeq \chi_{C_3\cap (A_1 \cup A_2)}  \abs{z-y}.
 \]
Thus, using \eqref{eq:kest1} with the mean value formula or the conditions $A_2$ for any $\varepsilon = \varepsilon_1 + \varepsilon_2 \in [0,1]$
\[
\begin{ma}
 \chi_{C_3 \cap (A_1 \cup A_2)} k(x,y,z,w) &\leq&  \abs{w}^{-m-\delta}\frac{\min\{\abs{x-z},\abs{z-y}\}^{-m+\tau-\varepsilon}\ \abs{x-y}^\varepsilon}{\abs{x-y}^{2m+2\beta}} \abs{x-y}^{m+\beta-\varepsilon} \abs{w}^\varepsilon \chi_{C_3 \cap (A_1 \cup A_2)}  \\
&&+ \abs{w}^{-m-\delta}\fracm{\abs{x-y}^{m+\beta}}\ \abs{z-y}^{-m+\tau-\varepsilon} \abs{w}^\varepsilon \chi_{C_3 \cap (A_1 \cup A_2)}\\
&\aleq&\abs{w}^{-m-\delta}\frac{\abs{x-z}^{-m+\tau-\varepsilon}}{\abs{x-y}^{m+\beta}}\ \abs{w}^{\varepsilon_1}\ \abs{x-y}^{\varepsilon_2} \chi_{C_3 \cap (A_1 \cup A_2)}  \\
&&+ \abs{w}^{-m-\delta}\fracm{\abs{x-y}^{m+\beta}}\ \abs{z-y}^{-m+\tau-\varepsilon} \abs{w}^{\varepsilon_1}\ \abs{x-y}^{\varepsilon_2} \chi_{C_3 \cap (A_1 \cup A_2)}\\
\end{ma} 
\]
Consequently, if we choose $\varepsilon_1 \in (\delta,\delta+\epsilon/2)$, $\varepsilon_2 \in (\beta,\beta+\epsilon/2)$ such that $\varepsilon = \varepsilon_1 + \varepsilon_2 < \min\{1,\tau\}$, using that
\[
 \int \chi_{C_3} \abs{w}^{-m-\delta+\varepsilon_1} dw \approx \abs{x-y}^{\varepsilon_1-\delta}, 
\]
we arrive at
\[
\begin{ma}
 \int \chi_{C_3 \cap (A_1\cup A_2)} k(x,y,z,w)\ dw  
 &\aleq& \abs{x-z}^{-m+\tau-\varepsilon}\ \abs{x-y}^{-m+\varepsilon-\delta-\beta}\\
&&+ \abs{z-y}^{-m+\tau-\varepsilon}  \abs{x-y}^{-m-\varepsilon-\delta-\beta}.
\end{ma} 
 \]
This implies for $s := \tau-\varepsilon > 0$, $t := \varepsilon - \delta - \beta  \in (0,\epsilon)$, \eqref{eq:setvaluedclaim} for $X = C_3 \cap (A_1 \cup A_2)$:
\[
 \int\int\int \chi_{C_3 \cap (A_1 \cup A_2)} k(x,y,z,w) \ B(y)\ A(z)\ dw\ dy\ dz \aleq \lapms{s} \abs{A}(x)\ \lapms{t} \abs{B}(x) + \lapms{t} \brac{\abs{B}\ \lapms{s} \abs{A}\ }(x).
\]
${}$\\[1em]
\paragraph*{\underline{Ad $\mathbf{C_3 \cap A_3}$:}}
The last case is $C_3 \cap A_3$, where we have by \eqref{eq:kest1}
\[
\begin{ma}
\chi_{C_3\cap A_3} k(x,y,z,w) &\aleq&  \abs{w}^{-m-\delta+\varepsilon} \frac{\abs{z-y}^{-m+\tau}}{\abs{x-y}^{m+\beta}} \abs{x-y}^{-\varepsilon}  \chi_{C_3\cap A_3} \\
&&+ \abs{w}^{-m-\delta}\frac{\abs{\abs{z-y}^{-m+\tau}  -\abs{z-y+w}^{-m+\tau}}}{\abs{x-y}^{m+\beta}} \chi_{C_3\cap A_3} \\
&\aleq&  \abs{w}^{-m-\delta+\varepsilon_1} \frac{\abs{z-y}^{-m+\tau-\varepsilon}}{\abs{x-y}^{m+\beta-\varepsilon_2}}  \chi_{C_3\cap A_3} \\
&&+ \abs{w}^{-m-\delta-\varepsilon_2}\frac{\abs{\abs{z-y}^{-m+\tau}  -\abs{z-y+w}^{-m+\tau}}}{\abs{x-y}^{m+\beta-\varepsilon_2}} \chi_{C_3\cap A_3} \\
\end{ma}
\]
While the first term on the right-hand side behaves exactly as before, for the second term we need yet another case study: If $\abs{w} \leq \fracm{2} \abs{z-y}$, we can proceed as in the cases before using the mean value formula. Note that on the other hand, 
\[
 \chi_{\abs{w} > \fracm{2} \abs{z-y}} \abs{z-y+w} < 3 \abs{w},
\]
Consequently,
\[
\begin{ma}
 &&\chi_{\abs{w} > \fracm{2} \abs{z-y}} \abs{w}^{-m-\delta-\varepsilon_2}\frac{\abs{\abs{z-y}^{-m+\tau}  -\abs{z-y+w}^{-m+\tau}}}{\abs{x-y}^{m+\beta-\varepsilon_2}} \chi_{C_3\cap A_3}\\
 &\aleq&\chi_{\abs{z-y} < 2\abs{w}} \abs{w}^{-m-\delta-\varepsilon_2}\frac{\abs{z-y}^{-m+\tau} }{\abs{x-y}^{m+\beta-\varepsilon_2}} \chi_{C_3\cap A_3}\\
 &&+ \chi_{\abs{z-y+w} < 3\abs{w}} \abs{w}^{-m-\delta-\varepsilon_2+\varepsilon}\frac{\abs{z-y+w}^{-m+\tau-\varepsilon} }{\abs{x-y}^{m+\beta-\varepsilon_2}} \chi_{C_3\cap A_3}\\
 &=:& I + II.
 \end{ma}
\]
Since
\[
 \int \chi_{\abs{w} > \fracm{2} \abs{z-y}} \abs{w}^{-m-\delta-\varepsilon_2}dw \aeq \abs{z-y}^{-\delta-\varepsilon_2},
\]
we arrive at
\[
 \int I\ dw = \frac{\abs{z-y}^{-m+\tau-\delta-\varepsilon_2} }{\abs{x-y}^{m+\beta-\varepsilon_2}} \chi_{C_3\cap A_3}
\]
so setting $\varepsilon = \varepsilon_1 + \varepsilon_2 \in (\beta+\delta,\min\{1,\tau,\beta+\delta+\epsilon\})$, for $\varepsilon_1 \in (\delta,\delta+\epsilon/2)$, $\varepsilon_2 \in (\beta,\beta+\epsilon/2)$ gives the estimate \eqref{eq:setvaluedclaim} for $I$, and as for $II$,
\[
 \int \abs{w}^{-m-\delta-\varepsilon_2+\varepsilon}\frac{\abs{z-y+w}^{-m+\tau-\varepsilon} }{\abs{x-y}^{m+\beta-\varepsilon_2}}\ \abs{A}(z)\ dz \overset{\varepsilon < \tau}{\aeq} 
 \abs{w}^{-m-\delta-\varepsilon_1}\ \abs{x-y}^{-m-\beta+\varepsilon_2}\ \lapms{\tau-\varepsilon} \abs{A}(y-w),
\]
\[
 \int \abs{w}^{-m-\delta-\varepsilon_2+\varepsilon}\ \abs{x-y}^{-m-\beta+\varepsilon_2}\ \lapms{\tau-\varepsilon} \abs{A}(y-w)\ dw \overset{\varepsilon_1 > \delta}{\aeq} \abs{x-y}^{-m-\beta+\varepsilon_2}\ \lapms{\tau-\delta-\varepsilon_2} \abs{A}(y),
\]
and finally,
\[
\int\abs{x-y}^{-m-\beta+\varepsilon_2}\ (\lapms{\tau-\delta-\varepsilon_2} \abs{A}(y))\ \abs{B}(y)\ dy \overset{\varepsilon_2 > \beta}{\aeq}
\lapms{\varepsilon_2-\beta}((\lapms{\tau-\delta-\varepsilon_2} \abs{A})\ \abs{B})(x).
\]
Thus, also $II$ has the required estimate and we have shown \eqref{eq:setvaluedclaim} for $X = C_3\cap A_3$.
\end{proof}

The following estimate should be compared to the estimates in \cite{Chanillo82}, who extended arguments in \cite{CRW76} from Riesz transforms to Riesz Potentials. Their estimates treat cases in which one of the involved functions $b$ belongs to $BMO$, which one usually uses in applications for estimates of that expression in terms of $\laps{s} b$. But if one knows that $\laps{s} b$ exists, then the following estimates are more precise than their $BMO$-counterparts in terms of Lorentz space estimates.
\begin{lemma} 
\label{la:commRzPot}
For any $\delta > 0$ such that $s + \delta < 1$ and any $\gamma \in (s,s+\delta)$, we have
 \[
\begin{ma}
\abs{\laps{s}\mathcal{C}(a,\lapms{s})[b]} &\leq& C_{s,\delta,\gamma}\  \lapms{s+\delta-\gamma} \abs{\laps{s+\delta} a}\ \lapms{\gamma-s} \abs{\lapms{s} b}\\
&&+ C_{s,\delta,\gamma}\ \min \left \{\lapms{\gamma-s}\brac{\abs{\lapms{s} b}\ \lapms{s+\delta-\gamma} \abs{\laps{s+\delta} a}}, \lapms{\gamma-s}\brac{\lapms{s+\delta-\gamma}\abs{\lapms{s} b}\ \abs{\laps{s+\delta} a}} \right \}
\end{ma}
\]
\end{lemma}

\begin{proof}
For $\delta > 0$ such that $s + \delta < 1$.
Set
\[
 B := \lapms{s} b,\quad A := \laps{s+\delta} a.
\]
Then,
\[
 \mathcal{C}(a,\lapms{s})[b] = \lapms{s} \brac{\brac{\lapms{s+\delta} A} \brac{\laps{s} B} - \laps{s} \brac{\brac{\lapms{s+\delta}A} B}}.
\]
Now,
\[
\begin{ma}
 \laps{s} \brac{\brac{\lapms{s+\delta}A} B}(x) &=& c_s \intl_{\R^m} \frac{\lapms{s+\delta}A(x)\  B(x)-\lapms{s+\delta}A(y)\ B(y)}{\abs{x-y}^{n+s}}\ dy\\
&=&  \lapms{s+\delta}A(x)\ c_s\intl_{\R^m} \frac{B(x) - B(y)}{\abs{x-y}^{n+s}}\ dy + c_s \intl_{\R^m} \frac{\lapms{s+\delta}A(x)\ -\lapms{s+\delta}A(y)}{\abs{x-y}^{n+s}}\ B(y)\ dy\\
&=&  \lapms{s+\delta}A(x)\ \laps{s} B(x) + c_{s,\delta} \intl_{\R^m} \intl_{\R^m} \frac{\abs{z-x}^{-n+(s+\delta)} - \abs{z-y}^{-n+(s+\delta)}}{\abs{x-y}^{n+s}}\ A(z)\ B(y)\ dz\ dy\\
\end{ma}
\]
Let now $\gamma \in (s,s+\delta) \subset (0,1)$ and denote
\[
 k(x,y,z) := \frac{\abs{z-x}^{-n+(s+\delta)} - \abs{z-y}^{-n+(s+\delta)}}{\abs{x-y}^{n+s}}.
\]
Now we follow the strategy in \cite{Sfracenergy}. We decompose the space $(x,y,z) \in \R^{3n}$ into several subspaces depending on the relations of $\abs{z-y}$, $\abs{x-y}$, $\abs{x-z}$:
\[
 1 \leq \chi_1(x,y,z) + \chi_2(x,y,z) + \chi_3(x,y,z) + \chi_4(x,y,z) \quad \mbox{for $x,y,z \in \R^m$},
\]
where
\[
 \chi_1 := \chi_{\abs{x-y} \leq  2\abs{z-y}}\ \chi_{\abs{x-y} \leq 2 \abs{x-z}},
\]
\[
 \chi_2 := \chi_{\abs{x-y} \leq  2\abs{z-y}}\ \chi_{\abs{x-y} > 2 \abs{x-z}},
\]
\[
 \chi_3 := \chi_{\abs{x-y} >  2\abs{z-y}}\ \chi_{\abs{x-y} \leq 2 \abs{x-z}}\ \chi_{\abs{x-z} \leq 2 \abs{z-y}}.
\]
\[
 \chi_4 := \chi_{\abs{x-y} >  2\abs{z-y}}\ \chi_{\abs{x-y} \leq 2 \abs{x-z}}\ \chi_{\abs{x-z} > 2 \abs{z-y}}.
\]
Then, by the mean value theorem
\[
 \chi_1(x,y,z) k(x,y,z) \aleq \abs{z-x}^{-n+s+\delta-\gamma}\ \abs{y-x}^{-n-s+\gamma}.
\]
Same holds for
\[
 \chi_2(x,y,z) k(x,y,z) + \chi_3(x,y,z) k(x,y,z) \aleq \abs{z-x}^{-n+s+\delta-\gamma}\ \abs{y-x}^{-n-s+\gamma}.
\]
Finally,
\[
 \chi_4(x,y,z) k(x,y,z) \aleq \abs{z-y}^{-n+s+\delta-\gamma}\ \abs{y-x}^{-n-s+\gamma}.
\]
Note that
\[
 \chi_4(x,y,z)\abs{x-y} \leq \chi_4(x,y,z)\abs{x-z}+ \chi_4(x,y,z)\abs{z-y}\leq \frac{3}{2} \chi_4(x,y,z)\abs{x-z},
\]
and
\[
 \chi_4(x,y,z)\abs{x-y} \geq \chi_4(x,y,z)\abs{x-z}- \chi_4(x,y,z)\abs{z-y}\geq \frac{1}{2} \chi_4(x,y,z)\abs{x-z},
\]
that is
\[
 \chi_4(x,y,z)\abs{x-y} \aeq \chi_4(x,y,z)\abs{x-z}.
\]
That is,
\[
  \chi_4(x,y,z) k(x,y,z) \aleq \abs{z-y}^{-n+s+\delta-\gamma}\ \min \{\abs{y-x}^{-n-s+\gamma}, \abs{z-x}^{-n-s+\gamma}\}
\]
This implies,
\[
\begin{ma}
&&\abs{\laps{s} \brac{\brac{\lapms{s+\delta}A} B}(x)- \lapms{s+\delta}A(x)\ \laps{s} B(x)}\\
 &\leq& c_{s,\delta,\gamma}\  \lapms{s+\delta-\gamma} \abs{A}\ \lapms{\gamma-s} \abs{B}
+ c_{s,\delta,\gamma}\ \min \left \{\lapms{\gamma-s}\brac{\abs{B}\ \lapms{s+\delta-\gamma} \abs{A}}, \lapms{\gamma-s}\brac{\lapms{s+\delta-\gamma}\abs{B}\ \abs{A}}\right \}.
\end{ma}
\]
It is a simple adaption, to show that in the claim, for both terms on the right-hand side, we could have chosen different $\gamma$.
\end{proof}

For $s = 0$, a (non-trivial) version of Lemma \ref{la:commRzPot}, is the following result, for any Riesz transform $\Rz$. Like Lemma~\ref{la:commRzPot} was related to 
Chanillo's \cite{Chanillo82}, this estimate is related to \cite{CRW76}. The proof follows by the same arguments as the proof of Lemma~\ref{la:commRzPot}, we leave the details to the reader.
 \begin{lemma} \label{la:commRzTransf}
Then, for any $\delta \in (0,1)$ and any $\gamma_i \in (0,\delta)$, $i = 1,2$, we have
 \[
\begin{ma}
\abs{\mathcal{C}(a,\Rz)[b]} &\leq& C_{\Rz,\delta,\gamma_1}\  \lapms{\delta-\gamma_1} \abs{\laps{\delta} a}\ \lapms{\gamma_1} \abs{b}
+ C_{\Rz,\delta,\gamma_2}\ \lapms{\gamma_2}\brac{\lapms{\delta-\gamma_2}\abs{b}\  \abs{\laps{\delta} a}}.
\end{ma}
\]
\end{lemma}

\subsection{Fractional product rules in the Hardy-space via para-products -- including the limit case}\label{ss:drprodrules}
\label{s:commi2}
In this section we introduce several commutators, and show how to use techniques developed by Da Lio and Rivi\`{e}re in \cite{DR1dSphere} in order to estimate their behavior involving the Hardy spaces $\mathcal{H}$. For the case $\mu < 1$ the new contribution are the commutators themselves, the techniques for the proof of their behavior follows the arguments of Da Lio and Rivi\`ere. In the case $\mu = 1$, these arguments have to be extended and more precise, in order to show the same behavior for this limit case, which was unknown up to now. The essential commutator estimate is
\begin{equation}\label{eq:c2:hardyguy}
 \Vrac{\laps{\mu} \brac{\Rz[h]\ \lapms{\mu} b - \Rz[h\ \lapms{\mu} b]}}_{\mathcal{H}} \leq \Vert h \Vert_{2}\ \Vert b \Vert_{2} \quad \mbox{whenever $\mu \in (0,1]$}.
\end{equation}
From this we will conclude
\begin{equation}\label{eq:c2:CfRlaphvarphiest}
 \vrac{ \mathcal{C}(f,\Rz)[\laps{\mu} \varphi] }_2 \aleq \vrac{\laps{\mu} f}_2\ [\varphi]_{BMO},
\end{equation}
and
\begin{equation}\label{eq:c2:H1varphigest}
 \vrac{H_\mu(\varphi,g)}_2 \aleq \vrac{\laps{\mu} g}_2\ [\varphi]_{BMO}.
\end{equation}
as well as
\begin{equation}\label{eq:c2:laphH1inhardy}
 \vrac{\laps{\mu} H_{\mu}(a,b)}_{\mathcal{H}} \aleq \vrac{\laps{\mu} a}_2\ \vrac{\laps{\mu} b}_2 \quad \mbox{for $\mu \in (0,1]$}.
\end{equation}
Note that \eqref{eq:c2:laphH1inhardy} is already known for $\mu < 1$ \cite{DR1dSphere}, and we are going to prove it only for the new situation $\mu = 1$, when we will show that it follows from \eqref{eq:c2:hardyguy}.
\begin{proof}[Proof of \eqref{eq:c2:hardyguy}] 
Let
\[
 T(h,b) := \laps{\mu} \brac{\Rz[h]\ \lapms{\mu}b - \Rz[h\ \lapms{\mu}b]}
\]
We will follow the basic ideas of the proof of the commutator theorem in \cite{DR1dSphere}, which goes through without great changes, if $\mu < 1$, and then point out where it fails if $\mu = 1$. In the latter case, we have make a precise computation of the failing term, and show that in fact this term is essentially a commutator as in \cite{CRW76} and another good term. For a general reference, we refer the reader to the proof in \cite{DR1dSphere}, as we will sketch some of the parts which behave no worse than in their case. In particular, the notion of Besov- and Triebel spaces and their respective estimates are taken from \cite{DR1dSphere} without too much changes. 
\subsection*{The para-products and Littlewood-Paley decomposition}
We need to control three parts of $T$. We introduce the projections $\Pi_1$, $\Pi_2$, $\Pi_3$ defined via
\[
\begin{ma}
 \Pi_1 T(f,g) &:=& \sum_{j\in \Z} T(f_j,g^{j-4}),\\
 \Pi_2 T(f,g) &:=& \sum_{j\in \Z} T(f^{j-4},g_j),\\
 \Pi_3 T(f,g) &:=& \sum_{j\in \Z} \sum_{l=j-4}^{j+4} T(f_j,g_l).\\
\end{ma}
\]
Here,
\[
 \supp f_j^\wedge \subset \{\abs{\xi} \in (2^{j-1},2^{j+1}) \},
\]
and
\[
 \supp (f^j)^\wedge \subset \{\abs{\xi} \leq 2^{j+1} \}.
\]
These terms come from the Littlewood-Paley decomposition.
Note that then for $k \in \Z$,
\begin{align} 
\Pi_1 T(f,g)_k &:= \sum_{j=k-3}^{k+3} T(f_j,g^{j-4})_k, \label{eq:c2:Pi1Tfgk}\\
 \Pi_2 T(f,g)_k &:= \sum_{j=k-3}^{k+3} T(f^{j-4},g_j)_k,\label{eq:c2:Pi2Tfgk}\\
 \Pi_3 T(f,g)_k &:= \sum_{j = k-5}^{\infty} \sum_{l=j-4}^{j+4} T(f_j,g_l)_k.\label{eq:c2:Pi3Tfgk}
\end{align}

\subsection*{Estimates from \cite{DR1dSphere}}
Let us recall the following estimates, whose proof up to small adaptions can be found in \cite{DR1dSphere}, for a general overview we refer to Tao's lecture notes \cite{TaoLectures}:
First, denoting with $\mathcal{M}$ the maximal function,
\begin{equation}\label{eq:c2:fjsupmaxfct}
 \sup_{k \in \Z} \abs{2^{-\gamma }\laps{\gamma} f^{k}(x)} \aleq \mathcal{M} f(x) \quad \mbox{for any $\gamma \geq 0$ and almost every $x \in \R^m$.}
\end{equation}
Here, for $\gamma = 0$, we set $\laps{0} := \operatorname{Id}$. In fact, for a suitable $\tilde{\phi}$ in the Schwartz class,
\[
 2^{-\gamma k}\laps{\gamma} f^{k}(x) = 2^{(m-\gamma)k} \int \tilde{\phi}(2^k(x- y))\ \laps{\gamma} f(y) = = 2^{(m)j} \int (\laps{\gamma}\tilde{\phi})(2^j(x- y))\  f(y). 
\]
Now the argument follows the exactly the one of \cite{DR1dSphere}, using that
\[
 \abs{\laps{\gamma}\tilde{\phi}(z)} \aleq \begin{cases}
                                     \frac{1}{1+\abs{z}^{m+\gamma}} \quad &\mbox{if $\gamma > 0$}\\
				     \frac{1}{1+\abs{z}^{\beta}} \quad &\mbox{if $\gamma = 0$, for any $\beta > 0$},
                                    \end{cases}
\]
and we arrive at
\[
 2^{-\gamma k}\laps{\gamma} f^{k}(x) \aleq \mathcal{M}f(x)\ \sum_{j\in \Z}^\infty 2^{mj} \sup_{\abs{z} \leq 2^j} \abs{\laps{\gamma} \tilde{\phi}(z)},
\]
which is convergent if $\gamma \geq 0$, but might be divergent for $\gamma < 0$. This shows \eqref{eq:c2:fjsupmaxfct}.\qed

Let $N$ be a zero-multiplier operator with $0$-homogeneous symbol $n$. We have (cf. \cite[Notes 2 for 25A, Lemma 1.1]{TaoLectures})
\begin{equation}\label{eq:commis:Binftyinfyest}
 \sup_{k} \vrac{2^{-\gamma k} N \laps{\gamma} f_{k}}_{\infty} \aleq C_N\ \vrac{f}_{\dot{B}_{\infty,\infty}^0} \quad \mbox{for $\gamma \geq 0$}
\end{equation}
and thus also
\begin{equation}\label{eq:commis:fhochjBinftyinfyest}
 \vrac{\laps{\gamma} N f^{j}}_{\infty} \aleq 2^{\gamma j} \vrac{f}_{\dot{B}_{\infty,\infty}^0} \quad \mbox{for $\gamma > 0$}
\end{equation}
A crucial role is played by
\begin{equation}\label{eq:commis:bhatjest}
\brac{\int_{\R^m} \sum_{j \in \Z} \abs{2^{j(\gamma-\alpha)}N\laps{\alpha-\gamma} f^{j}}^2}^{\fracm{2}} \aleq C_{\alpha-\gamma,N} \vrac{f}_2 \quad \mbox{whenever $0 \leq \gamma < \alpha$}.
\end{equation}
which follows from
\[
\begin{ma}
  \brac{\int_{\R^m} \sum_{j \in \Z} \abs{2^{j(\gamma-\alpha)}N\laps{\alpha-\gamma} f^{j}}^2}^{\fracm{2}}
  &\overset{\gamma < \alpha}{\aleq}& C_{\alpha-\gamma,N} \brac{\int_{\R^m} \sum_{j \in \Z} \abs{2^{j(\gamma-\alpha)} \laps{\alpha-\gamma} f_j}^2}^{\fracm{2}},
\end{ma}
\]
an argument which can be found in the commutator estimates in \cite{DR1dSphere}.

\subsection*{Estimate of $\Pi_2$}
We have
\[
 \begin{ma}
\vrac{\Pi_2 T(h,b)}_{\mathcal{H}} 
&\aleq& \vrac{\laps{\mu}\Pi_2  (\Rz[h]\ \lapms{\mu}b)}_{\mathcal{H}} + \vrac{\Rz[\laps{\mu} \Pi_2 (h\ \lapms{\mu}b)]}_{\mathcal{H}}\\
&\aleq& \vrac{\laps{\mu}\Pi_2 (\Rz[h]\ \lapms{\mu}b)}_{\mathcal{H}} + \vrac{\laps{\mu} \Pi_2 (h\ \lapms{\mu}b)}_{\mathcal{H}}.
\end{ma}
\]
Let $\tilde{h} := \Rz[h]$ or $h$, respectively. Then
\[
\begin{ma}
\vrac{\laps{\mu}\Pi_2 (\tilde{h}\ \lapms{\mu}b)}_{\mathcal{H}} 
&\aeq& \vrac{\laps{\mu}\Pi_2 (\tilde{h}\ \lapms{\mu}b)}_{\dot{F}^0_{1,2}}\\
&\aeq& \vrac{\Pi_2 (\tilde{h}\ \lapms{\mu}b)}_{\dot{F}^\mu_{1,2}}\\
&\aeq& \intl_{\R^m} \brac{\sum_{k\in\Z} 2^{2\mu k}\abs{\Pi_2 (\tilde{h}\ \lapms{\mu}b)}^2}^{\fracm{2}}\\
&\overset{\eqref{eq:c2:Pi2Tfgk}}{\aeq}& \intl_{\R^m} \brac{\sum_{k\in\Z} 2^{2\mu k}\abs{\sum_{j=k-3}^{k+3} (\tilde{h}^{j-4}\ \lapms{\mu}b_j)_k}^2}^{\fracm{2}}\\
&\aleq& \intl_{\R^m} \brac{\sum_{k\in\Z}\max_{j\in [k-3,k+3]\cap \Z} 2^{2\mu k}\abs{(\tilde{h}^{j-4}\ \lapms{\mu}b_j)_k}^2}^{\fracm{2}}\\
&\aleq& \intl_{\R^m} \brac{\sum_{k\in\Z} 2^{2\mu k} (\sup_{i}\abs{\tilde{h}^{i}}^2)\ \abs{\lapms{\mu}b_k}^2}^{\fracm{2}}
\end{ma}
\]
Thus,
\[
\begin{ma}
\vrac{\laps{\mu}\Pi_2 (\tilde{h}\ \lapms{\mu}b)}_{\mathcal{H}} 
&\overset{\eqref{eq:c2:fjsupmaxfct}}{\aleq}& \intl_{\R^m} \mathcal{M}\tilde{h}\ \brac{\sum_{k\in\Z} 2^{2\mu k}\abs{\lapms{\mu}b_j}^2}^{\fracm{2}}\\
&\aleq& \vrac{\tilde{h}}_{2}\ \vrac{b}_{\dot{B}_{2,2}^0}\\
&\aleq&\vrac{h}_{2}\ \vrac{b}_{2}.
\end{ma}
\]
It is worth noting, that this argument holds for any $\mu > 0$.\\
\subsection*{Estimate of $\Pi_1$}
Now we treat $\Pi_1 T$. For some $\psi$, $\vrac{\psi}_{\dot{B}^0_{\infty,\infty}} \leq 1$ we have
\[
\begin{ma}
 \vrac{\Pi_1 T(h,b)}_{\mathcal{H}} &\aleq& \vrac{\Pi_1 T(h,b)}_{\dot{B}^0_{1,1}}\\
&\aleq& \int \Pi_1 T(h,b)^\wedge(\xi)\ \psi^\vee(\xi)\ d\xi\\
&=& c\ \sum_{j\in \Z} \int \int k(\xi,\eta)\ h_j^\wedge(\eta)\ (b^{j-4})^\wedge(\xi-\eta)\ \psi^\vee(\xi)\ d\eta\ d\xi,
\end{ma}
\]
where
\begin{equation}\label{eq:commis:kfordifficultguy}
 k(\xi,\eta) = \brac{\frac{\abs{\xi}}{\abs{\xi-\eta}}}^\mu \brac{\frac{\eta}{\abs{\eta}} - \frac{\xi}{\abs{\xi}}}.
\end{equation}
Note that by the support of the different factors, we only consider $\xi$, $\eta$ such that $\abs{\xi-\eta} \aleq 2^{j-3}$, $2^{j-1} \leq \abs{\eta} \leq 2^{j+1}$, and thus $2^{j-2} \leq \abs{\xi} \leq 2^{j+2}$. In particular,
\begin{equation}\label{la:commis:Pi2symbbehaviour}
 \frac{\abs{\xi-\eta}}{\abs{\eta}} \leq \fracm{2}.
\end{equation}
We can assume w.l.o.g. that if \eqref{la:commis:Pi2symbbehaviour} is satisfied, following identity holds, and the right hand side converges absolutely
\begin{equation}\label{eq:commis:kxietadecompmul1}
\begin{ma}
  k(\xi,\eta) &=&  \brac{\frac{\abs{\xi}}{\abs{\xi-\eta}}}^\mu\ \sum_{l=1}^\infty \fracm{l!} m_l(\eta)\ n_l(\xi-\eta) \brac{\frac{\abs{\xi-\eta}}{\abs{\eta}}}^l\\
    &=& \sum_{l=1}^\infty \fracm{l!} m_l(\eta)\ n_l(\xi-\eta)\ \abs{\xi-\eta}^{l-\mu}\ \abs{\xi}^\mu\ \abs{\eta}^{-l}.
 \end{ma}
\end{equation}
where $n_l$, $m_l$ are zero-homogeneous functions, and we will denote their respective operators with $N_l$, $M_l$. In fact, one can check that this is true, if \eqref{la:commis:Pi2symbbehaviour} is satisfied for $2^{-L}$ on the right-hand side, for some $L \in \N$. If $L$ is not $1$, then we need to replace for the decompositions $\Pi_1$, $\Pi_2$, $\Pi_3$ and use instead
\[
\begin{ma}
 \tilde{\Pi}_1 T(f,g) &:=& \sum_{j\in \Z} T(f_j,g^{j-3-L}),\\
 \tilde{\Pi}_2 T(f,g) &:=& \sum_{j\in \Z} T(f^{j-3-L},g_j),\\
 \tilde{\Pi}_3 T(f,g) &:=& \sum_{j\in \Z} \sum_{l=j-3-L}^{j+3+L} T(f_j,g_l).\\
\end{ma}
\]
Thus, for the sake of simplicity, we are going to assume that $L=1$.
%
%

\subsubsection*{Estimates for $\mathbf{\Pi_1}$ if $\mathbf{\mu < 1}$.}
The expansion in \eqref{eq:commis:kxietadecompmul1} is precise enough, if $\mu < 1$, since then $l-\mu > 0$ for all $l \in \N$. Here $(\cdot)_{\tilde{k}}$ is the Fourier cutoff on $\abs{\xi} \in (2^{k-2},2^{k+2})$.
\begin{equation}\label{eq:commis:Pi1Thbest}
\begin{ma}
 \vrac{\Pi_1 T(h,b)}_{\mathcal{H}} 
&\aleq& \sum_{l=1}^\infty \fracm{l!}\ \sum_{j\in \Z} \int M_l \lapms{l} h_j\ N_l\laps{l-\mu} b^{j-4}\ \laps{\mu}  \psi_{\tilde{j}}\\
&\aleq& \sum_{l=1}^\infty \fracm{l!}\ \int \sum_{j\in \Z} \abs{2^{jl} M_l \lapms{l} h_j}\ \abs{2^{j(\mu-l)}N_l\laps{l-\mu} b^{j-4}}\ \sup_{\tilde{k}} \vrac{2^{-\mu \tilde{k}}\laps{\mu} \psi_{\tilde{k}}}_{\infty}\\
&\overset{\eqref{eq:commis:Binftyinfyest}}{\aleq}& \sum_{l=1}^\infty \fracm{l!} 2^{4(\mu-l)}\ \brac{\int_{\R^m} \sum_{j\in \Z} \abs{2^{jl} M_l \lapms{l} h_j}^2}^{\fracm{2}}\ \brac{\int_{\R^m} \sum_{j \in \Z} \abs{2^{(j-4)(\mu-l)}N_l\laps{l-\mu} b^{j-4}}^2}^{\fracm{2}} \\
&\overset{\ontop{l \geq 1 > \mu}{\eqref{eq:commis:bhatjest}}}{\aleq}& C_{1-\mu}\ \sum_{l=1}^\infty c_l\ 2^{4(\mu-l)}\ \vrac{h}_2\ \vrac{b}_2\\
&\aleq& \vrac{h}_2\ \vrac{b}_2.
\end{ma} 
\end{equation}
The crucial point in the above estimate is the application of \eqref{eq:commis:bhatjest}, and it was there that we used that $\mu < l$ for all $l \geq 1$, that is $\mu < 1$. In particular, if the summation on $l$ above had started with $l = 2$, this estimate would have been true for any $\mu < 2$. Also, note that for $\mu = 0$, this estimate still holds.

\subsubsection*{Adaptions for $\mathbf{\Pi_1}$ if $\mathbf{\mu = 1}$.} Thus, if $\mu = 1$, we have to be more precise and actually need to find good controls of the Taylor-expansion for $l=1$:
\[
   k(\xi,\eta) \overset{\eqref{eq:commis:kfordifficultguy}}{=} -k_1(\xi,\eta) + \sum_{l=2}^\infty \fracm{l!} m_l(\eta)\ n_l(\xi-\eta)\ \abs{\xi-\eta}^{l-1}\ \abs{\eta}^{1-l},
\]
where 
\[
\begin{ma}
 k_1(\xi,\eta) &=& \frac{\abs{\xi}}{\abs{\xi-\eta}} \frac{-(\xi-\eta)\abs{\eta}+\eta \frac{\eta^k}{\abs{\eta}}(\xi-\eta)^k}{\abs{\eta}^2}\\
&=& \frac{\abs{\xi}}{\abs{\eta}} \brac{-\frac{\xi-\eta}{\abs{\xi-\eta}}+\frac{\eta}{\abs{\eta}} \frac{\eta^k}{\abs{\eta}}\frac{(\xi-\eta)^k}{\abs{\xi-\eta}}}\\
&=& \frac{\abs{\xi}}{\abs{\eta}} \brac{-\frac{\eta}{\abs{\eta}}-\frac{\xi-\eta}{\abs{\xi-\eta}}+\frac{\eta}{\abs{\eta}}+\frac{\eta}{\abs{\eta}} \frac{\eta^k}{\abs{\eta}}\frac{(\xi-\eta)^k}{\abs{\xi-\eta}}}\\
&=& \frac{\abs{\xi}}{\abs{\eta}} \brac{-\brac{\frac{\eta}{\abs{\eta}}+\frac{\xi-\eta}{\abs{\xi-\eta}}}+\frac{\eta}{\abs{\eta}} \frac{\eta^k}{\abs{\eta}}\brac{\frac{\eta^k}{\abs{\eta}}+\frac{(\xi-\eta)^k}{\abs{\xi-\eta}}}}\\
&=& \frac{\abs{\xi}-\abs{\eta}}{\abs{\eta}} \brac{-\brac{\frac{\eta}{\abs{\eta}}+\frac{\xi-\eta}{\abs{\xi-\eta}}}+\frac{\eta}{\abs{\eta}} \frac{\eta^k}{\abs{\eta}}\brac{\frac{\eta^k}{\abs{\eta}}+\frac{(\xi-\eta)^k}{\abs{\xi-\eta}}}}\\
&& -\brac{\frac{\eta}{\abs{\eta}}+\frac{\xi-\eta}{\abs{\xi-\eta}}}+\frac{\eta}{\abs{\eta}} \frac{\eta^k}{\abs{\eta}}\brac{\frac{\eta^k}{\abs{\eta}}+\frac{(\xi-\eta)^k}{\abs{\xi-\eta}}}\\
&=:& \frac{\abs{\xi}-\abs{\eta}}{\abs{\eta}}\ \tilde{m}(\eta,\xi-\eta)+f(\eta,\xi-\eta).
\end{ma}
\]
Now we define bilinear operators $\tilde{T}_1$, $\tilde{T}_2$ via
\[
 \tilde{T}_1(a,b)^\wedge(\xi) := \int f(\eta,\xi-\eta)\ a^\wedge(\eta)\ b^\wedge(\xi-\eta)\ d\eta,
\]
\[
 \tilde{T}_2(a,b)^\wedge(\xi) := \int \frac{\abs{\xi}-\abs{\eta}}{\abs{\eta}} \tilde{m}(\eta,\xi-\eta)\ a(\eta)\ b(\xi-\eta).
\]
By the commutator arguments of \cite{CRW76}, $\tilde{T}_1(\cdot,\cdot)$ can be estimated
\[
\vrac{\tilde{T}_1(h,b)}_{\mathcal{H}} \aleq \vrac{h}_2\ \vrac{b}_2.
\]
Indeed, for some constants $c_1$, $c_2$,
\[
 \tilde{T}_1(h,b) = c_1 (\Rz [a] + \Rz[b]) + c_2 (\Rz_k [\Rz\Rz_k[a]]\ b + \Rz\Rz_k[a]\ \Rz_k[b])
\]
Thus,
\[
\begin{ma}
  \vrac{\Pi_1 T(h,b)}_{\mathcal{H}} &\aleq&  \vrac{\Pi_1 (T(h,b)-\tilde{T}_1(h,b)-\tilde{T}_2(h,b))}_{\mathcal{H}} + \vrac{\Pi_1 \tilde{T}_2(h,b))}_{\mathcal{H}}\ + \vrac{h}_2\ \vrac{b}_2\\
&=:& I + II + \vrac{h}_2\ \vrac{b}_2.
\end{ma}
\]
The term $I$ can be estimated in the same manner as in \eqref{eq:commis:Pi1Thbest} for $\mu < 1$, since the only term for which \eqref{eq:commis:bhatjest} was not applicable is now cut away.

Thus it remains to estimate $II$. Note that $\tilde{m}$ is a finite sum of products of zero-order multipliers of $\eta$ and $\xi-\eta$, which thus plays no role in our argument. The crucial point is, that by \eqref{la:commis:Pi2symbbehaviour} 
\[
 \frac{\abs{\xi}-\abs{\eta}}{\abs{\eta}} = \sum_{l=1}^\infty \fracm{l!}\ m_l(\eta)\ n_l(\xi-\eta)\ \abs{\eta}^{-l}\ \abs{\xi-\eta}^{l},
\]
so the exponent of $\abs{\xi-\eta}$ is always strictly greater than zero. The estimate on $II$ follows then from along the arguments in \eqref{eq:commis:Pi1Thbest} for $\mu = 0$.
\subsection*{Estimate of $\Pi_3$}
It remains to estimate $\Pi_3$. Let $N_1$, $N_2$ be a CZ-zero-multiplier operator (in this case or the identity or the Riesz Transform $\Rz$). In order to estimate $\Pi_3 T(h,b)$ in the Hardy-space norm, it then suffices to control terms of the form
\[
 \vrac{N_1 \laps{\mu} \Pi_3 (N_2 h\ \lapms{\mu}b)}_{\mathcal{H}}.
\]
Again, for some $\psi$, $\vrac{\psi}_{\dot{B}^0_{\infty,\infty}} \leq 1$, 
\[
 \begin{ma}
\vrac{N_1 \laps{\mu} \Pi_3 (N_2 h\ \lapms{\mu}b)}_{\mathcal{H}} 
&\aleq& \vrac{\laps{\mu}\Pi_3 (N_2 h\ \lapms{\mu}b)}_{\mathcal{H}}\\ 
&\aleq& \vrac{\laps{\mu}\Pi_3 (N_2 h\ \lapms{\mu}b)}_{\dot{B}^0_{1,1}}\\
&\aleq& \int \laps{\mu}\Pi_3 (N_2 h\ \lapms{\mu}b)\ \psi\\
&=& \int \Pi_3 (N_2 h\ \lapms{\mu}b)\ \laps{\mu} \psi\\
&=&  \sum_{j\in \Z} \sum_{l=j-4}^{j+4} \int N_2 h_j\ \lapms{\mu}b_l\ \laps{\mu} \psi\\
&=&  \sum_{j\in \Z} \sum_{l=j-4}^{j+4} \int N_2 h_j\ \lapms{\mu}b_l\ \laps{\mu} \psi^{j+6}\\
&\overset{\ontop{\mu > 0}{\eqref{eq:commis:fhochjBinftyinfyest}}}{\aleq}& \sum_{j\in \Z} \sum_{l=j-4}^{j+4} \int \abs{N_2 h_j}\ \abs{\lapms{\mu}b_l}\ 2^{\mu j} \vrac{\psi}_{\dot{B}_{\infty,\infty}^0}\\
&\aleq& \vrac{N_2 h}_{\dot{F}^0_{2,2}}\ \vrac{\lapms{\mu}b}_{\dot{F}^\mu_{2,2}}\\
&\aleq& \vrac{h}_2\ \vrac{b}_2.
 \end{ma}
\]
This shows \eqref{eq:c2:hardyguy}
\end{proof}

Then we are able to give the

\begin{proof}[Proof of \eqref{eq:c2:CfRlaphvarphiest}]
The estimate is a consequence of \eqref{eq:c2:hardyguy}: Let $\psi \in C_0^\infty$, $\vrac{\psi}_2 \leq 1$ such that
\[
\begin{ma}
 \vrac{ \mathcal{C}(f,\Rz)[\laps{\mu} \varphi] }_2 &\leq& 2 \int \mathcal{C}(f,\Rz)[\laps{\mu} \varphi] \ \psi \\
&=& 2 \int (f \Rz [\laps{\mu} \varphi] - \Rz[f\laps{\mu} \varphi]) \psi\\
&=& -2 \int \varphi\ \laps{\mu} (\Rz [f \psi]- f \Rz[\psi]).
\end{ma}
\]
Taking $h := \psi$ and $b := \laps{\mu} f$, one concludes the proof of \eqref{eq:c2:CfRlaphvarphiest} by \eqref{eq:c2:hardyguy}.
\end{proof}

\begin{proof}[Proof of \eqref{eq:c2:H1varphigest}]
If $\mu = 1$, \eqref{eq:c2:H1varphigest} follows from the following easy argument: 
\[
 \vrac{H_1(\varphi,g)}_2 = c \vrac{\Rz_i \Rz_i H_1(\varphi,g)}_2 \aleq \sum_i \vrac{\Rz_i H_1(\varphi,g)}_2.
\]
Moreover,
\[
\begin{ma}
 \Rz_i H_1(\varphi,g) 
&=& g \partial_i \varphi + \varphi \partial_i g - \Rz_i [g\ \laps{1} \varphi] - \Rz_i [\varphi\ \laps{1} g] \\
&=&  g \Rz_i [\laps{1} \varphi] + \varphi \Rz_i [\laps{1} g] - \Rz_i [g\ \laps{1} \varphi] - \Rz_i [\varphi\ \laps{1} g] \\
&=&  \mathcal{C}(g,\Rz_i)[\laps{1} \varphi] +\mathcal{C}(\varphi,\Rz_i)[\laps{1} g].
\end{ma} 
\]
Consequently, by the \cite{CRW76}-commutator theorem
\[
 \vrac{H_1(\varphi,g)}_2 \aleq \sum_i \vrac{\mathcal{C}(g,\Rz_i)[\laps{1} \varphi]}_2 + \sum_i \vrac{\mathcal{C}(\varphi,\Rz_i)[\laps{1} g]}_2
\aleq \sum_i \vrac{\mathcal{C}(g,\Rz_i)[\laps{1} \varphi]}_2  + [\varphi]_{BMO}\ \vrac{\laps{1} g}_2, 
\]
and therefore we have reduced \eqref{eq:c2:H1varphigest} to \eqref{eq:c2:CfRlaphvarphiest}.\\
If $\mu < 1$, we show that for
\[
 \tilde{H}_\mu(a,b) := \laps{\mu} a\ b - a \laps{\mu} b - \laps{\mu} (ab)
\]
we have the estimate
\[
 \vrac{\tilde{H}_\mu(a,b)}_{\mathcal{H}} \aleq \vrac{a}_2\ \vrac{\laps{\mu} b}_2.
\]
Use again the decomposition in $\Pi_1$, $\Pi_2$, $\Pi_3$.  We have for any $\mu > 0$,
\begin{equation}\label{eq:commis:Htildegoodguys}
 \vrac{\Pi_3(\laps{\mu}(a b)}_{\mathcal{H}} + \vrac{\Pi_1(a \laps{\mu} b)}_{\mathcal{H}} + \vrac{\Pi_2((\laps{\mu} a)\ b)}_{\mathcal{H}} + \vrac{\Pi_2(a \laps{\mu} b)}_{\mathcal{H}} \leq \vrac{a}_2\ \vrac{\laps{\mu}b}_2,
\end{equation}
\paragraph*{Estimate of $\Pi_1$.} We have,
\[
\vrac{\Pi_1\tilde{H}_\mu(a,b)}_{\mathcal{H}} \overset{\eqref{eq:commis:Htildegoodguys}}{\aleq} \|\sum_j\brac{(\laps{\mu}a_j) b^{j-4}-\laps{\mu}(a_j b^{j-4})}\|_{\mathcal{H}} + \vrac{a}_2\ \vrac{\laps{\mu}b}_2. 
\]
Thus we have to estimate
\[
 \sum_{j=-\infty}^\infty \int_{\R^m} \brac{(\laps{\mu}a_j) b^{j-4}-\laps{\mu}(a_j b^{j-4})}\ \psi_{\tilde{j}}
\]
The respective kernel is
\[
 k(\xi,\eta) = \abs{\eta}^{\mu} - \abs{\xi}^\mu = \sum_{l = 1}^\infty c_l m_l(\eta)\ n_l(\xi-\eta)\ \abs{\eta}^{\mu-l} \abs{\xi-\eta}^l,
\]
that is
\[
\begin{ma}
 &&\sum_{j=-\infty}^\infty \int_{\R^m} \brac{(\laps{\mu}a_j) b^{j-4}-\laps{\mu}(a_j b^{j-4})}\ \psi_{\tilde{j}}\\
 &\overset{\mu < 1}{=}&  \sum_{l=1}^\infty c_l 2^{4(\mu-l)}\ \sum_{j=-\infty}^\infty \int_{\R^m} 2^{j(l-\mu)} (N_l\lapms{l-\mu}a_j)\ 2^{(\mu-l)(j-4)}\laps{l-\mu}\laps{\mu}b^{j-4}\ \psi_{\tilde{j}}\\
 &\overset{\mu < 1}{\aleq}& \vrac{a}_2\ \vrac{\laps{\mu}b}_2,
\end{ma}
\]
where for the last step we can just copy the remaining arguments from \eqref{eq:commis:Pi1Thbest}, since $\mu < 1 \leq l$.
\paragraph*{Estimate of $\Pi_2$.} Since
\[
 \tilde{H}_\mu(a,b) = \laps{\mu} a\ b - 2 a \laps{\mu} b + (a \laps{\mu} b - \laps{\mu} (ab)),
\]
we can 
\[
\vrac{\Pi_2\tilde{H}_\mu(a,b)}_{\mathcal{H}} \overset{\eqref{eq:commis:Htildegoodguys}}{\aleq} 
\|\sum_j\brac{\laps{\mu}(a^{j-4} b_{j})-a^{j-4} \laps{\mu}b_{j}}\|_{\mathcal{H}} + \vrac{a}_2\ \vrac{\laps{\mu}b}_2. 
\]
Since $\mu < 1$, as in the argument for $\Pi_1$, we can estimate
\[
 \|\sum_j\brac{\laps{\mu}(a^{j-4} b_{j})-a^{j-4} \laps{\mu}b_{j}}\|_{\mathcal{H}} \aleq \vrac{a}_2\ \vrac{\laps{\mu}b}_2. 
\]
\paragraph*{Estimate of $\Pi_3$.} Finally, we have
\[
\vrac{\Pi_3\tilde{H}_\mu(a,b)}_{\mathcal{H}} \overset{\eqref{eq:commis:Htildegoodguys}}{\aleq} 
\|\sum_{j\approx i}\brac{\laps{\mu}(a_{j}) b_{i}-a_{j} \laps{\mu}b_{i}}\|_{\mathcal{H}} + \vrac{a}_2\ \vrac{\laps{\mu}b}_2,
\]
and have to estimate
\[
 \sum_{j = -\infty}^\infty \sum_{i=j-4}^{j+4}\brac{\laps{\mu}(a_{j}) b_{i}-a_{j} \laps{\mu}b_{i}} \psi^{j-6} + 
 \sum_{j = -\infty}^\infty \sum_{i=j-4}^{j+4} \sum_{k = j-6}^{j+6} \brac{\laps{\mu}(a_{j}) b_{i}-a_{j} \laps{\mu}b_{i}} \psi_{k} =: I + II,
\]
The kernel here is (note that for $\xi \in \supp (\psi^{j-6})^\wedge$, $\eta \in \supp (a_j)^\wedge$, $\abs{\xi}/\abs{\eta} \leq \fracm{2}$ which ensures convergence)
\begin{equation}\label{eq:commis:pi3htilde:kexp}
 k(\eta,\xi) = \abs{\eta}^\mu -\abs{\xi-\eta}^\mu = \sum_{l = 1}^\infty c_l\ m_l(\xi-\eta) n_l(\xi) \abs{\eta}^{\mu-l}\ \abs{\xi}^l,
\end{equation}
thus we have
\[
\begin{ma}
 I &=& \sum_{l=1}^\infty c_l\ \sum_{j = -\infty}^\infty \sum_{i=j-4}^{j+4} \int M_l \laps{\mu-l} a_{j}\ b_{i}\ N_l \laps{l}\psi^{j-6}\\
 &\aleq& \sum_{l=1}^\infty c_l\  2^{-6l} \sup_k 2^{-l(k-6)} \vrac{N_l \laps{l}\psi^{k-6}}_\infty\ \sum_{j = -\infty}^\infty \sum_{i=j-4}^{j+4} \int \abs{2^{(l-\mu)j} N_l \laps{\mu-l} a_{j}}\ \abs{2^{\mu i} b_{i}}\\
 &\overset{\eqref{eq:commis:fhochjBinftyinfyest}}{\aleq}& \vrac{\psi}_{\dot{B}^0_{\infty,\infty}}\ \sum_{l=1}^\infty c_l\  2^{-6l} \sum_{j = -\infty}^\infty \sum_{i=j-4}^{j+4} \int \abs{2^{(l-\mu)j} M_l \laps{\mu-l} a_{j}}\ \abs{2^{\mu i} b_{i}}\\
 &\aleq& \vrac{\psi}_{\dot{B}^0_{\infty,\infty}}\ \vrac{a}_2\ \vrac{\laps{\mu} b}_2.
\end{ma}
\]
For $II$, where the expansion \eqref{eq:commis:pi3htilde:kexp} might not be convergent, the situation is even easier,
\[
\begin{ma}
 \abs{II} &\aleq& \sum_{j = -\infty}^\infty \sum_{i=j-4}^{j+4} \sum_{k = j-6}^{j+6} \int \brac{\abs{2^{-\mu j} \laps{\mu}a_{j}}\ \abs{2^{\mu j} b_{i}} + \abs{a_{j}}\  \abs{\laps{\mu}b_{i}}} \abs{\psi_{k}}\\
 &\aleq& \vrac{\psi}_{\dot{B}^0_{\infty,\infty}} 
 \brac{ 
 \brac{ \int \sum_j \abs{2^{-\mu j} \laps{\mu}a_{j}}^2}^{\fracm{2}}\ \brac{\int \sum_j \abs{2^{\mu j} b_{i}}^2}^{\fracm{2}} 
 + \brac{\int \abs{a_{j}}^2}^{\fracm{2}}\  \brac{\int \abs{\laps{\mu}b_{i}}^2}^{\fracm{2}}
 } 
 \\
&\aleq& \vrac{\psi}_{\dot{B}^0_{\infty,\infty}} \vrac{a}_2\ \vrac{\laps{\mu} b}_2.
\end{ma}
\]

\end{proof}

\begin{proof}[Proof of \eqref{eq:c2:laphH1inhardy} for $\mu = 1$]
We have (by the classical product rule, or equivalently expanding the symbols $\abs{\xi}^2 = \abs{\xi-\eta}^2 + \abs{\eta}^2 + 2 \xi \cdot (\xi-\eta)$). With $\Rz_k$ we will by a abuse of notation call the linear operators with symbol $\xi^k/\abs{\xi}$. By the classical product rule for $\Rz_i \laps{1} = c \partial_i$
\[
\begin{ma}
 \Rz_i H_{1}(a,b) &=& a \Rz_i \laps{1} b + b \Rz_i \laps{1} a - \Rz_i(a \laps{1} b) - \Rz_i(b \laps{1} a)\\
&=& \lapms{1} (\laps{1} a)\ \Rz_i \laps{1} b + \lapms{1} (\laps{1} b) \Rz_i \laps{1} a - \Rz_i(\lapms{1} (\laps{1} a)\ \laps{1} b) - \Rz_i(\lapms{1} (\laps{1} b)\ \laps{1} a)\\
&=& \lapms{1} (\laps{1} a)\ \Rz_i \laps{1} b - \Rz_i(\lapms{1} (\laps{1} a)\ \laps{1} b)\\
&& +  \lapms{1} (\laps{1} b) \Rz_i \laps{1} a  - \Rz_i(\lapms{1} (\laps{1} b)\ \laps{1} a)\\
\end{ma}
\]
Thus, $\Rz_i \laps{1}  H_{1}(a,b)$ can be estimated via \eqref{eq:c2:hardyguy}, and as this holds for any $i \in \N$, we have \eqref{eq:c2:laphH1inhardy} for $\mu = 1$.
\end{proof}

\newpage
\newcommand{\antisymm}{\omega}

\section{Energy approach for optimal frame: Proof of Theorem~\ref{th:intro:energy}}
In this section we construct a suitable frame $P$ for our equation, transforming the \emph{antisymmetric} (essentially) $L^2$-potential $\Omega[]$ into an $L^{2,1}$- or even better in an $\lapms{\mu} \mathcal{H}$-potential $\Omega^P[]$. Here, $\mathcal{H}$ is the Hardy space, and with the previous statement we essentially mean that
\begin{equation}\label{eq:energymot:improvedOmega}
 \int \Omega^P[f] \leq C_{\Omega^P}\ \vrac{f}_{(2,\infty)}, \quad \mbox{or}\quad \int \Omega^P[\laps{\mu} \varphi] \leq C_{\Omega^P}\ \vrac{\varphi}_{BMO}, \quad \mbox{respectively,}
\end{equation}
where $BMO$ is the space dual to $\mathcal{H}$. This is an improvement, since for the non-transformed $\Omega$, we only had the estimate
\begin{equation}\label{eq:energymot:Omegawithouttransform}
 \int \Omega[f] \leq C_{\Omega}\ \vrac{f}_{2}.
\end{equation}
For motivation of the arguments presented here, let us recall the classical setting \cite{Riv06}, where we have the equation (usually for $w^i := \nabla u^i \in L^2(\R^m,\R^2)$) 
\[
 -\operatorname{div} (w^i) = \tilde{\Omega}_{ik} \cdot w^l,
\]
for $\tilde{\Omega}_{ik} = - \tilde{\Omega}_{ki} \in L^2(\R^m,\R^2)$, and we look for an orthogonal transformation $P \in W^{1,2}(\R^m,SO(N))$, $SO(N) \subset \R^{N \times N}$ being the special orthogonal group, such that
\begin{equation}\label{eq:energymot:improvedtOmega}
 \int \tilde{\Omega}^P_{ik} \cdot \nabla \varphi = 0,
\end{equation}
where
\[
 \tilde{\Omega}^P_{ij} = P_{ik} \nabla P^T_{kj} + P_{ik} \tilde{\Omega}_{kl} P^T_{lj}, \quad \mbox{or equivalently,}\quad -\operatorname{div} (P_{il} w^l) = \tilde{\Omega}^P_{ik} \cdot P_{kl} w^l.
\]
Also in this case, the estimate \eqref{eq:energymot:improvedtOmega} is an improvement from the estimate for the non-transformed $\tilde{\Omega}$
\[
 \int \tilde{\Omega} \cdot \nabla \varphi \leq C_{\tilde{\Omega}}\ \vrac{\nabla \varphi}_2,
\]
philosophically similar to the improvement \eqref{eq:energymot:improvedOmega} from the starting point \eqref{eq:energymot:Omegawithouttransform}. 

For the construction of $P$ such that \eqref{eq:energymot:improvedtOmega} holds, there are two different arguments known: Rivi\`{e}re \cite{Riv06} adapted a result by Uhlenbeck \cite{Uhlenbeck82} which is based on the continuity method (for the set $t\Omega$, $t \in [0,1]$) and relies on non-elementary a-priori estimates for $\tilde{\Omega}^P$ which also needs $L^2$-smallness of $\tilde{\Omega}$. In \cite{IchEnergie} we proposed to use arguments from H\'{e}lein's moving frame method \cite{Hel91}: Then the construction of $P$ relies on the fact that \eqref{eq:energymot:improvedtOmega} is the Euler-Lagrange equation of the energy
\begin{equation}\label{eq:energymot:classicalOmegaPenergy}
 \tilde{E}(Q) := \vrac{\tilde{\Omega}^Q}_{2}^2, \quad \mbox{$Q \in SO(N)$, a.e.},
\end{equation}
the minimizer of which exists by the elementary direct method.

Both construction arguments have been generalized to the fractional setting for $\Omega[] \equiv \Omega\cdot$ a pointwise multiplication-operator \cite{DR1dMan, Sfracenergy}. In our situation, where $\Omega[]$ is allowed to be a linear bounded operator from $L^2$ to $L^1$, we adapt the argument in \cite{Hel91,IchEnergie,Sfracenergy}, and minimize essentially the energy
\[
 E(Q) := \sup_{\psi \in L^2} \int \Omega^Q[\psi], \quad \mbox{$Q \in SO(N)$, a.e.}.
\]
While the construction of a minimizer of $E$, see Lemma~\ref{la:energy:minexists}, is not much more difficult as in the earlier situations \cite{Hel91,IchEnergie,Sfracenergy}, when computing the Euler-Lagrange equations, see Lemma~\ref{la:energy:Pmineuler}, we have several error terms, which stem from commutators of the form $f \Omega [g] - \Omega[fg]$, which are trivial if $\Omega[]$ is a pointwise-multiplication operator $\Omega[] = \Omega \cdot$. In Lemma~\ref{la:betterOmegaP} we then show that these error terms all behave well enough, if we take the for us relevant case of $\Omega[]$ being of the form $A \Rz[]$.

\subsection{Preliminary propositions}
Here we recall some elementary statements, which enter the proof of Theorem~\ref{th:intro:energy}. Proposition~\ref{pr:fintsuppoutsnormcontrolled} and Proposition~\ref{pr:energy:rieszimpliesexdual} are simple duality arguments for linear, bounded mappings between Banach spaces. Proposition~\ref{pr:L1BMOcompactsupport} is a quantified embedding from $L^1$ into $BMO$.

\begin{proposition}\label{pr:fintsuppoutsnormcontrolled}
For any $s > 0$ there exists $\Lambda_0, C_s > 1$ such that the following holds: Let $f \in L^2(\R^m)$, $\laps{s} f \in L^2(\R^m)$ and assume $f \equiv 0$ on $\R^m \backslash B_r$ for some $B_r \subset \R^m$. Then for any $\Lambda \geq \Lambda_0$,
\[
 \Vrac{\laps{s} f}_{2,\R^m \backslash B_{\Lambda r}} \leq C_s\ \Lambda^{-\frac{m}{2}-s}\ \Vert \laps{s} f \Vert_{2,B_{\Lambda r}}
\]
\end{proposition}
\begin{proof}
Using Corollary \ref{co:QuasiLocalityII},
\[
 \Vrac{\laps{s} f}_{2,\R^m \backslash B_{\Lambda r}} 
 \aleq \Lambda^{-\frac{m}{2} -s}\ \Vert \laps{s} f \Vert_{2,\R^m \backslash B_{\Lambda r}} + \Lambda^{-\frac{m}{2} -s}\ \Vert \laps{s} f \Vert_{2,B_{\Lambda r}}.
\]
Thus, if $\Lambda > \Lambda_0$ for a $\Lambda_0$ depending only on $s$, we can absorb and conclude.
\end{proof}

\begin{proposition}\label{pr:energy:rieszimpliesexdual}
Let $A: L^2(\R^m) \to L^1(\R^m)$ be a linear, bounded operator. Then there exists $\bar{g} \in L^2(\R^m)$, $\Vert \bar{g} \Vert_{2,\R^m} = 1$ such that
\[
 \sup_{\Vert \psi \Vert_{2,\R^m} \leq 1} \int A[\psi] = \int A[\bar{g}].
\]
In particular (taking instead of $A$ the operator $\tilde{A} := A[\chi_D \cdot]$, for any $D \subset \R^m$ there exists  $\bar{g}_D \in L^2(D)$, $\Vert \bar{g}_D \Vert_{2,D} \leq 1$, $\supp \bar{g} \subset \overline{D}$, such that
\[
 \sup_{\Vert \psi \Vert_{2,\R^m} \leq 1, \supp \psi \subset \overline{D}} \int A[\psi] = \int A[\bar{g}_D].
\]
\end{proposition}
\begin{proof}
As $f^\ast \in \brac{L^2(\R^m)}^\ast$ for
\[
 f^\ast(\psi) := \int A[\psi]
\]
is a linear bounded functional, there exists a unique $\bar{f} \in L^2(\R^m)$ such that
\begin{equation}\label{eq:rieszdual:gf}
 \int A[\psi] = \int \bar{f} \psi,
\end{equation}
and
\[
\Vert \bar{f} \Vert_{2,\R^m} =  \sup_{\Vert \psi \Vert_{2,\R^m} \leq 1} \int A[\psi].
\]
On the other hand,
\[
 \Vert \bar{f} \Vert_{2,\R^m}^2 = \int \bar{f}\ \bar{f} \overset{\eqref{eq:rieszdual:gf}}{=} \int A[\bar{f}],
\]
which in turn implies that for $\bar{g} := \Vrac{ \bar{f} }_{2}^{-1}\ \bar{f}$,
\[
 \sup_{\Vert \psi \Vert_{2,\R^m} \leq 1} \int A[\psi] = \Vert \bar{f} \Vert_{2,\R^m} = \int A[\bar{g}].
\]
\end{proof}

\begin{proposition}\label{pr:energy:linfty}
Let $A: L^2(\R^m) \to L^1(\R^m)$ be a linear, bounded operator. Then there exists a linear, bounded operator $A^\ast: L^\infty(\R^m) \to L^2(\R^m)$ such that
\[
\int g\ A[f] = \int f\ A^\ast [g] \quad \mbox{for any $f \in L^2(\R^m)$, $g \in L^\infty(\R^m)$} .
\]
Moreover, $\bar{g} = \Vert A(1) \Vert_{2}^{-1}\ A^\ast(1)$ for the $\bar{g}$ from Proposition~\ref{pr:energy:rieszimpliesexdual}.
\end{proposition}
\begin{proof}
For any $g \in L^\infty(\R^m)$, we have
\[
 T_g[\cdot] := \int g\ A[\cdot] \in \brac{L^2(\R^m)}^\ast,
\]
thus there exists a unique $A^\ast_g \in L^2(\R^m)$ such that
\[
 T_g[f] = \int A^\ast_g\ f \quad \mbox{for all $f \in L^2(\R^m)$}.
\]
For $\lambda_1,\lambda_2 \in \R$ and $g_1,g_2 \in L^\infty$, and any $f \in L^2(\R^m)$
\[
 \int A^\ast_{\lambda_1 g_1+\lambda_2 g_2}\ f = \lambda_1 \int g_1\ A[f] + \lambda_2 \int g_2\ A[f] = \lambda_1 \int A_{g_1}\ f + \lambda_2 \int A_{g_2}\ f,
\]
which implies that $A^\ast[\cdot] := A^\ast_{\cdot}$ is linear. As for boundedness, note that
\[
 \Vert A^\ast[g] \Vert_{2,\R^m} = \sup_{\Vert f \Vert_{2,\R^m} \leq 1} \int A^\ast[g] f = \sup_{\Vert f \Vert_{2,\R^m} \leq 1} \int g A[f] \leq \Vert g \Vert_{\infty}\ \Vert A \Vert_{L^2 \to L^1},
\]
that is
\[
 \Vert A^\ast \Vert_{L^\infty \to L^2} \leq \Vert A \Vert_{L^2 \to L^1}.
\]
Finally, we have that
\[
\int \bar{f}\ \psi \overset{\eqref{eq:rieszdual:gf}}{=} \int A[\psi]\ = \int A^\ast[1]\ \psi,
\]
which - given that it holds for any $\psi \in L^2(\R^m)$ - implies $\bar{f} = A^\ast[1]$.
\end{proof}

\begin{proposition}\label{pr:L1BMOcompactsupport}
Let $\varphi \in C_0^\infty(B_r)$, then
\[
 \vrac{\varphi}_1 \leq C_m\ r^m\ [\varphi]_{BMO}.
\]
\end{proposition}
\begin{proof}
Let $\Lambda > 0$, then
\[
 \vrac{\varphi}_{1,B_r} \leq \vrac{\varphi - \abs{B_{\Lambda r}}^{-1} \int \varphi }_{B_r} + \frac{\abs{B_r}}{B_{\Lambda r}} \vrac{\varphi}_{1,B_r},
\]
and consequently for, say, $\Lambda = 2$,
\[
 \vrac{\varphi}_{1,B_r} \leq 2 \vrac{\varphi - \abs{B_{\Lambda r}}^{-1} \int \varphi }_{B_r} \aleq (\Lambda r)^m\ \abs{B_{\Lambda r}}^{-1} \vrac{\varphi - \abs{B_{\Lambda r}}^{-1} \int \varphi }_{B_{\Lambda r}} \aleq (\Lambda r)^m\ [\varphi]_{BMO}.
\]

\end{proof}

\subsection{Energy with potentials}
\label{s:potentialenergy}
Let $\Omega^{i,j}: L^2(\R^m) \to L^1(\R^m)$, $1 \leq i,j \leq N$ be a linear bounded Operator. And set
\[
 \Omega^Q_{ij} [f]:= \brac{\laps{\mu} (Q-I)_{ik}}\ Q^T_{kj}\ f + Q_{ik} \Omega_{kl} [Q^T_{lj} f],
\]
for $\supp(Q-I) \subset B_r$, $\laps{\mu} Q \in L^2(\R^{N\times N})$ and $Q \in SO(N)$ almost everywhere. For $\psi : \R^n \to \R^{N\times N}$, we write
\[
 \Omega^Q [\psi] := \brac{\laps{\mu} (Q-I)_{ik}}\ Q^T_{kj}\ \psi_{ij} + Q_{ik} \Omega_{kl} [Q^T_{lj} \psi_{ij}],
\]
Having in mind \eqref{eq:energymot:classicalOmegaPenergy}, we then define the energy
\begin{equation}\label{eq:energy:EQ}
 E(Q) \equiv E_{r,x,\Lambda,s,2} (Q) := \begin{cases}
\sup_{\ontop{\psi \in C_0^\infty(B_{\Lambda r}(x),\R^{N\times N})}{\Vert \psi \Vert_{2} \leq 1}} \intl_{\R^m} (\Omega^Q) [\psi] \quad &\mbox{if $\supp (Q-I) \subset \overline{B_r(x)}$},\\
\infty \quad &\mbox{else.}
\end{cases}
\end{equation}
Obviously, $Q \equiv I$ is admissible and $E(I) < \infty$. Thus there exists a minimizing sequence, and one can hope for a minimizer:
\begin{lemma}\label{la:energy:minexists}
For any $\mu > 0$ there exists $\Lambda_0 > 1$ such that for any $\Lambda \geq \Lambda_0$, the following holds: There exists an admissible function $P$ for $E$ such that $E(P) \leq E(Q)$ for any other admissible function $Q$. Moreover,
\begin{equation}\label{eq:energy:minex:Pest}
\Vrac{\laps{\mu} P}_{2,B_{\Lambda r}(x)} + \Lambda^{\frac{m}{2}+\mu} \Vrac{\laps{\mu} P}_{2,\R^m \backslash B_{\Lambda r}(x)} \leq C_\mu\ \Vert \Omega \Vert_{2 \to 1,B_{\Lambda r}(x)}.
\end{equation}
Here,
\[
\Vert \Omega \Vert_{2 \to 1,D} := \sup_{\psi \in C_0^\infty(D,\R^{N\times N}), \vrac{\psi}_{2}\leq 1}\vrac{\Omega[\psi]}_1
\]
\end{lemma}
\begin{proof}
Take $\Lambda_0$ from Proposition~\ref{pr:fintsuppoutsnormcontrolled} and assume $\Lambda \geq \Lambda_0$. We have for any $\psi \in C_0^\infty(B_{\Lambda r},\R^{N\times N})$, $\Vert \psi \Vert_{2} \leq 1$
\[
\begin{ma}
 E(Q) &\geq& \int (\laps{\mu} (Q-I)\ Q^T)_{ij} \psi_{ij} + \int Q \Omega [Q^T \psi]\\
&\geq& \int (\laps{\mu} (Q-I)\ Q^T)_{ij}\ \psi_{ij} - \Vert \Omega \Vert_{2\to 1,B_{\Lambda r}},
\end{ma}
\]
which (taking the supremum over such $\psi$) implies
\[
\Vert \laps{\mu} (Q-I) \Vert_{2,B_{\Lambda r}} \leq E(Q) + \Vert \Omega \Vert_{2\to 1,B_{\Lambda r}}.
\]
According to Proposition~\ref{pr:fintsuppoutsnormcontrolled}, this implies (as $Q \equiv I$ on $\R^n \backslash B_r$),
\[
 \Vert \laps{\mu} (Q-I) \Vert_{2,\R^m} \leq C_\mu\ \brac{E(Q) + \Vert \Omega \Vert_{2\to 1,B_{\Lambda r}}}.
\]
Consequently, for a minimizing sequence $P_k$,
\[
  \Vert \laps{\mu} (P_k-I) \Vert_{2,\R^m} \leq C_\mu\ \Vert \Omega \Vert_{2\to 1,B_{\Lambda r}},
\]
and up to taking a subsequence, we may assume that there is an admissible function $P$ such that $\laps{\mu} P_k$ converges $L^2$-weakly to $\laps{\mu} P$ and $P_k$ converges pointwise and strongly to $P$. 

Then, for any fixed $\psi \in C_0^\infty(B_{\Lambda r})$, $\Vert \psi \Vert_{2,\R^{N\times N}} \leq 1$
\[
 E(P_k) \geq \int \Omega^P [\psi] + \int \Omega^{P_k} [\psi] - \Omega^P [\psi].
\]
We claim that
\begin{equation}\label{eq:omegapkpsito0}
 \int \Omega^{P_k} [\psi] - \Omega^P [\psi] \xrightarrow{k \to \infty} 0,
\end{equation}
which, once proven, implies that
\[
 \inf E(\cdot) \geq \int \Omega^{P} \psi,
\]
which by the arbitrary choice of $\psi$ implies that $P$ is a minimizer. In order to show \eqref{eq:omegapkpsito0}, note that
\[
\begin{ma}
 \Omega^{P_k} [\psi] - \Omega^P [\psi] &=& \laps{\mu} P_k\ (P_k^T - P^T) \psi + \laps{\mu} (P_k-P)\ P^T\psi +
(P_k -P)\Omega \left [P_k^T \psi\right ] + P \Omega \left [(P_k^T - P^T) \psi\right ]\\
&=:& I_k + II_k + III_k+ IV_k.
\end{ma}
\]
Since $\abs{P_k}$, $\abs{P} \leq 1$, all terms of the form $(P_k^T - P^T) \psi \xrightarrow{k \to \infty} 0$ in $L^2$, by Lebesgue's dominated convergence. Thus, $\int I_k + \int IV_k \xrightarrow{k \to \infty} 0$. By the weak $L^2$-convergence of $\laps{\mu} P_k$, also $\int II_k \xrightarrow{k \to \infty} 0$. Since $P_k^T \psi \to P^T\psi$ in $L^2(\R^m)$, also $\Omega \left [P_k^T \psi \right ] \xrightarrow {k \to \infty} \Omega [P^T \psi]$ in $L^1$ and in particular pointwise almost everywhere. Then also $\int III_k \xrightarrow{k \to \infty} 0$.
\end{proof}
\begin{lemma}\label{la:energy:Pmineuler}
Let $P$ be a minimizer of $E(\cdot)$ as in \eqref{eq:energy:EQ}, and assume that \begin{equation}\label{eq:energy:Omegaantisym}
                                                                                  \Omega_{ij}[] = - \Omega_{ji}[] \quad 1 \leq i,j \leq N.
                                                                                 \end{equation}
 Then for any $\varphi \in C_0^\infty(B_r(x))$, \[
\begin{ma}
- \int \Omega^P[\laps{\mu} \varphi]
 &=& \fracm{2} \int H_{\mu}(P-I,P^T-I)\ \laps{\mu} \varphi\\
 &&-\int so(P\mathcal{C}\brac{\varphi, \Omega}\left [P^T \overline{\Omega^P}^T \chi_{D_\Lambda}\right ])  \\
&& + \int so(\overline{\Omega^P} \chi_{D_\Lambda}P H_{\mu}(\varphi,P^T-I))\\
&& - \int so(\mathcal{C}\brac{P,\Omega}[\laps{\mu} \varphi] P^T)\\
&& + \int \Omega^P[(1-\chi_{D_\Lambda})\laps{\mu} \varphi].
\end{ma}
\]
Here, we denote for a matrix $A \in R^{N\times N}$, the antisymmetric part with $so(A) = 2^{-1} (A - A^T)$, and for a mapping $g: L^2 \to L^1$, we denote $\overline{g}$ as in Proposition~\ref{pr:energy:rieszimpliesexdual}.
\end{lemma}
\begin{proof}
We set $D := B_r(x)$ and $D_\Lambda := B_{\Lambda r}(x)$. Let $\varphi \in C_0^\infty(D)$, $\antisymm \in so(N)$. We distort the minimizer $P$ of $E(\cdot)$ by
\[
Q_\varepsilon := e^{\varepsilon \varphi \antisymm} P = P + \varepsilon \varphi\ \antisymm\ P + o(\varepsilon) \in H^{\frac{n}{2}}_I(D,SO(N)),
\]
that is we know that
\begin{equation}\label{eq:energy:EQepsgeqEP}
 E(Q_\varepsilon) - E(P) \geq 0
\end{equation}
We compute
\begin{equation}\label{eq:en:distqlq}
\begin{ma}
&&\laps{\mu} (Q_\varepsilon-I)\ Q^T\\
&=& \laps{\mu} (P-I)\ P^T + \varepsilon \varphi \brac{\antisymm\ \laps{\mu} (P-I)\ P^T - \laps{\mu} (P-I)\ P^T\ \antisymm} + \varepsilon \laps{\mu} \varphi\ \antisymm + \varepsilon\ \antisymm\ H_{\mu}(\varphi,P-I)P^T + o(\varepsilon),
\end{ma}
\end{equation}
and
\begin{equation} \label{eq:en:distqoq}
Q_\varepsilon \Omega \left [Q^T_\varepsilon \cdot \right]= P \Omega \left [P^T \cdot \right]+ \varepsilon \brac{\varphi\ \antisymm\ P \Omega \left [P^T \cdot \right ] - P \Omega \left [P^T\ \antisymm \varphi \cdot \right ]}+ o(\varepsilon).
\end{equation}
Together, we infer from \eqref{eq:en:distqlq} and \eqref{eq:en:distqoq} (denoting the Hilbert-Schmidt matrix product $A:B := A_{ij}B_{ij}$)
\[
\Omega^{Q_\varepsilon}[\psi] = \Omega^P[\psi] + \varepsilon \brac{\varphi \antisymm\ \Omega^P [\psi] - \Omega^P\ [\antisymm \psi \varphi]} + \varepsilon \laps{\mu} \varphi\ \antisymm: \psi + \varepsilon\ \antisymm \ H_{\mu}(\varphi,P-I)\ P^T: \psi + o(\varepsilon)[\psi].
\]
Thus, for any $\varepsilon > 0$, $\psi \in C_0^\infty(D_\Lambda,\R^{N\times N})$, $\Vrac{\psi}_{2} \leq 1$,
\[
\begin{ma}
 \fracm{\varepsilon}\brac{E(Q_\varepsilon) - E(P)} &\geq& \fracm{\varepsilon} \brac{\int \Omega^P [\psi] - E(P)}\\
&& + \int \brac{\varphi \antisymm\ \Omega^P [\psi] - \Omega^P\ [\antisymm \psi \varphi]}\\
&& + \int \laps{\mu} \varphi\ \antisymm :\psi \\
&& + \int \antisymm\ H_{\mu}(\varphi,P-I)\ P^T:\psi\\
&& + o(1).
\end{ma}
\]
Let $\overline{\psi} \in L^2(D_\Lambda)$ such that $E(P) = \int \Omega^P [\overline{\psi}]$ (cf. Proposition~\ref{pr:energy:rieszimpliesexdual}), this implies for the choice $\psi := \overline{\psi}$
\[
\begin{ma}
 0 \overset{\eqref{eq:energy:EQepsgeqEP}}{\geq} \fracm{\varepsilon}\brac{E(Q_\varepsilon) - E(P)} &\geq& \int \brac{\varphi \antisymm\ \Omega^P [\overline{\psi}] - \Omega^P\ [\antisymm \overline{\psi} \varphi]}\\
&&+ \int \laps{\mu} \varphi\ \antisymm :\overline{\psi}\\
&& + \int \antisymm\ H_{\mu}(\varphi,P-I)\ P^T:\overline{\psi}\\
&& + o(1).
\end{ma}
\]
Letting $\varepsilon \to 0$, we then have
\[
\begin{ma}
 -\int \laps{\mu} \varphi\ \antisymm :\overline{\psi}&\geq& \int \varphi \antisymm\ \Omega^P [\overline{\psi}] - \Omega^P\ [\antisymm \overline{\psi} \varphi]\\
&& + \int \antisymm\ H_{\mu}(\varphi,P-I)\ P^T:\overline{\psi}
\end{ma}
\]
which holds for any $\varphi \in C_0^\infty(B_r)$. Replacing $\varphi$ by $-\varphi$, we arrive at
\begin{equation}\label{eq:energy:firstel}
 -\int \laps{\mu} \varphi\ \antisymm :\overline{\psi} = \int \varphi \antisymm\ \Omega^P [\overline{\psi}] - \Omega^P\ [\antisymm \overline{\psi}\varphi] + \int \antisymm\ H_{\mu}(\varphi,P-I)\ P^T:\overline{\psi}.
\end{equation}
Now we need to be more specific about the characteristics of $\overline{\psi}$. We have
\[
 E(P) = \sup_{\psi} \int \Omega^P [\psi] = \sup_{\psi} \int \laps{\mu} P_{ik}\ P_{kj}^T\ \psi_{ij} + P_{ik} \Omega_{kl} [ P^T_{lj} \psi_{ij}].
\]
Let $\Omega_{kl}^\ast : L^\infty(\R^m) \to L^2(\R^m)$ be the linear bounded operator such that (cf. Proposition~\ref{pr:energy:linfty})
\[
 \intl_{\R^m} g \Omega_{kl}[f] = \intl_{\R^m} \Omega_{kl}^\ast[g]\ f, \quad \mbox{for any $f \in L^2(\R^m)$, $g \in L^\infty(\R^m)$}.
\]
Set then,
\[
 \brac{\brac{\Omega^P}^\ast}_{ij}[f] := \laps{\mu}P_{ik}\ P_{kj}^Tf + \Omega_{kl}^\ast \left [f P_{ik} \right ]\   P^T_{lj} \in L^2(\R^m),
\]
and
\[
 \brac{\overline{\Omega^P}}_{ij} := \brac{\brac{\Omega^P}^\ast}_{ij}[1] \in L^2(\R^m).
\]
Since
\[
 \int g\ \brac{\Omega^P}_{ij} [f] = \int \brac{\brac{\Omega^P}^\ast}_{ij}[g]\ f \quad \mbox{for all $f \in L^2(\R^m)$, $g \in L^\infty(\R^m)$,}
\]
we have
\[
 E(P) = \sup_{\psi} \int \overline{\Omega^P}:\psi\ \chi_{D_\Lambda} = c \int \overline{\Omega^P}:\overline{\Omega^P}\ \chi_{D_\Lambda} = c \int \Omega^P[\overline{\Omega^P}\ \chi_{D_\Lambda}],
\]
for some normalizing constant $c$. That is,
\[
 \brac{E(P)}^2 = \int \brac{\Omega^P}_{ij} [\chi_{D_\Lambda} \overline{\Omega^P}_{ij}],
\]
and we can assume $\overline{\psi} = c \chi_{D_\Lambda} \overline{\Omega^P} = c \chi_{D_\Lambda} \overline{\Omega^P}$ for some normalizing constant $c$. Now,
\[
 -\int \laps{\mu} \varphi\ \antisymm :\overline{\psi} = -c\ \int \laps{\mu} \varphi\ \antisymm_{ij}  \chi_{D_\Lambda} \overline{\Omega^P_{ij}}
 = -\antisymm_{ij}\ \int \Omega^P_{ij}[\laps{\mu} \varphi] + \int \omega: \Omega^P[(1-\chi_{D_\Lambda})\laps{\mu} \varphi].
\]

Consequently, \eqref{eq:energy:firstel} reads as
\[
\begin{ma}
 -\int \antisymm :\Omega^P[\laps{\mu} \varphi] &=& \int \varphi \antisymm_{ik}\ \brac{\Omega^P}_{kj} [\brac{\overline{\Omega^P}}_{ij}\chi_{D_\Lambda}] - \brac{\Omega^P}_{ij}\ [\antisymm_{ik} \brac{\overline{\Omega^P}}_{kj} \varphi]\\
&& + \intl_{D_{\Lambda}} \antisymm\ H_{\mu}(\varphi,P-I)\ P^T:\overline{\Omega^P}\\
&& + \int \omega: \Omega^P[(1-\chi_{D_\Lambda})\laps{\mu} \varphi].
\end{ma}
\]
Note that, since $\varphi \in C_0^\infty(\R^m) \subset L^\infty$,
\begin{equation}\label{eq:energy:alphaikomegaomegabar}
 \antisymm_{ik}\ \int \brac{\Omega^P}_{ij}\ [\brac{\overline{\Omega^P}}_{kj} \varphi] = \antisymm_{ik}\ \int \brac{\overline{\Omega^P}}_{ij}\ \brac{\overline{\Omega^P}}_{kj} \varphi \overset{\antisymm \in so}{=} 0.
\end{equation}
By the same argument,
\[
\antisymm_{ik}\ \int \varphi \ \brac{\Omega^P}_{kj} [\chi_{D_\Lambda}\brac{\overline{\Omega^P}}_{ij}]
 = \antisymm_{ik}\ \int \varphi \ \brac{\Omega^P}_{kj} [\chi_{D_\Lambda}\brac{\overline{\Omega^P}}_{ij}] - \antisymm_{ik}\ \int \brac{\Omega^P}_{kj} [\brac{\overline{\Omega^P}}_{ij}\varphi]\\
\]
and
\[
\begin{ma}
 &&\antisymm_{ik}\ \int \varphi \ \brac{\Omega^P}_{kj} [\brac{\overline{\Omega^P}}_{ij}\chi_{D_\Lambda}]\\
&=&\antisymm_{ik}\ \int \brac{\brac{\Omega^P}^\ast}_{kj}[\varphi]\ \brac{\overline{\Omega^P}}_{ij}\ \chi_{D_\Lambda}\\
&=&\antisymm_{ik}\ \int \varphi \brac{\brac{\Omega^P}^\ast}_{kj}[1]\ \brac{\overline{\Omega^P}}_{ij}\ \chi_{D_\Lambda} - \antisymm_{ik}\ \int \brac{\varphi \brac{\brac{\Omega^P}^\ast}_{kj}[1]-\brac{\brac{\Omega^P}^\ast}_{kj}[\varphi]}\ \brac{\overline{\Omega^P}}_{ij}\ \chi_{D_\Lambda}\\
&\overset{\supp \varphi}{=}&\antisymm_{ik}\ \int \varphi \brac{\brac{\Omega^P}^\ast}_{kj}[1]\ \brac{\overline{\Omega^P}}_{ij} - \antisymm_{ik}\ \int \brac{\varphi \brac{\brac{\Omega^P}^\ast}_{kj}[1]-\brac{\brac{\Omega^P}^\ast}_{kj}[\varphi]}\ \brac{\overline{\Omega^P}}_{ij}\ \chi_{D_\Lambda}\\
&\overset{\eqref{eq:energy:alphaikomegaomegabar}}{=}& 0 - \antisymm_{ik}\ \int \brac{\varphi \brac{\brac{\Omega^P}^\ast}_{kj}[1]-\brac{\brac{\Omega^P}^\ast}_{kj}[\varphi]}\ \brac{\overline{\Omega^P}}_{ij}\ \chi_{D_\Lambda}\\
&=&   \antisymm_{ik}\ \int \mathcal{C} (\varphi, \Omega^P_{kj})[(\overline{\Omega^P})_{ij}\ \chi_{D_\Lambda}] ,
\end{ma}
\]
where we denote the commutator $\mathcal{C}$
\[
 \mathcal{C}(b,T)[f] := b\ Tf - T(bf).
\]
Thus, we arrive at
\[
 \begin{ma}
 -\int \antisymm: so(\Omega^P[\laps{\mu} \varphi])_{ij} &=&  \antisymm_{ik}\ \int \mathcal{C}\brac{\varphi, \brac{\Omega^P}_{kj}}[\brac{\overline{\Omega^P}}_{ij}\chi_{D_\Lambda}]\\
&& + \int \antisymm\ H_{\mu}(\varphi,P-I)\ P^T:\overline{\Omega^P} \chi_{D_\Lambda}\\
&& + \int \omega: \Omega^P[(1-\chi_{D_\Lambda})\laps{\mu} \varphi].
\end{ma}
\]
One checks, that
\[
 \mathcal{C}\brac{\varphi, \brac{\Omega^P}_{kj}}[\brac{\overline{\Omega^P}}_{ij} \chi_{D_\Lambda}]
= P_{kl} \mathcal{C}\brac{\varphi,\Omega^{ls}}[P^T_{sj} \brac{\overline{\Omega^P}}_{ij}\chi_{D_\Lambda}]
\]
Next, (and here the antisymmetry of $\Omega$, \eqref{eq:energy:Omegaantisym}, plays its role)
\[
\begin{ma}
 so(\Omega^P[\laps{\mu} \varphi])_{ij}
&=& so(\laps{\mu}(P-I)\ P^T)_{ij}\ \laps{\mu} \varphi+ \fracm{2} P_{ik} \Omega_{kl}[P_{jl} \laps{\mu} \varphi] - \fracm{2} P_{jk} \Omega_{kl}[P_{il} \laps{\mu} \varphi]\\
&\overset{\eqref{eq:energy:Omegaantisym}}{=}& so(\laps{\mu}(P-I)\ P^T)_{ij}\ \laps{\mu} \varphi+ \fracm{2} P_{ik} \Omega_{kl}[P_{jl} \laps{\mu} \varphi] + \fracm{2} P_{jl} \Omega_{kl}[P_{ik} \laps{\mu} \varphi]\\
&=& so(\laps{\mu}(P-I)\ P^T)_{ij}\ \laps{\mu} \varphi+ \fracm{2} P_{ik} \Omega_{kl}[P_{jl} \laps{\mu} \varphi] + \fracm{2} P_{jl} P_{ik} \Omega_{kl}[\laps{\mu} \varphi]\\
&&-\fracm{2} P_{jl} \mathcal{C}\brac{P_{ik}, \Omega_{kl}}[\laps{\mu} \varphi]\\
&=& so(\laps{\mu}(P-I)\ P^T)_{ij}\ \laps{\mu} \varphi
+ P_{ik} \Omega_{kl}[P_{jl} \laps{\mu} \varphi]\\
&&+ \fracm{2} P_{ik} \mathcal{C}\brac{P_{jl},\Omega_{kl}}[\laps{\mu} \varphi]
-\fracm{2} P_{jl} \mathcal{C}\brac{P_{ik}, \Omega_{kl}}[\laps{\mu} \varphi],
 + \int \Omega^P[(1-\chi_{D_\Lambda})\laps{\mu} \varphi].
\end{ma}
\]
and
\[
\begin{ma}
 so(\laps{\mu}(P-I)\ P^T) &=& \fracm{2} \laps{\mu}(P-I)\ P^T - \fracm{2} P \laps{\mu}(P^T-I)\\
&=& \laps{\mu}(P-I)\ P^T + \fracm{2} \brac{-\laps{\mu}(P-I)\ P^T - P \laps{\mu}(P^T-I)}\\
&=& \laps{\mu}(P-I)\ P^T + \fracm{2} \brac{\laps{\mu} (PP^T) - \laps{\mu}(P-I)\ P^T - P \laps{\mu}(P^T-I)}\\
&=& \laps{\mu}(P-I)\ P^T + \fracm{2} H_{\mu}(P-I,P^T-I)\\
\end{ma}
\]
This implies finally (going with $\omega_{ij} \in \{-1,0,1\}$ through all the possible matrices with two non-zero entries)
\[
\begin{ma}
- \int \Omega^P[\laps{\mu} \varphi]
 &=& \fracm{2} \int H_{\mu}(P-I,P^T-I)\ \laps{\mu} \varphi\\
 &&+\int so(P\mathcal{C}\brac{\varphi, \Omega}\left [P^T \overline{\Omega^P}^T \chi_{D_\Lambda}\right ])  \\
&& + \int so(\overline{\Omega^P} \chi_{D_\Lambda}P H_{\mu}(\varphi,P^T-I))\\
&& - \int so(\mathcal{C}\brac{P,\Omega}[\laps{\mu} \varphi] P^T)\\
&& + \int \Omega^P[(1-\chi_{D_\Lambda})\laps{\mu} \varphi].
\end{ma}
\]
\end{proof}

Then, using the commutator estimates in \cite{CRW76}, \eqref{eq:c2:CfRlaphvarphiest}, \eqref{eq:c2:H1varphigest}, and \eqref{eq:c2:laphH1inhardy}, we have shown the following Lemma, which implies Theorem~\ref{th:intro:energy}
\begin{lemma}\label{la:betterOmegaP}
Let $P$ be a minimizer of $E(\cdot)$ as in \eqref{eq:energy:EQ}, Lemma~\ref{la:energy:Pmineuler}. Assume moreover, that $\Omega$ satisfies \eqref{eq:OmegaisRiesztrafo}. Then for any $\varphi \in C_0^\infty(B_r)$
\[
- \int \Omega^P[\laps{\mu} \varphi] \aleq \Lambda^{-\frac{m}{2}-\mu}\ r^{\frac{m}{2}-\mu}\ \vrac{A}_{2}\ [\varphi]_{BMO} + \vrac{A}_{2}^{2}\ \begin{cases}
                                                           [\varphi]_{BMO} \quad &\mbox{if $\mu \in (0,1]$},\\ 
                                                           \vrac{\laps{\mu} \varphi}_{(2,\infty)} \quad &\mbox{if $\mu > 1$}. 
                                                          \end{cases} 
\]
\end{lemma}
\begin{proof}
By Lemma~\ref{la:energy:minexists} and Lemma~\ref{la:energy:Pmineuler},
\[
 \vrac{\Omega^P}_{2\to1}+\vrac{\overline{\Omega^P}}_{2} + \vrac{\laps{\mu} P}_2 \aleq \vrac{\Omega[]}_{2->1} \aleq \vrac{A}_{2},
\]
and by Lemma~\ref{la:energy:Pmineuler} we need to estimate
\begin{align}
 &\int H_{\mu}(P-I,P^T-I)\ \laps{\mu} \varphi\label{eq:bo:1} \\
 &|\int so(P\mathcal{C}\brac{\varphi, \Omega}\left [P^T \overline{\Omega^P}^T \chi_{D_\Lambda}\right ])|  
 \aleq \vrac{A}_2\ \vrac{\mathcal{C}\brac{\varphi, \Rz}\left [P^T \overline{\Omega^P}^T \chi_{D_\Lambda}\right ])}_2
 \label{eq:bo:2}, \\
&|\int so(\overline{\Omega^P} \chi_{D_\Lambda}P H_{\mu}(\varphi,P^T-I))| \aleq \vrac{\overline{\Omega^P}}_2\ \vrac{H_{\mu}(\varphi,P^T-I)}_2\ \label{eq:bo:3}|, \\
&|\int so(\mathcal{C}\brac{P,\Omega}[\laps{\mu} \varphi] P^T)| \aleq \vrac{A}_2\ \vrac{\mathcal{C}\brac{P,\Rz}[\laps{\mu} \varphi]}_2\ \label{eq:bo:4}, \\
&|\int \Omega^P[(1-\chi_{D_\Lambda})\laps{\mu} \varphi]| \aleq \vrac{\Omega^P}_{2\to1}\ \vrac{(1-\chi_{D_\Lambda})\laps{\mu} \varphi]}_2.\label{eq:bo:5}
\end{align}
The estimate of \eqref{eq:bo:1} is immediate from \eqref{eq:c2:laphH1inhardy}, for the estimate of \eqref{eq:bo:2} we apply \cite{CRW76}. For the estimate of \eqref{eq:bo:3} we use \eqref{eq:c2:H1varphigest}, for \eqref{eq:bo:4} we have \eqref{eq:c2:CfRlaphvarphiest}.

It remains to estimate \eqref{eq:bo:5}, which follows from
\[
\begin{ma}
 \vrac{(1-\chi_{D_\Lambda})\laps{\mu} \varphi]}_2 
 &\aleq& \sum_{k=1}^\infty  \vrac{\laps{\mu} \varphi}_{2,A^k_{\Lambda r}} \overset{\sref{L}{la:QuasiLocality}}{\aleq}\sum_{k=1}^\infty  (2^k \Lambda r)^{-\frac{m}{2}-\mu} \vrac{\varphi}_{1}\\
 &\overset{\sref{P}{pr:L1BMOcompactsupport}}{\aleq}& \sum_{k=1}^\infty  (2^k \Lambda r)^{-\frac{m}{2}-\mu} r^m\ [\varphi]_{BMO}.
 \end{ma}
\]
\end{proof}


\begin{appendix}
\newpage
\section{Some facts on our fractional operators}\label{s:fracfacts}
The fractional laplacian $\lap^{\frac{s}{2}}$ is usually defined via its Fourier-symbol $-\abs{\xi}^s$. Here, we will mostly use the negative fractional laplacian $(-\lap)^{\frac{s}{2}} \equiv \laps{s}$ (which here plays the role of the gradient, or the divergence and rotation in the classical settings), defined via its symbol $\abs{\xi}^s$. These operators are defined for $s \in (-m,m)$, if $s < 0$, we write $\lap^{\frac{s}{2}} \equiv \lapms{\abs{s}}$.\\
Most of the time, we will use the potential definition: For Schwartz functions $f$,
\[
 \laps{s} f(x) =  c_s\ \lim_{\varepsilon \to 0} \int_{\abs{x-y}>\varepsilon} \frac{f(x)-f(y)}{\abs{x-y}^{m+s}}\ dy \quad \mbox{for $s \in (0,2)$}.
\]
The inverse is the Riesz potential,
\[
 \lapms{s} f(x) = \tilde{c}_s\ \lim_{\varepsilon \to 0} \int_{\abs{x-y}>\varepsilon} \frac{f(y)}{\abs{x-y}^{m-s}}\ dy \quad \mbox{for $s \in (0,m)$}.
\]
We refer, e.g., to \cite{SKM93, Landkof72} on hyper-singular operators, generalizations, and different representation formulas.

Next, we state some  useful facts about the fractional laplacian, which we are going to use throughout our paper, as standard repertoire.

We have the standard Poincar\'e inequality, for a proof, we refer, e.g., to \cite{NHarmS10Arxiv}.
\begin{lemma}[Poincar\'e inequality with compact support]\label{la:standardpoinc}
Let $s \in [0,m)$, $p \in (1,\infty)$, $q \in [1,\infty]$, then for any $B_r \subset \R^m$, and any $f \in C_0^\infty(B_r)$
\[
 \vrac{f}_{(p_1,q_1)} \leq C_s\ r^{s} \vrac{\laps{s} f}_{(p_1,q_1)}.
\]
\end{lemma}

The (scaling invatiant) Sobolev inequality takes the form
\begin{lemma}[Sobolev inequality]\label{la:standardsob}
Let $s \in [0,m)$, $p_1,p_2 \in [1,\infty)$, $q \in [1,\infty]$, for any $f \in \Sw(\R^m)$,
\[
 \vrac{f}_{(p_1,q)} \leq \vrac{\laps{s} f}_{(p_2,q)},
\]
where
\[
 \fracm{p_2} = \fracm{p_1} + \frac{s}{m}.
\]
\end{lemma}

For $p_1 = \infty$, we have the following limiting version of Sobolev's inequality:
\begin{lemma}[Limiting Sobolev inequality]\label{la:limitsobpoinc}
Let $s \in (0,m)$. For any $f \in \Sw(\R^m)$,
\[
 \vrac{f}_{\infty} \leq \vrac{\laps{s} f}_{(\frac{m}{s},1)},
\]
\end{lemma}

Also, we have the following H\"older-like inequality
\begin{lemma}[H\"older inequality]\label{la:hoelder}
Let $s \in [0,m)$, then for any $p_1 < p_2$, for any $B_r \subset \R^m$, and any $f \in C_0^\infty(B_r)$
\[
 \vrac{\laps{s}f}_{(p_1,q_1)} \leq C_{s,p_1,p_2}\ r^{\frac{m}{p_1}-\frac{m}{p_2}}\ \vrac{\laps{s} f}_{(p_2,\infty)}
\]
\end{lemma}
\begin{proof}
Let $\Lambda > 2$, then
\[
  \vrac{\laps{s}f}_{(p_1,q_1),B_{\Lambda r}} \aleq C_{s,p_1,p_2,\Lambda}\ r^{\frac{m}{p_1}-\frac{m}{p_2}}\ \vrac{\laps{s} f}_{(p_2,\infty)}.
\]
On the other hand, for some $\theta > 0$, by Lemma \ref{la:QuasiLocality}, Lemma \ref{la:standardpoinc},
\[
  \vrac{\laps{s}f}_{(p_1,q_1),\R^n \backslash B_{\Lambda r}} \aleq \Lambda^{-\theta} r^{-s} \vrac{f}_{(p_1,q_1)} \aleq \Lambda^{-\theta} \vrac{\laps{s} f}_{(p_1,q_1)}.
\]
For sufficiently large $\Lambda$ we can absorb the latter term into the left-hand side, and obtain the claim.
\end{proof}

From the Lemmata before, we also have
\begin{lemma}[Poincar\'e-Sobolev inequality with compact support]\label{la:sobpoinc}
Let $s \in (0,m)$, $p_1,q_1 \in (1,\infty)$, then we have $s \leq t$, for any $B_r \subset \R^m$, and any $f \in C_0^\infty(B_r)$
\[
 \vrac{\laps{s}f}_{(p_1,q_1)} \leq C_{p_1,q_1,p_2,q_2, s,t}\ r^{\frac{m}{p_1}-\frac{m}{p_2}+s-t}\ \vrac{\laps{t} f}_{(p_2,q_2)}, 
\]
where $p_2 \in (1,\infty)$ such that
\[
 \fracm{p_2} \leq \fracm{p_1} + \frac{s-t}{m}
\]
and $q_2 = \infty$ if the above inequality is strict, else $q_1 = q_2$.
\end{lemma}

A very important ingredient in our arguments is the boundedness of the Riesz potential on Morrey spaces.
\begin{lemma}[\cite{Adams75}]\label{la:adams}
Let $s \in [0,m)$, $p_1,p_2 \in (1,\infty)$, $q \in [1,\infty]$, $\lambda \in (0,m]$, such that
\[
\fracm{p_1} = \fracm{p_2} - \frac{s}{\lambda}.
\]
Then for any $f \in \Sw(\R^m)$,
\[
\vrac{\lapms{s} f}_{(p_1, q)_\lambda} \aleq \vrac{\lapms{s} f}_{(p_2, q)_\lambda}.
\]
\end{lemma}

The following is an easy equivalence result, recall \eqref{eq:def:Alambdark}.
\begin{proposition} \label{pr:summationdoesntmatter}
Let $\Lambda > 2$, $\sigma > 0$. Then,
\[
 \sum_{k=K_0}^\infty 2^{-k\sigma}\ \vrac{f}_{(p,q),A^k_r} \leq C_\sigma \vrac{f}_{(p,q),B_{\Lambda r}} + \sum_{k=0}^\infty 2^{-k\sigma}\ \vrac{f}_{(p,q),A^k_{\Lambda r}}
\]
\end{proposition}
\begin{proof}
 Let $k_0 := \lfloor \log_2\Lambda\rfloor \geq 1$, then
 \[
  2^{k_0} \leq \Lambda < 2^{k_0+1}
 \]
 We have
 \[
  2^{-\sigma (l+k_0)} \vrac{f}_{(p,q),A^k_{2^{l+k_0} r}} \aleq 2^{-\sigma k_0} 2^{-\sigma (l-1)} \vrac{f}_{(p,q),A^k_{2^{l-1} \Lambda r}} + 2^{-\sigma k_0}\ 2^{\sigma l} \vrac{f}_{(p,q),A^k_{2^l \Lambda r}}.
 \]
\end{proof}

\newpage
\section{Quasi-locality}
In this section we gather some facts which quantify the quasi-local behaviour of operators like fractional laplacians $\laps{\alpha}$, Riesz transforms $\Rz$, and Riesz potentials $\lapms{s}$. With ``quasi-local'' we mean the following: Let $A \subset \R^m$ be some domain and assume that $\supp f \subset A$. If we take $T$ to be any of the above mentioned operators, then there is no reason why $\supp Tf \subset A$, nor $\supp Tf \subset B_{\delta} A$ for some $\delta > 0$. Nevertheless, if we take a domain $B \subset \R^m$, $\dist (A,B) > \epsilon > 0$, then $Tf \in C^\infty(B)$. In fact, in this case
\[
 Tf(x) = k \ast f(x) \quad \mbox{for $x \in B$},
\]
where $k$ is a kernel of the form $k(y) = h(y/\abs{y})\ \abs{y}^{-n-s}$ for some $s \in (-m,m)$, $h$ some smooth function on $\S^{m-1}$. Since $\supp f \subset A$ and $x \in B$, we can replace
\[
 Tf(x) = \tilde{k} \ast f(x),
\]
where $\tilde{k}(y) = (1-\eta(y))k(y)$, and $\eta \in C_0^\infty(B_{\varepsilon}(0))$, $\eta \equiv 1$ on $B_{\epsilon/2}(0)$. Obviously, $\tilde{k} \in C^\infty(\R^m)$, and consequently so is $Tf$. In fact, by the usual Young-inequality, we have
\[
 \vrac{Tf}_{\infty,B} \leq \vrac{\tilde{k}}_\infty\ \vrac{f}_1 \leq \vrac{k}_{\infty, \R^m \backslash B_{\varepsilon/2}(0)}\ \vrac{f}_1 \leq C_{\vrac{h}_\infty}\ \varepsilon^{-n-s}\ \vrac{f}_1.
\]
That is, although we cannot ensure that $Tf \equiv 0$ in $B$ (as it would be, e.g., the case for local operators like $\nabla$), we can at least quantify that the farer away $B$ is from $A$, the less is the norm of $Tf$ on $B$. In particular, we have

\begin{lemma}[Quasi-locality (I)]\label{la:QuasiLocality}
 Let $p_1,p_2,q_1,q_2 \in [1, \infty]$, $s \in (-m,m)$
 and $\Omega_1, \Omega_2 \subset \R^m$ be disjoint domains with
 $d:= \dist(\Omega_1, \Omega_2) >0$ and with positive and finite Lebesgue measure. 
 Then, for any $f\in\mathcal S(\R)$, 
 \begin{equation*}
  \|\lap^{\frac{s}{2}} (f \chi_{\Omega_2})\|_{(p_1,q_1),\Omega_1} 
  \leq  d^{-m-s} |\Omega_1|^{1/p_1} |\Omega_2|^{1-1/p_2} \|f\|_{(p_2,q_2), \Omega_2},
 \end{equation*}
 where we set
 \[
  \lap^{\frac{s}{2}} := \begin{cases}
                         \laps{s} \quad &\mbox{if $s > 0$,}\\
			  Id\ \mbox{or}\ \Rz \quad &\mbox{if $s = 0$,}\\
			  \lapms{\abs{s}} \quad &\mbox{if $s < 0$.}\\
                        \end{cases}
 \]
\end{lemma}
Often we will use the above also for $\Omega_1$ or $\Omega_2$ to be a complement of some ball $B_{r}$. This is valid, since $\R^m \backslash B_r = \bigcup_{k=1}^\infty A^k_r$, recall \eqref{eq:def:Alambdark}. Then
\[
 \chi_{\R^m \backslash B_r} = \sum_{k=1}^\infty \chi_{A^k_r},
\]
and for each $A^k_r$ we have the correct estimate, so that for $s \in (-m,m)$ the sum on $k$ is convergent. Consequently, as a special case, using also Poincar\'e inequality (cf. Section~\ref{s:fracfacts}), we have

\begin{corollary}[Quasi-locality (II)] \label{co:QuasiLocalityII}
Let $p_1,p_2 \in (1, \infty)$, $q_1,q_2 \in [1,\infty]$, $s,t \in [0,m)$. Then, for any $B_r \subset \R^m$, $f\in\mathcal S(\R)$, $\Lambda > 1$, whenever 
\begin{equation*}
  \|\laps{s} (f \chi_{B_r})\|_{(p_1,q_1),\R^m \backslash B_{\Lambda r}} 
  \leq C_{s,p_1,p_2,q_1} \Lambda^{-m-s+\frac{m}{p_1}}\ r^{\frac{m}{p_1}-\frac{m}{p_2}-s+t} \|\laps{t} (\chi_{B_r} f)\|_{(p_2,q_2), B_r}.
 \end{equation*}
\end{corollary}

 \begin{lemma}[Quasilocality (III)]\label{la:quasilocIII}
 Let $f, g \in \Sw(\R^m)$, $\Omega_1, \Omega_2 \subset \R^m$ be disjoint domains with $d:= \dist(\Omega_1, \Omega_2) >0$ and with positive and finite Lebesgue measure. 
\[
 \vrac{\laps{s} ((\lap^{\frac{t}{2}} f\chi_{\Omega_1}) g\chi_{\Omega_2})}_{(p_1,q_1)}\\
 \aleq \sup_{\alpha \in [0,s]} d^{-m-t-\alpha}\ \vrac{ f\chi_{\Omega_1}}_{1}\ \vrac{\sabs{\laps{s-\alpha}(g\chi_{\Omega_2})} }_{(p_1,q_1)}
\]
for any $t \in (-m,m)$, $s \in (0,m)$.
\end{lemma}
\begin{proof}
By the disjoint support, pointwise almost everywhere,
\[
 (\lap^{\frac{t}{2}} f\chi_{\Omega_1})\ g\chi_{\Omega_2}(x) = c\int \abs{x-y}^{-m-t} f(y)\chi_{\Omega_1}(y)\ g(x) \chi_{\Omega_2}(x)\ dy.
\]
Let $\eta \equiv \eta \in C^\infty([0,\infty))$, $\eta \equiv 1$ on $[1,\infty)$, $\eta \equiv 0$ on $[0,1]$, and set
\[
 K_{t,d} (z) := \abs{z}^{-m-t}\ \eta(z/d) \in L^{p}(\R^m)\cap L^\infty(\R^m) \cap C^\infty(\R^m) \quad \mbox{for any $p > \max \{\frac{m}{m+t},1\}$}.
\]
It is worth noting the scaling behaviour of $K_{t,d}$ in $d$,
\begin{equation}\label{eq:props:scalingK}
 K_{t,d}(z) = d^{-m-t} K_{t,1}(z/d).
\end{equation}
We have
\[
 (\lap^{\frac{t}{2}} f\chi_{\Omega_1})\ g\chi_{\Omega_2} = (K_{t,d} \ast f\chi_{\Omega_1})\ g\chi_{\Omega_2}.
\]
Now we know that
\[
 \laps{s} K_{t,d} \in C^\infty(\R^m),
\]
and on the other hand
\[
 \abs{\laps{s} K_{t,d}(x)} \aleq C_d\ \abs{x}^{-m-s-t},
\]
so together
\[
  \laps{s} K_{t,d} \in L^p(\R^m)\cap L^\infty(\R^m) \cap C^\infty(\R^m), \quad \mbox{for any $p > \max \{\frac{m}{m+t+s},1\}$}.
\]
By the scaling \eqref{eq:props:scalingK}, we also know the exact dependence on $d$:
\[
 \vrac{\laps{\alpha} K_{t,d}}_{(p,q)} = d^{-m-t-\alpha+\frac{m}{p}} \vrac{\laps{\alpha} K_{t,1}}_{(p,q)}
\]

Consequently,
\[
  \vrac{\laps{s} (K_{t,d}\ast f\chi_{\Omega_1})}_{\infty} \aleq  \vrac{f\chi_{\Omega_1}}_{1}.
\]
Then our product rule, Lemma \ref{la:comm:Hsest}, tells us that we essentially have to estimate
\[
\begin{ma}
 &&\vrac{\laps{s} ((\laps{t} f\chi_{\Omega_1}) g\chi_{\Omega_2})}_{(p_1,q_1)}\\
 &\aleq& \sup_{\alpha \in [0,s]} \vrac{ \abs{\laps{\alpha}K_{t,d}\ast  f\chi_{\Omega_1}}\ \sabs{\laps{s-\alpha}(g\chi_{\Omega_2})} }_{(p_1,q_1)}\\
 &\aleq& \sup_{\alpha \in [0,s]} \vrac{ \abs{\laps{\alpha}K_{t,d}\ast  f\chi_{\Omega_1}}}_{\infty}\ \vrac{\sabs{\laps{s-\alpha}(g\chi_{\Omega_2})} }_{(p_1,q_1)}\\
 &\aleq& \sup_{\alpha \in [0,s]} \vrac{\laps{\alpha}K_{t,d}}_{\infty} \vrac{ f\chi_{\Omega_1}}_{1}\ \vrac{\sabs{\laps{s-\alpha}(g\chi_{\Omega_2})} }_{(p_1,q_1)}\\
 &\aleq& \sup_{\alpha \in [0,s]} d^{-m-t-\alpha}\ \vrac{ f\chi_{\Omega_1}}_{1}\ \vrac{\sabs{\laps{s-\alpha}(g\chi_{\Omega_2})} }_{(p_1,q_1)}.
 \end{ma}
\]

\end{proof}

\newpage
\section{Left-hand side estimates}

\begin{lemma}[Left-hand side estimates]\label{la:lhsest}
For a uniform constant $C$, and any $\kappa \in [\mu,2\mu)$, $\mu \leq \frac{m}{2}$, $\Lambda \geq 4$,
\[
\begin{ma}
\vrac{v}_{(\frac{m}{m-\kappa},\infty),B_{\Lambda^{-1}r}} 
&\leq& C\ \sup_{\varphi \in C_0^{\infty}(B_r,\R^N)}  \fracm{\vrac{\laps{\tau} \varphi}_{(\frac{m}{\tau+\kappa-\mu},1)}}\ \int v\cdot \laps{\mu}\varphi\\
&&+ C\ \Lambda^{\kappa-m}\ \vrac{v}_{(\frac{m}{m-\kappa},\infty),B_r} \\
&&+ C\  \Lambda^{\kappa-m}\ \sum_{k=0}^\infty 2^{k(\kappa-m)}\ \vrac{v}_{(\frac{m}{m-\kappa},\infty),B_{A^k_r}}.
\end{ma}
\]
More generally, for $\alpha \in (0,\mu]$,
\[
\begin{ma}
\vrac{\laps{\mu-\alpha}v}_{(\frac{m}{m+\mu-\alpha-\kappa},\infty),B_{\Lambda^{-1}r}} 
&\leq& C\ \sup_{\varphi \in C_0^{\infty}(B_r,\R^N)}  \fracm{\vrac{\laps{\alpha} \varphi}_{(\frac{m}{\alpha+\kappa-\mu},1)}}\ \int v\cdot \laps{\mu}\varphi\\
&&+ C\ \Lambda^{\kappa-m+\alpha-\mu}\ \vrac{v}_{(\frac{m}{m-\kappa},\infty),B_r} \\
&&+ C\  \Lambda^{\kappa-m+\alpha-\mu}\ \sum_{k=0}^\infty 2^{k(\kappa-m+\alpha-\mu)}\ \vrac{v}_{(\frac{m}{m-\kappa},\infty),A^k_r}.
\end{ma}
\]
\end{lemma}
\begin{proof}
Let $f \in C_0^\infty(B_{\Lambda^{-1}r},\R^N)$, $\vrac{f}_{(\frac{m}{\alpha+\kappa-\mu},1)} \leq 1$ such that 
\[
 \vrac{\laps{\mu-\alpha} v}_{(\frac{m}{m+\mu-\alpha-\kappa},\infty),B_{\Lambda^{-1}r}} \leq 2 \int \laps{\mu-\alpha}v \cdot f.
\]
Decompose for the usual cutoff $\eta_{r/2} \in C_0^\infty(B_{\frac{r}{2}})$,
\[
 f = \laps{\alpha} (\eta_{\frac{r}{2}} \lapms{\alpha} f) + \laps{\alpha} ((1-\eta_{\frac{r}{2}}) \lapms{\alpha} f) =: \laps{\alpha} g_1 + \laps{\alpha} g_2.
\]
As usual, using Lemma \ref{la:lapsbetaHalpha} (for $\beta = 0$) as an approximate product rule, for finitely many $s_k \in [0,\alpha]$, 
\[
 \vrac{\laps{\alpha} g_1}_{(\frac{m}{\alpha+\kappa-\mu},1)} 
\aleq \sum_{k} \vrac{\lapms{s_k} \abs{\laps{\alpha}\eta_{\frac{r}{2}}}}_{(\frac{m}{\alpha-s_k},\infty)}\ \vrac{\lapms{\alpha-s_k}\abs{f}}_{(\frac{m}{s_k+\kappa-\mu},1)} \aleq \vrac{f}_{p}.
\]
As for $g_2$, for a usual decomposition unity $\eta_l \in C_0^\infty (B_{2^{l} r} \backslash B_{2^{l-2}r})$
\[
 \laps{\alpha} g_2 = \sum_{l=-2}^\infty \laps{\alpha} (\eta_{l} \lapms{\alpha} f) =: \sum_{l=-2}^\infty \laps{\alpha}\tilde{g}_l
\]
and with the help of Lemma \ref{la:quasilocIII},
\[
 \vrac{\laps{\mu}\tilde{g}_l}_{(\frac{m}{\kappa},1)}  \aleq 
(2^l \Lambda)^{\kappa-m-\mu+\alpha}\ \vrac{f}_{\frac{m}{\alpha+\kappa-\mu}} \leq (2^l \Lambda)^{\kappa-m+\alpha-\mu} \vrac{f}_{\frac{m}{\alpha+\kappa-\mu}},
\]
and for $k \geq 1$,
\[
 \vrac{\laps{\mu}\tilde{g}_l}_{(\frac{m}{\kappa},1),B_{2^{k}r} \backslash B_{2^{k-1}r}}  \aleq \Lambda^{\kappa-m+\alpha-\mu}\ 2^{k\kappa+\max\{k,l\}(-m-\mu)+l\alpha} \ \vrac{f}_{\frac{m}{\alpha+\kappa-\mu}}.
\]
Consequently, for any $k \in \N_0$,
\[
 \vrac{\laps{\mu}g_2}_{(\frac{m}{\kappa},1),A^k_r}  \aleq (2^k \Lambda)^{\kappa-m+\alpha-\mu}\  \vrac{f}_{\frac{m}{\alpha+\kappa-\mu}}.
\]
So we conclude using
\[
\int v \cdot \laps{\mu}g_2 \aleq
\vrac{v}_{(\frac{m}{m-\kappa},\infty),B_r}\ \vrac{\laps{\mu}g_2}_{(\frac{m}{\kappa},1),B_{r}}  
+ \sum_{k=1}^\infty \vrac{v}_{(\frac{m}{m-\kappa},\infty),A^k_r}\ \vrac{\laps{\mu}g_2}_{(\frac{m}{\kappa},1),A^k_r}\\
\]
\end{proof}

\section{Iteration}
\begin{lemma}\label{la:iteration}
Let $q \in (0,1)$, $K > 1$, $\varepsilon > 0$ and assume
\begin{equation}\label{eq:it:smallness}
 \varepsilon+q^{2K}+ K\varepsilon \fracm{1-q} \leq \fracm{16}.
\end{equation}
Let moreover $\psi, \Phi: (0,\infty) \to (0,\infty)$, $\psi \leq \Phi$, and $\Phi$ monotone rising. Assume that for all $\lambda \in (0,1]$
\begin{equation}\label{eq:it:start}
 \Phi\brac{2^{-K}\lambda} \leq \varepsilon\ \Phi(\lambda) + \varepsilon \sum_{k=0}^\infty q^k\ \psi\brac{2^k\lambda}.
\end{equation}
If there is $G < \infty$ so that for all $\lambda \in (0,1)$
\[
 \sum_{k=0}^\infty q^k\ \psi(2^k\lambda) \leq G.
\]
Then, for all $\lambda \in (0,1)$,
\[
 \Phi(\lambda) \leq 32\ \lambda^{\fracm{2K}} \brac{\Phi(1) + G}.
\]
\end{lemma}
For a proof, we refer, e.g., to the arxiv-version of \cite{Sfracenergy}, or \cite{BPSknot12}.

\end{appendix}

\bibliographystyle{alpha}%
\bibliography{bib}%
\end{document}